%% file: main.tex
\documentclass[11pt]{article}
\usepackage{fullpage}
\usepackage[psamsfonts]{eucal}
\usepackage{amsmath}
\usepackage{amsfonts}
\usepackage{amssymb}
\usepackage{amstext}
\usepackage{amscd}
\usepackage{amsthm}
\usepackage{makeidx}
\usepackage{graphicx}
\usepackage[colorlinks=true,linkcolor=blue]{hyperref}
\usepackage{supertabular}
\usepackage{enumerate}
\usepackage{titling}
\usepackage{wasysym}
\usepackage{comment}
\usepackage{scalerel}
\usepackage{cleveref}
\usepackage{listings}
\usepackage{relsize}
\usepackage{microtype}
\usepackage{abbreviations}
\usepackage{xcolor}
\renewcommand{\aa}{\mathfrak{a}}
\newcommand{\uu}{\mathfrak{u}}

\newcommand{\WW}{\mathcal{W}}
\newcommand{\EE}{\mathcal{E}}
\newcommand{\VV}{\mathcal{V}}
\newcommand{\YY}{\mathcal{Y}}
\newcommand{\GG}{\mathcal{G}}
\newcommand{\XX}{\mathcal{X}}
\newcommand{\UU}{\mathcal{U}}

\newcommand{\ww}{\mathbf{w}}

\newcommand{\vv}{\mathbf{v}}
\newcommand{\vone}{\mathbf{1}}

\newcommand{\cpois}{c_{\scaleto{\ref{lem:sharp_poisson_tail}}{3pt}}}
\newcommand{\cweib}{c_{\scaleto{\ref{lem:weibullbound}}{3pt}}}
\newcommand{\omegam}{\Omega_{\scaleto{\ref{dfn:omegadef}}{3pt}}}

\begin{document}
\title{The Spectral Edge of Constant Degree Erd\H os-R\'enyi Graphs}
\author{Ella Hiesmayr\thanks{\texttt{ella.hiesmayr@berkeley.edu}. University of California, Berkeley. Supported by the Citadel Securities Berkeley Statistics PhD Fellowship.} \and Theo McKenzie\thanks{\texttt{theom@stanford.edu}. Stanford University. Supported by NSF GRFP Grant DGE-1752814 and NSF Grant DMS-2212881. }}
\date{}
\maketitle

\begin{abstract} We show that for an Erd\H os-R\'enyi graph on $N$ vertices with expected degree $d$ satisfying $\log^{-1/9}N\leq d\leq \log^{1/40}N$, the largest eigenvalues can be precisely determined by small neighborhoods around vertices of close to maximal degree. Moreover, under the added condition that $d\geq\log^{-1/15}N$, the corresponding eigenvectors are localized, in that the mass of the eigenvector decays exponentially away from the high degree vertex. This dependence on local neighborhoods implies that the edge eigenvalues converge to a Poisson point process.
These theorems extend a result of Alt, Ducatez, and Knowles, who showed the same behavior for $d$ satisfying $(\log\log N)^4\ll d\leq (1-o_{N}(1))\frac{1}{\log 4-1}\log N$. 

To achieve high accuracy in the constant degree regime, instead of attempting to guess an approximate eigenvector of a local neighborhood, we analyze the true eigenvector of a local neighborhood, and show it must be localized and depend on local geometry. 
 
     \end{abstract}
\input{1.Introduction.tex}

\input{2.Preliminary.tex}

\input{3.Regimes.tex}

\input{4.Fine.tex}

\input{5.Rough.tex}

\input{6.Structure.tex}

\input{7.Anticoncentration.tex}

\input{8.Eigenvector}

\bibliographystyle{alpha}
\bibliography{ref}

\appendix
\include{9.appendix}

\end{document}

%% file: 1.Introduction.tex
\section{Introduction}\label{sec:intro}

Finding the spectrum and eigenvectors of the adjacency matrix of graphs is an ubiquitous problem in combinatorics and spectral theory, with important applications to computer science and mathematical physics, see \cite{chung1997spectral,kottos1997quantum, alon1998spectral} for general overviews on their applicability. More specifically, researchers have studied the ``typical'' behavior of eigenvalues and eigenvectors through random models. The most well-studied of these is the Erd\H os-R\'enyi model $\GG(N,\frac dN)$, where each of the possible $\binom N2$ edges is included independently with probability $\frac dN$. Therefore the expected degree of a vertex is $\frac{N-1}{N}d\approx d$. See the monograph of Guionnet for an overview of known results and the state of the field for this model \cite{guionnet2021bernoulli}.

In this paper, we focus on the specific behavior of the eigenvalues near the edge of the spectrum. This spectral edge has received much attention, as it has its own specific applications. For example, the edge governs the mixing rate of Markov chains, and graph partitioning (as shown in \cite{hoory2006expander}), and also achieves localization in the Anderson model (\cite{anderson1958absence}, further discussed later).\footnote{Unless otherwise specified, theorems stated in the introduction are of phenomena that occur with high probability.}  The behavior of the extreme eigenvalues and eigenvectors in the Erd\H os-R\'enyi model is known to go through various phase transitions. When $d\gg N^{1/3}$ these edge eigenvalues have Tracy-Widom fluctuations,\footnote{Here and throughout, when writing $x\ll y$, we mean that $x=o(y)$, $x\lesssim y$ means $x=O(y)$, and $x\asymp y$ means $x=\Theta(y)$.} similar to the fluctuations of the eigenvalues of a GOE matrix \cite{erdHos2012spectral,erdHos2013spectral,lee2018local}. When $N^\epsilon \leq d\ll N^{1/3}$, for some fixed $\epsilon>0$, the top eigenvalues lose GOE behavior and edge eigenvalues become Gaussian distributed \cite{huang2020transition,he2021fluctuations}.

For sparser Erd\H os-R\'enyi graphs, Krivelevich and Sudakov showed using a graph decomposition that eigenvalues are governed by the statistics surrounding the highest degree vertices \cite{krivelevich2003largest}. Therefore, if $\mu_k$ is the expected number of vertices of degree $k$, we define
\begin{equation}\label{eq:udef}
\uu:=\arg\min_{k\in \Z} \left \{\max \left \{ \mu_k,\mu_k^{-1} \right \} \right \}.
\end{equation}
 $\uu$ is roughly the largest degree we expect to occur in the graph. 
Krivelevich and Sudakov showed the largest eigenvalue of an Erd\H os-R\'enyi graph is $(1+o_N(1))\max\{d,\sqrt{\uu}\}$. The balls and bins paradigm implies that for $\log^{-1}N\ll d\ll\log N$, $\uu=\Theta( \frac{\log N }{\log \log N})$. Therefore, this shows that there is a phase transition in the largest eigenvalue at $d=\sqrt{\frac{\log N}{\log\log N}}$.

In fact, we begin to see the local affect of high degree vertices at $d\asymp \log N$. This is well known to be the threshold for connectivity (as was shown in the original work of Erd\H os and R\'enyi \cite{erdHos1960evolution}), but is also the threshold for large fluctuations in the degree sequence, as opposed to greater concentration seen for larger $d$. Specifically, Benaych-Georges, Bordenave, and Knowles showed that when $d\gg \log N$, edge eigenvalues converge to the edge of the support of the asymptotic eigenvalue distribution, but when $d\ll \log N$, roughly, edge eigenvalues are ``governed'' by the largest degree vertices of the adjacency matrix \cite{benaych2019largest,benaych2020spectral}, with the specific threshold later given by Alt, Ducatez, and Knowles \cite{alt2021extremal}.

Alt, Ducatez, and Knowles further studied this problem, and managed to obtain impressively detailed results. Through the works of \cite{alt2021delocalization,alt2022completely,alt2023poisson, alt2023localized}, the authors show a transition between the occurrence of delocalized eigenvectors in the bulk of the spectrum, and localized eigenvectors near the edge for $d\lesssim\log N$ (the specific bounds on $d$ vary paper to paper, but all results are for this sparse regime). 

We focus specifically on \cite{alt2023poisson}. In this result, Alt, Ducatez, and Knowles show that the largest eigenvalues of the graph are determined by two combinatorial statistics around the high degree vertices. As was shown previously through \cite{krivelevich2003largest} and \cite{benaych2019largest}, the primary term is the degree of the high degree vertex, which we denote here by $\alpha_x$. To gain the necessary levels of accuracy, they also track the secondary term $\beta_x$, which is the number of vertices of distance exactly 2 from a high degree vertex $x$ in the graph.\footnote{We have translated their parameters into the unnormalized versions we use in our proof.} Note that this notation slightly differs from the one in \cite{alt2023poisson}, where $\alpha_x$ and $\beta_x$ are normalized by $d$. Reinterpreting their result, they show the following. 

\begin{thm}[\cite{alt2023poisson}]
    For $\zeta> 4$ and sufficiently small constant $\xi>0$, assume that $(\log\log N)^\zeta\leq d\leq (\frac{1}{\log 4-1}-(\log N)^{-\xi})\log N$. For $K:=d^{1/2-2/\zeta-16\xi}$ there are some $\delta, \epsilon>0$ such that the first $K$ eigenvalues are of the form 
    \begin{equation}\label{eq:adk}
   \frac{\alpha_x}{\sqrt{\alpha_x-\frac{\beta_x}{2\alpha_x}(\frac{\alpha_x}{d}+\frac{\beta_x}{\alpha_x d})+\frac{\beta_x}{2\alpha_x}\sqrt{(\frac{\alpha_x}{d}+\frac{\beta_x}{d\alpha_x})^2-4\frac{\alpha_x}{d}}}}+O(d^{-\epsilon})
    \end{equation}
    for $K$ vertices of degree at least $\uu-\frac{d^{-\delta}}{\log \uu}$.
\end{thm}

This is done by, given $\alpha_x$ and $\beta_x$, making an educated guess for the structure of the eigenvector. To use this approximate eigenvector it is crucial that the statistics of the local neighborhoods of high degree vertices are concentrated, in particular that the degrees of the vertices in the neighborhood is reasonably close to $d$, which is approximately their expected value. According to \eqref{eq:adk}, to be able to properly estimate eigenvalues using this strategy it is crucial that $d \gg 1$.

\subsection{Main results}

Finding similar behavior when $d$ is constant has been an open question asked by Guionnet \cite{guionnet2021bernoulli}. In this paper, we use new techniques to show that in fact, the same behavior holds in the constant $d$ regime, and continues until the average degree is subconstant. Namely, we prove the following.

\begin{thm}\label{thm:maineigenvalue}
Consider a $G\sim \GG(N,\frac dN)$ graph with $\log^{-1/9}N\leq d\leq \log^{1/40}N$. With high probability, for each of the $e^{\log ^{1/8}N}$ largest eigenvalues $\lambda$ of the adjacency matrix, there is some vertex $x$ such that 
\[\lambda=\sqrt{\alpha_x+\frac{\beta_x}{\alpha_x}+\frac{d^2+d}{\alpha_x}}+O((d^{3/2}+1)\uu^{-11/6}).\]
Moreover, for $k\leq e^{\log^{1/8}N}$, the vertex $x$ corresponding to the $k$th largest eigenvalue is the $k$th vertex in the lexicographic ordering $(\alpha_x, \beta_x)$.
\end{thm}

As we will show, these high degree vertices are spaced throughout the graph, and have almost independent local statistics. Therefore, similar to \cite{alt2023poisson}, the distribution of the highest eigenvalues is described by a Poisson point process with density given by the probability of existence of an $(\alpha,\beta)$ pair, where $\alpha$ is close to maximal among all vertices. 

To this end, define the discrete intensity measure $\rho:\R\rightarrow \R$,
\[
\rho(\frac{s}{\uu}):= N\sum_{\ell=0}^{2\log^{1/8}N}\left(\frac{e^{-d}d^{\uu-\ell}}{(\uu-\ell)!}\frac{e^{-d(\uu-\ell)}(d(\uu-\ell))^{(s-\uu+\ell)(\uu-\ell)}}{ ((s-\uu+\ell)(\uu-\ell))!} \vone_{\langle s(u-\ell)\rangle=0}\vone_{(s-u+\ell)(u-\ell)\leq d(\uu-\ell)+\uu^{7/8}}\right).
\]
where $\vone_{\langle x\rangle =0}$ is the indicator that $x$ is a whole number. $\rho$ induces a Poisson point process $\Psi$. 
The meaning of $\rho$ is that it is the intensity measure of $\alpha+\frac{\beta}{\alpha}$ if $\alpha\sim Pois(d)$ and $\beta\sim Pois(d\alpha)$, restricted to $\alpha\in[\uu-2\log^{1/8}N,\uu],\beta\in[0,d\alpha+\uu^{7/8}]$. As we will see, Theorem \ref{thm:maineigenvalue} implies $\Psi$ approximates the density of $\lambda_x^2$ at the edge of the spectrum. 

Formally, we we will consider proximity in L\'evy-Prokhorov distance, which is a metrization of the weak topology. Namely we define, for two Borel measures $\nu_1,\nu_2$ on $\R$,
\[
D(\nu_1,\nu_2)=\inf\{\epsilon>0|\forall A\in \mathfrak B,\nu_1(A)\leq \nu_2(A_\epsilon)+\epsilon\textnormal{ and }\nu_2(A)\leq \nu_1(A_\epsilon)+\epsilon\}
\]
where $\mathfrak B$ is the set of Borel measurable sets in $\R$ and $A_\epsilon$ is the neighborhood of radius $\epsilon$ around $A$. 

Moreover, for $K>0$, define $\kappa(K)$ as
\begin{equation}\label{eq:kappadef}
\kappa(K)=\inf\{s\in\R,\rho([s,\infty))\leq K\}.
\end{equation}
 
\begin{thm}\label{thm:process}
Set $K=e^{\log^{1/8}N}$, and recall the definition of $\kappa(K)$ from \eqref{eq:kappadef}. Consider the density function
\[
\Phi:=\sum_{\lambda\in spec(A)}\delta_{\uu(\lambda^2-\frac{d^2+d}{\uu})}.
\]
Then for $d$ as defined in Theorem \ref{thm:maineigenvalue},
\[\lim_{N\rightarrow\infty}D(\Phi\vone_{[\kappa(K),\infty)},\Psi\vone_{[\kappa(K),\infty)})\rightarrow0.\]

\end{thm}
We multiplied both point processes by $\uu$ to emphasize that the approximation of $\lambda^2$ in Theorem \ref{thm:maineigenvalue} is $o(1/\uu)$. 
As shown in Theorem \ref{thm:maineigenvalue}, the largest eigenvalues are approximately the square root of a rational function on local statistics, therefore it is simpler to consider the point process of $\lambda^2$, as it leads to a nicer expression. Much of our analysis will be on the squared eigenvalue $\lambda^2$.

Theorem \ref{thm:process} implies the fluctuations of the top eigenvalue. Similar to if there is increasing degree as in \cite{alt2023poisson}, fluctuations are determined at two scales. If the expected number of vertices of degree $\uu$ (that is $\mu_\uu$) is constant, then the maximum degree has nonzero variance and will dominate the fluctuation of the maximum eigenvalue. Therefore, in this case the top eigenvalue will fluctuate at a large scale, as its first order fluctuation is a shifted Bernoulli, based on whether there is a vertex of degree $\uu$ or not. 
If the expected number of vertices of degree $\uu$ is subconstant or superconstant, then the largest degree of the graph becomes deterministic, and the fluctuations are determined by $\beta$. As $\beta$ is then distributed as a Poisson, the fluctuations become much smaller and become those of the maximum of Poissons. 

The fact that these eigenvalues are almost completely determined by local neighborhoods is intrinsically linked to the fact that eigenvectors are Anderson localized (see \cite{anderson1958absence}). That is, they decay exponentially around a fixed vertex. We show the following, which implies eigenvectors are close to as localized as possible. 

\begin{thm}\label{thm:maineigenvector}
Define $B_r(x)$ to be the ball of radius $r$ around $x$ in $G$. For $c\leq 1/15$, consider $\log^{-c}N\leq d\leq \log^{1/40}N$, and $K_2=\log^{o_N(1)} N$. Moreover, we fix $r'\geq 1$, and if $c>0$, we add the requirement that $r'\leq 1/(3c)$. With high probability, the eigenvectors $\vv$ corresponding to the $K_2$ largest eigenvalues are \emph{exponentially localized} in the sense that for each $\vv$ there is some vertex $x$ such that
\[\vv|_{x}=\frac1{\sqrt 2}+O((1+d^{-1/2})\uu^{-1/3})\]
and for $1\leq i\leq r'$,
\[
\|\vv|_{S_i(x)}\| =\left(\frac{d}{\alpha}\right)^{(i-1)/2}\frac1{\sqrt 2}(1+O((1+d^{-1/2}+d^{-i+1})\uu^{-1/3}))
\]
and
\[
\|\vv|_{[N]\backslash B_i(x)}\|=\left(\frac{d}{\alpha}\right)^{i/2}\frac{1}{\sqrt 2}(1+O((1+d^{-1/2}+d^{-r'+1})\uu^{-1/3})).
\]
\end{thm}
Our desire to keep our result as general as possible has resulted in this long expression for our error. There are multiple error terms and for $d\asymp 1$, the ratio of $d,i,$ and $\uu$ governs which type of error will dominate. 

We consider a significantly smaller number of eigenvalues in this theorem than in Theorem \ref{thm:maineigenvalue} as in order to show eigenvector localization, we must quantify the gaps of the eigenvalues induced by Theorem \ref{thm:maineigenvalue} and Theorem \ref{thm:process}, whereas the previous theorems do not require such gaps. We do this only for the region given above, as our focus is the spectral edge, but most likely further analysis could extend this to further eigenvalues. 

\subsection{Idea of Proof}

We follow the general framework of the proof of \cite{alt2023poisson}. Our overall goal is to show that the largest eigenvalues are determined by the local geometry of the highest degree vertices. Once we show this, the Poisson point process and eigenvector structure follow from the randomness of the graph. We show this determination by classifying vertices by degree, and associating an eigenvector with each high degree vertex. The main issue with generalizing the classification in \cite{alt2023poisson} to constant degree is that in this regime, the fluctuations of statistics of the balls surrounding individual entries become too large to guess the eigenvector based on the inputs $\alpha_x,\beta_x$ only.

Such a formula already is quite technical, so we avoid directly coming up with a more involved equation that gives a more accurate guess. Instead of making a guess, \emph{we use the true top eigenvector of the ball of radius $r$ around the high degree vertex}. The advantage of such a method is that the only error in the eigenvector equation comes from truncating at level $r$. Therefore, if we can show the eigenvector is localized away from level $r$, then the error from this truncation can be drastically smaller. The disadvantage is that, considering we are not making a fixed guess, we initially have no information about the eigenvector or eigenvalue, even whether it is localized or not.

 The neighborhoods of the highest degree vertices are typically tree-like, and the central vertex has much higher degree than all others. Therefore, by analyzing the eigenvector equation at each vertex, we create a system of linear equations for the eigenvector and eigenvalue. Because the neighborhood is a tree, this simplifies into a recursive equation for the eigenvalue. Moreover, because the central vertex has much higher degree, we can show this equation can be truncated up to small error with a rational function of local statistics, giving near exact dependence. In summary, as opposed to guessing a specific approximate eigenvector, we show that with high probability whatever eigenvector \emph{does} occur can give the approximation in Theorem \ref{thm:maineigenvalue}. 

 When we fully write out the equation for the eigenvalue, the first two terms are $\alpha_x$ and $\frac{\beta_x}{\alpha_x}$, which are used in \cite{alt2023poisson} to completely determine the eigenvector. We show, explicitly giving the next few terms, that the equation for the eigenvalue beyond $\alpha,\beta$ concentrates. Specifically, although the fluctuation of statistics increase as degree decreases, the dependence on the fluctuating statistics decreases at a quicker rate. In fact, in our regime, the eigenvalue decays quickly enough that it implies the lexicographic ordering of Theorem \ref{thm:maineigenvalue}.

Given the dependence on statistics of the eigenvalue and eigenvectors of local neighborhoods, we translate this into statements about the overall graph. By standard perturbation theoretic arguments, this reduces to showing these local statistics are well separated for vertices corresponding to the edge of the spectrum. The requisite statistics are binomially distributed, which are approximately Poisson. Therefore it becomes useful to give precise tail bounds on the Poisson distribution. Tao recently gave such a tight, two-sided bound on his blog \cite{tao2022improved}, which is sufficient to show that $(\alpha_x,\beta_x)$ that are close to lexicographically maximizing are well separated.

Once we have proper control over these high degree vertices, we need to control the contribution of the rest of the spectrum. To do this, for vertices of still somewhat large degree, we proceed as per \cite{alt2023poisson} and take a localized test vector supported on a pruned graph, where edges are pruned in such a way that high degree vertices are separated and neighborhoods are tree-like. For all other vertices, we can use the decomposition of Krivelevich and Sudakov, who show that after deleting a subgraph of disjoint stars surrounding high degree vertices, the spectral radius is far from $\sqrt{\uu}$ \cite{krivelevich2003largest}. This makes it much simpler to bound the spectral radius of the operator away from the largest eigenvectors of the highest degree vertices, for which we also use that the second .

\subsection{Related Work}
There are many results concerning the bulk of the spectrum of random matrices. Focusing specifically on sparse Erd\H os-R\'enyi graphs, Khorunzhy, Shcherbina, and Vengerovsky, then Zakharevich, analyzed the moments of the limiting distribution of Wigner matrices of general models that include constant degree Erd\H os-R\'enyi graphs in order to study the limiting measure of the spectral distribution \cite{khorunzhy2004eigenvalue, zakharevich2006generalization}. Benaych-Georges, Guionnet, and Male give a central limit theorem for linear statistics of a model that includes constant degree Erd\H os-Renyi graphs \cite{benaych2014central}. Along with these bulk results, Bhaswar Bhattacharya and Ganguly, then Bhaswar Bhattacharya, Sohom Bhattacharya, and Ganguly, give a large deviation principle for the edge eigenvalues of very sparse Erd\H os-R\'enyi graphs \cite{bhattacharya2020upper, bhattacharya2021spectral}.

These random matrices have been studied as a model for quantum physics, specifically Hamiltonians of disordered systems. We see similar eigenvector localization in the edge of the spectrum in the Anderson model (see \cite{anderson1958absence}), where vertices on an integer lattice are given random potential, and we study the spectrum of the resulting Schr\"odinger operator. Eigenvectors near the edge of the spectrum are known to be localized for various models (e.g. \cite{goldsheid1977random, frohlich1983absence, aizenman1993localization, ding2020localization}) whereas there has been less progress on the structure of eigenvectors in the bulk. 

 L\'evy matrices, a model of Wigner matrices where entries are sampled from distributions with heavy tails, have also proved to be a useful model for studying eigenvector localization. Focusing specifically on results concerning the edge of the spectrum, there is a transition from delocalized eigenvectors in the bulk to localized eigenvectors at the edge \cite{bordenave2013localization,bordenave2017delocalization}. Moreover, similar to the sparse Erd\H os-R\'enyi model, eigenvalues near the edge of the spectrum in sparse L\'evy models are known to converge to a Poisson point process \cite{soshnikov2004poisson, auffinger2009poisson}.

\subsection{Negative eigenvalues}
The discussion above concerns the largest eigenvalues of the adjacency matrix, however, by the exact same analysis we can consider the most negative eigenvalues. Under the high probability assumption that the neighborhood of every high degree vertex is a tree, by the bipartite nature of a tree, every positive eigenvalue of a neighborhood of a tree has a corresponding negative eigenvalue that is of the same magnitude and has the same localization properties. Therefore, Theorems \ref{thm:maineigenvalue}, \ref{thm:process}, and \ref{thm:maineigenvector} all apply to the most negative eigenvalues as well. 

\subsection{Extension of results}
We believe that by increasing the analysis from our given set of local statistics to higher moments, our methods can be used to give even more accurate formulae for the largest eigenvalues based on the degree sequence of the highest degree vertices. Such an argument could show separation of the largest $\log^aN$ eigenvalues for any fixed $a\geq 0$, giving a more specific (and more complicated) Poisson point process and showing eigenvector localization for all these $\log^aN$ eigenvectors. However, for simplicity of the argument, in this work we only consider $K=\log^{o_N(1)}N$.

Using this same argument, we may be able to improve the necessary lower bound on $d$. Some estimates and concentration results required us to lower bound $d$ by $\log^{-c} N$ for some $c < 1$, and error bounds simplify given our concrete assumption on $d$, but there are also important structural consideration for smaller $d$. If $d=\log^{-c}N$ for $c>0$, then the connected components surrounding high degree vertices can be of small radius, which means that neighborhoods of high degree vertices could be identical, leading to the eigenvalues of those neighborhoods being identical, and thus the eigenvectors would no longer be localized around one high-degree vertex. This implies that we cannot remove the dependence of $d$ on $r$ in Theorem \ref{thm:maineigenvector}. 

On the other hand, including more terms in our expression for $\lambda_x$ could improve the lower bound on $d$ for Theorem \ref{thm:maineigenvalue} and Theorem \ref{thm:process}. However, in our opinion we will need new methods to achieve the threshold of $d=e^{-(\log\log N)^2}$ appearing in \cite{krivelevich2003largest} and \cite{bhattacharya2021spectral}. 
Such a threshold is natural as for $d\leq e^{-(\log\log N)^2}$, all connected components are of size $(1+o(1))\uu$, making localization and independence trivial.

Similarly, by using slightly tighter bounds on the probabilities of some tail events, we expect we can improve the upper bound on $d$ in Theorem \ref{thm:maineigenvalue}. However, there is a natural barrier at $\log^{1/2}N$, both in the exponential rate of decay of eigenvectors, and the application of the method in \cite{krivelevich2003largest}.  We are not motivated to optimize our technique for the upper bound considering results are already known for larger $d$ from \cite{alt2023poisson}.

\subsection{Overview of the paper}
In Section \ref{sec:prelim}, we overview the notation and give some preliminary theorems we will use throughout. In Section \ref{sec:regimes}, we define the different vertex degree regimes, where vertices are classified by degree. Our goal will be to be to give an explicit unitary transform $U$ such that $UAU^{*}$ has a block decomposition that is close to diagonal. Each diagonal block corresponding to high degree vertices will correspond to a space spanned by vectors localized around these vertices. 

This is similar to the decomposition from\cite{alt2023poisson}, and it also bears resemblance to the more directly combinatorial decompositions of other sparse matrix results \cite{krivelevich2003largest, bhattacharya2021spectral}. The majority of the rest of the proof is dedicated to analyzing each of these regimes. In Section \ref{sec:fine}, we analyze the eigenvector and eigenvalue of the largest eigenvalue of the ball of radius $r$ surrounding the highest degree vertices, and we show that it is localized and well approximated by a formula involving only $\alpha, \beta,$ and $d$. In Section \ref{sec:rough}, we use a test vector to show that vectors corresponding to vertices of high, but not too high degree, do not have large eigenvalues. In Section \ref{sec:finaldecomp}, we analyze the bulk using the decomposition from \cite{krivelevich2003largest}, and we show the given block decomposition is such that the highest degree vertices dominate. This gives Theorem \ref{thm:maineigenvalue}.  In Section \ref{sec:anticoncentration}, we use this formula to prove the Poisson process, and in Section \ref{sec:eigenvector} we show this implies Theorem \ref{thm:maineigenvector}.

\subsection*{Acknowledgements}
We thank Shirshendu Ganguly for introducing us to this question and for many helpful discussions.

%% file: 2.Preliminary.tex
\section{Preliminaries}\label{sec:prelim}

In this section we introduce the notation we will use, define the exact parameters we will use, as well as relevant results about Erd\H{o}s-R\'{e}nyi graphs for the regime we are interested in. We then introduce distributional results we will use, in particular about the Binomial and Poisson distributions, and comparisons between them. Finally we state some spectral properties of graphs that we will use repeatedly.

\subsection{Notation} 
We consider a graph $G=(V,E)$ sampled from the Erd\H{o}s-R\'enyi distribution $\GG(N,\frac dN)$ with adjacency matrix $A_G$. When $G$ is clear from the context, we write this as $A$. We consider our vertex set to be $[N]$. Each of the possible $\binom N2$ edges is included independently with probability $\frac dN$. In this paper, we work in the following regime.
\begin{dfn}[Choice of parameters]\label{dfn:rdfn}
\ecomm{}
We fix $c>0$ and we take the average degree $d(N)$ as any function such that for Theorem \ref{thm:maineigenvalue} and  Theorem \ref{thm:process}, $\log^{-1/9}N\leq  d\leq \log^{1/40} N$, and for Theorem \ref{thm:maineigenvector},  $\log^{-c}N\leq  d\leq \log^{1/40} N$ for $c\leq 1/15$.

Our analysis will be based on considering balls of radius $r$ around the highest degree vertices. Most results are true for sufficiently large $r$, but in fact, it is enough to take $r = 5$ in order to prove Theorems \ref{thm:maineigenvalue} and \ref{thm:process}. For Theorem \ref{thm:maineigenvector}, we need a slightly larger radius. So for the rest of this paper it is sufficient to take $r:=\max \left \{ 5, 2r' \right \}$,
where $r'$ is the parameter from Theorem \ref{thm:maineigenvector}.
\end{dfn}

For $i\geq 0$, we denote by $B_i(x)$ the ball of radius $i$ around $x$, rooted at $x$. Moreover, we define $S_i(x)$ to be the set of vertices $y$ such that the shortest path from $x$ to $y$ is of length $i$. In other words $S_i(x)$ are all vertices that are not in $B_{i-1}(x),$ and that are connected to $x$ by a path of length $i$. We also call $S_i(x)$ the sphere of radius $i$ around $x$.

Given a root vertex $x$, we define a partial ordering on vertices by writing $u \leq v$ if there is a shortest path from $x$ to $v$ that goes through $u$. We also write for $y\in [N]$,
\begin{equation}\label{eq:Ndef}
N_y:=\left|\{u\in [N]:u\geq y, u\sim y\}\right|
\end{equation}
as the number of children of $y$ in the rooted graph.
Similarly, for a rooted or unrooted graph, for a vertex $y\in [N]$ we define $\Gamma_y=\{z\in[N]:z\sim y\}$ to be the neighborhood of $y$.

The following parameters will be used to approximate the largest eigenvalues. 
\begin{dfn}\label{dfn:betadef}
    Our parameters are defined as follows.
    \begin{enumerate}
\item $\alpha_x:=|\Gamma_x|$, the degree of the vertex $x$ 
\item $\beta_x:=\sum_{y\sim x}N_y$, the number of vertices in $S_2(x)$,
\item $\beta^{(1,1)}_x=\sum_{y_2\in S_2(x)}N_{y_2}$, the number of vertices in $S_3(x)$, and
\item $\beta^{(2)}_x$=$\sum_{y\sim x}N_y^2$.
    \end{enumerate}
\end{dfn}
As we will see, the eigenvalue is mostly determined by these four statistics. In our regime, the last two statistics (as well as all others) are well concentrated enough that we can write an accurate enough formula for the eigenvalue based on only $\alpha,\beta$.  

The combinatorial aspects of an Erd\H os-R\'enyi graph are governed by binomial distributions. Therefore, by $\Bin(k;N,p)$ we denote the probability that a binomial random variable with $N$ trials and success probability $p$ is equal to $k$.

Recall the definition of $\uu$ from \eqref{eq:udef}. As $d$ is small, the maximum degree is almost deterministic. 

\begin{lemma}[\cite{bollobas01randomgraphs}, Theorem 3.7] \label{lem:max_degree}
Define $\mu_k=N\Bin(k;N-1,\frac dN)$ to be the expected number of vertices of degree $k$ in the graph. Define $\uu$ to be the positive integer value of $k$ that minimizes $\max\{\mu_k,\mu_k^{-1} \}$. Then if $d=o(\log N)$, with high probability the maximum degree is in $\{\uu-1,\uu\}$. 
\end{lemma}

In order to calculate $\uu$, note that by the Stirling approximation, having $\mu_k\approx 1$ implies
\[
(1+o_N(1))\log N-\uu\log \uu + \uu-d+\uu\log d-\frac12\log(2\pi \uu)=0.
\]

Therefore, in our regime of $d$,
\begin{equation}\label{eq:uapprox}
\uu =(1+o_N(1))\frac{\log N}{\log\log N-\log d}.
\end{equation}
and $\uu= \Theta \left (\frac{\log N}{\log\log N} \right )$.

Throughout, $\vone_X, \vone(X)$ denote the indicator on the event $X$ occurring. 

For a vector $v \in \mathbb{R}^2$ we denote by $\| v \|$ its Euclidean norm. For a matrix $A$ in $\mathbb{R}^{m \times n}$, we denote by $\| A \|$ the operator norm, i.e. $\| A \| = \sup_{v \in \mathbb{R}^n} \frac{ \| A v \|_2}{\| v\|_2}.$ 

\subsection{Distributional notation and comparisons}

 We will estimate our binomial distributions with  Poisson distributions, as the large $N$ limit of a binomial is a Poisson  when $p$ is small.
 The following approximations are standard, with the proofs given in the appendix in Section \ref{sec:estimates}.

\begin{lemma}\label{lem:binomtopois}
If $X\sim Binom(n,p)$ and $Y\sim Pois(np)$, and if $k,np\leq \sqrt n$, then
\[
\pr(X=k)= \left (1+O \left (\frac{k^2+(np)^2+1}{n} \right ) \right )\pr(Y=k).
\]
\end{lemma}
This implies that the tails are also the same up to a small error.

\begin{cor}\label{cor:bintopoistail}
If $X\sim Binom(n,p)$ and $Y\sim Pois(np)$, and if $k\leq \sqrt n$, $np\leq n^{1/2-c}$ for some fixed constant $c$, then
\[
\pr(X\geq k)= \left (1+O \left (\frac{k^2+(np)^2+1}{n} \right ) \right ) \pr(Y\geq k)+O\left(\left(e p\sqrt{n}\right)^{\sqrt n}\right).
\]
\end{cor}

Having established these comparisons, we use very tight bounds on the Poisson tail. Tao gives such a tight bound on his blog \cite{tao2022improved}, where he notes that forms of this bound are given previously \cite{glynn1987upper, talagrand1995concentration}. In the post, Tao gives the proof of the upper bound and leaves the proof of the lower bound to the reader. We prove both sides in the appendix in Section \ref{sec:estimates}.

\begin{lemma}\label{lem:sharp_poisson_tail}
        For $X \sim \Pois(\lambda)$ and $\delta\geq \frac1{\sqrt{\lambda}}$, for sufficiently large $\lambda$
    \begin{equation*}
      \P \big ( X \geq \lambda ( 1 + \delta ) \big ) 
        \leq
        \frac{ e^{ -\lambda h ( \delta ) } }{ \sqrt{\lambda\min\{\delta, \delta^2\}  } },
    \end{equation*}
    where $ h(\delta) = (\delta+1) \log ( \delta + 1 ) - \delta$.

    Moreover, if $\lambda(1+\delta)$ is an integer, then there is a universal constant $\cpois$ such that for sufficiently large $\lambda$
    \[
        \P \big ( X \geq \lambda ( 1 + \delta ) \big ) 
        \geq\cpois\frac{ e^{ -\lambda h ( \delta ) } }{ \sqrt{\lambda\min\{\delta, \delta^2\}  } }.
    \]
\end{lemma}

The integrality assumption is necessary as for very large $\delta$, the difference in probability between $\P( X \geq \lambda (1+\delta) )$ and $\P(X\geq\lambda(1+\delta)+1)$ is large enough that for very small $c>0$, this two sided bound could not possibly hold for $\lceil\lambda(1+\delta)\rceil$ and $\lfloor\lambda(1+\delta)\rfloor+c$ simultaneously. For its use in our paper, the integrality assumption is irrelevant, as we will only need the lower bound in the small $\delta$ regime. 

\begin{cor}\label{eq:poistailex}
For $X\sim Pois(\lambda)$, if $\lambda\delta^3=o_\lambda(1)$ and $\delta\geq \frac1{\sqrt{\lambda}}$, then
    \begin{equation}\label{eq:tightpoisson}
        (1 - o_\lambda(1)) \cpois
        \frac{ e^{  -\frac{\lambda\delta^2}{2}} }{\delta \sqrt{ \lambda  } }
        \leq 
        \P \big ( X \geq \lambda ( 1 + \delta ) \big ) 
        \leq
        (1+o_\lambda(1))\frac{ e^{ -\frac{\lambda\delta^2}{2}} }{ \delta\sqrt{ \lambda  } }.
    \end{equation}
\end{cor}
\begin{proof}
    For the upper bound,
    \[
     \P \big ( X \geq \lambda ( 1 + \delta ) \big ) \leq \frac{ e^{ -\lambda h ( \delta ) } }{ \sqrt{\lambda\min\{\delta, {\delta}^2\}  } }\leq (1+o_\lambda(1))\frac{ e^{ -\frac{\lambda\delta^2}{2}} }{ \delta\sqrt{ \lambda  } }
    \]
    by the Taylor expansion of $h(\delta)$.

    For the lower bound, define $\delta'=\delta+\frac1\lambda$. Then $\lambda(1+\delta)+1=\lambda(1+\delta')$. We have
     \[
     \P \big ( X \geq \lambda ( 1 + \delta ) \big ) =  \P \big ( X \geq \lceil\lambda ( 1 + \delta )\rceil \big ) \geq \cpois \frac{ e^{ -\lambda h ( \delta' ) } }{ \sqrt{\lambda\min\{\delta', {\delta'}^2\}  } }\geq (1-o_\lambda(1))\cpois\frac{ e^{ -\frac{\lambda{\delta}^2}{2}} }{ \delta\sqrt{ \lambda  } }.
    \]
\end{proof}

Because we will need these bounds repeatedly we also state the following estimate for the tails of binomials, which follows from Corollary \ref{cor:bintopoistail} and Lemma \ref{lem:sharp_poisson_tail}.
\begin{cor} \label{lem:binom_heavy_tail}
    If $X \sim \Binom(n,p)$, $np = o( \sqrt{n})$ and $\tau = o(\sqrt{n})$ then 
    \begin{equation*}
        \P \left ( X \geq \tau \right ) \leq (1+o(1)) e^{- \tau \log \tau + \tau \log( np )  - \tau + np}.
    \end{equation*}
\end{cor}

We will use a weaker, simpler version of this bound for low probability events. 
\begin{lemma}[\cite{janson2011random} Theorem 2.1]
    If $X\sim Binom(N,p)$, and $\lambda=Np$, then for $t\geq 0$
    \begin{equation*}
        \pr(X-\lambda\geq t)\leq \exp\left(-\frac{t^2}{2\lambda+2t/3}\right),~~ \pr(X-\lambda\leq -t)\leq \exp\left(-\frac{t^2}{2\lambda}\right).
    \end{equation*}\label{lem:weakertail}
\end{lemma}

Because we are dealing with sparser matrices than \cite{alt2023poisson} and need more precise estimates, we need to bound the sum of \emph{squares} of the degrees of neighbors of high degree vertices. Therefore we require estimates for the sum of distributions with heavy Weibull tails. Such a bound follows from \ref{lem:weakertail} and the tail results in \cite{bakhshizadeh2023sharp}. Justification for this generalization is given in the appendix in Section \ref{sec:estimates}.

\begin{lemma}\label{lem:weibullbound}
For any $n>0$ and $d=o(n^{1/3})$, consider $n$ independent i.i.d. samples $X_1,\ldots X_n\sim \Binom(N,\frac dN)$. There is some constant $\cweib>0$ such that for any $t >n^{2/3}$,
 \[
 \pr\left ( \left |\sum_{i=1}^n X_i^2-\E\left [ \sum_{i=1}^n X_i^2 \right ]\right |>t \right )\leq 2n\exp\left(-\frac{\cweib}{d^3+1}\sqrt{t}\right).
 \]
 Moreover, if $t>2(d^2+1)n^{2/3}$,
  \[
 \pr\left ( \left |\sum_{i=1}^n X_i^2-\E\left [ \sum_{i=1}^n X_i^2 \right ]\right |>t \right )\leq 2n\exp\left(-\cweib\sqrt{t}\right).
 \]
\end{lemma}

\subsection{Spectral properties of graphs}
We use the following spectral bound for trees.

\begin{lemma}\cite{kesten1959symmetric}   \label{lem:forest-bound}
    If $T$ is a forest with maximum degree bounded by $\Delta$, then $\lambda_{\max}(A_T)\le2\sqrt{\Delta-1}$.
\end{lemma}

We use the following result to quantify the proximity of true eigenvalues and eigenvectors to approximate ones.
\begin{lemma}[See \cite{alt2023poisson} Lemma E.1]\label{lem:estimator}
   Consider a self-adjoint matrix $M$ and $\Delta, \epsilon>0$ satisfying $5\epsilon\leq \Delta$. For $\lambda\in \R$, assume $M$ has a unique eigenvalue $\mu$ in the interval $[\lambda-\Delta,\lambda+\Delta]$ with eigenvector $\ww$. If there is a normalized vector $\mathbf{v}$ such that $\|(M-\lambda)\mathbf{v}\|\leq \epsilon$, then 
   \[
   \mu-\lambda=\langle \vv,(M-\lambda)\vv\rangle+O\left (\frac{\epsilon^2}{\Delta}\right ),\|\ww-\mathbf{v}\|=O\left (\frac{\epsilon}{\Delta}\right)
   \]
\end{lemma}

In particular this lemma will be used in the following form:

\begin{lemma} \label{lem:basic_eigenvalue_approximation}
For $i\geq 2$, let $A$ be the adjacency matrix of the ball of radius $i$ around a vertex $x$ of degree $\alpha$, such that $B_i(x)$ is a tree, $ \frac{|S_2(x)|}{\alpha} \leq s(n) \ll \alpha$ and the degree of each vertex in $B_i(x)\setminus{\{x\}}$ is at most $t(n) \leq \frac{\alpha}{5}$. 

Then the maximum eigenvalue $\mu$ of $A$ satisfies
$$\mu =  \sqrt{\alpha} +  O \left ( \frac{s(n) }{ \sqrt{\alpha} } \right ).$$

\end{lemma}
\begin{proof} 
We take as our test vector $\ww$ the eigenvector corresponding to the star graph consisting of the central vertex $x$ and its neighbors. Thus $\ww|_x = \frac1{\sqrt 2}$ and for $y \sim x$, $\ww|_y = \frac1{\sqrt{2 \alpha}}$. Since $B_i(x)$ is a tree, each $z \in S_2(x)$ has exactly one neighbor in $S_1(x)$, so $(A\ww)|_z=\frac{1}{\sqrt{2\alpha}}$. Moreover the number of non-zero entries in $A \ww - \sqrt{\alpha} \ww$ is $|S_2(x)| \leq \alpha s(n) $.

The above implies that
$$\| A \ww - \sqrt{\alpha} \ww \| \leq \sqrt { \alpha s(n)  \frac{1}{2 \alpha }} = \sqrt{ \frac{s(n)}{2} },$$
which corresponds to $\e$ in Lemma \ref{lem:estimator}. 

To utilize Lemma \ref{lem:estimator} we require $\Delta$ such that $A$ has a unique eigenvalue in $[\lambda - \Delta, \lambda + \Delta].$ For this we use eigenvalue interlacing: after deleting the row and column of $A$ corresponding to $x$, the matrix is the adjacency matrix of a forest with degree at most $ t(n) $. By the spectral radius of a tree from Lemma \ref{lem:forest-bound}, the maximum eigenvalue of this submatrix is at most $2 \sqrt{ t(n) }.$ Thus $A$ has at most one eigenvalue in the interval $\left [ 2.1 \sqrt{ t(n) }, 2 \sqrt{\alpha} - 2.1 \sqrt{ t(n) } \right ]$, namely, if any, its maximum eigenvalue. Thus we can take $\Delta = \sqrt{\alpha} - 2.1 \sqrt{ t(n) } \geq \left ( 1 - 2.1/\sqrt{5} \right ) \sqrt{\alpha}$ in Lemma \ref{lem:estimator}.
The estimates on the errors now simply follow from plugging in our values for $\e$ and $\Delta$, since $\e = \sqrt{\frac{s(n)}{2}} \ll \Delta$ by assumption.
\end{proof}

%% file: 3.Regimes.tex
\section{Regimes}\label{sec:regimes}

\newcommand{\llfloor}{}
\newcommand{\rrfloor}{}
\newcommand{\llceil}{}
\newcommand{\rrceil}{}


In a first step we will analyze the spectral contribution of high degree vertices, which we separate into three regimes. After defining these regimes we analyze their sizes. Finally we state a theorem about the approximate diagonalization of the adjacency matrix $A$, that we will use to prove some of our main theorems.

\begin{dfn}
We define the following sets of high degree vertices in our graph.

We define the sets for $m\geq 0$ \[
\XX_m := \left \{x \in [N]: \alpha_x \geq \uu - m \right \}.
\]
This then gives the regimes
\begin{enumerate}
        \item The \emph{fine regime}: $\mathcal{W} := \XX_{\uu^{1/4}}$.
        \item 
        The  \emph{intermediate regime}: $\mathcal{V} :=\XX_{\uu^{2/3}}$.
    \item The \emph{rough regime}: $\mathcal{U} :=\XX_{\uu/2}$.
\end{enumerate}

\end{dfn}
The precise thresholds are not significant and rather of technical nature, our analysis would continue to work if we replaced these thresholds with $\uu-\uu^{c_1}, \uu-\uu^{c_2}, c_3 \uu$ for some constants satisfying $0<c_1<1/2, 1/2<c_2<1$ and $0<c_3<1$.

For large $N$, we have $\WW\subset \VV\subset \UU$. 
First we bound the sizes of these sets. The upper bounds will let us perform union bounds, whereas the lower bound on $|\XX_m|$ tells us that all of our highest degree vertices have almost the same degree.

\begin{lemma}\label{lem:sizes}
For $m\geq 0$, with probability $1-O((\frac{\uu}{d})^{-(m+1/2)})$
\begin{align}\label{eq:sizeupper}
  |\XX_m| \leq  \frac32\left(\frac{\uu}{d}\right)^{m+1/2}.
\end{align}

Moreover, for $1\leq m \leq \uu^{c}$ with $c<1/2$, with probability $1-O \left (\left(\frac{\uu}{d}\right)^{-(m-1/2)}\right )$, 
\begin{align}\label{eq:sizeslower}
   \ |\XX_m|\geq \frac12\left(\frac{\uu}{d}\right)^{m-1/2}.
    \end{align}
\end{lemma}
\begin{proof}
By Lemma 3.11 in \cite{bollobas01randomgraphs}, we know that $\Var(|\XX_m|) = O(\E[|\XX_m|] )$. By a Chebyshev inequality, we have that with probability $1-O(\frac{1}{\E[|\XX_m|]})$,\begin{equation}\label{eq:bollobasvar}
\frac{1}{2}\E[|\XX_m|]\leq |\XX_m|\leq \frac32\E[|\XX_m|].
\end{equation}
Therefore, it is sufficient to show that $\E[|\XX_m|]$ and $\E[|\UU|]$ satisfy the above bounds, and that each is $\omega_N(1)$. Recall that $\mu_k := N \P( \text{ a vertex is of degree $k$ }) = N \Bin(k;N-1,p)$. Thus for any $k \in \mathbb{N}$,
\begin{equation}\label{eq:expratio}
\frac{\mu_{k+1}}{\mu_k} = \frac{\Bin(k+1;N-1,p)}{\Bin(k;N-1,p)} = \frac{(N-k-1)p}{(k+1)(1-p)}.
\end{equation}
By the definition of $\uu$, and the fact that $\mu_k$ monotonically decreases in $k$,
\begin{eqnarray*}
   \mu_{\uu}^{-1} \leq  \mu_{\uu-1}&=& \frac{N\uu(1-\frac dN)}{d(N-\uu)}\mu_{\uu},\\
    \mu_{\uu} \leq \mu_{\uu+1}^{-1}&=& \frac{N(\uu+1)(1-\frac dN)}{d(N-\uu-1)}\mu_{\uu}^{-1}
\end{eqnarray*}
Therefore
  \[
  (1-o_N(1))\sqrt{\frac{d}{\uu}}\leq \mu_\uu\leq (1+o_N(1))\sqrt{\frac{\uu}{d}}.
 \]
We have by \eqref{eq:expratio}, for $m \leq \uu,$
\begin{align*}
        \mu_{\uu-m  } 
        & = \mu_\uu \left ( \frac{N-d}{d} \right )^{m } 
        \prod_{ i = 1}^{m } \frac{\uu-i+1}{N-\uu+i-1} \\
&= (1+o_N(1))\mu_\uu  \frac{\uu^m}{d^m}\prod_{i=1}^m(1-\frac{i-1}{\uu}).
    \end{align*}
    An upper bound on this is $(1+o_N(1))(\frac{\uu}{d})^{m+1/2}$. Summing over all $0\leq n\leq m$ gives \eqref{eq:sizeupper}. Assuming that $m\leq \uu^c$ for $c<1/2$, 
    \[
    (1+o_N(1))\mu_\uu  \frac{\uu^m}{d^m}\prod_{i=1}^m(1-\frac{i-1}{\uu})\geq (1+o_N(1))\mu_\uu  \frac{\uu^m}{d^m} e^{-m^2/\uu}\geq (1+o_N(1))\mu_\uu  \frac{\uu^m}{d^m},
    \]
    giving \eqref{eq:sizeslower}.

    \end{proof}

\begin{cor} \label{cor:regimesizes}
    For our regimes this implies that with high probability
    \begin{align*}
        | \cW | \ll e^{u^c} \text{ for any } c > \frac{1}{4} \\
        | \cV | \ll e^{u^c} \text{ for any } c > \frac{2}{3} \\
        | \cU | \ll N^c \text{ for any } c > \frac{1}{2}
    \end{align*}
\end{cor}

    \begin{rem}
        Note that in the proof we derive bounds on the expected values of the sizes of these sets, which immediately imply bounds on the probability that a given vertex falls into one of the sets, since for any set $\cT$, $\E[|\cT|] = N \P( \text{ vertex } 1 \in \cT ).$
    \end{rem}
    
    Given this, we can introduce a structure theorem, subdividing $A$ according to the different regimes. We will decompose our matrix using a unitary transform $U$ made clear later. Here $D_\WW, D_{\VV\backslash \WW},D_{\UU\backslash \VV}$ are diagonal operators associated with the balls surrounding vertices in $\UU$. 

A summary of the results concerning this decomposition is as follows. 

    \begin{thm}\label{thm:structure}
        With high probability, there is a unitary transformation $U:\R^{N}\rightarrow\R^N$ such that we can write
            \begin{equation}\label{eq:blockdecomp}
       A=U\left(\begin{array}{cccc}
    D_{\WW}&0&0&E_{\WW}^*\\
    0& D_{\VV\backslash \WW}&0&E_{\VV\backslash \WW}^*\\
    0&0& D_{\UU\backslash \VV}+\EE_{\UU\backslash \VV}&E_{\UU\backslash \VV}^*\\
    E_{\WW}& E_{\VV\backslash \WW }&E_{\UU\backslash \VV }&\XX
\end{array}\right)U^{*}
\end{equation}
where \eqref{eq:blockdecomp} satisfies the following
        \begin{enumerate}
        \item 
The \emph{fine regime} operator $D_{\WW}$ is diagonal, of dimension $2|\WW|$, and has at least $e^{\log^{1/8}N}$ eigenvalues of value at least $\sqrt{\uu}-O(\uu^{-3/8})$. 
        \item
            The \emph{intermediate regime} operator $D_{\VV\backslash \WW}$ is diagonal, of dimension $2|\VV\backslash \WW|$, and $\|D_{\VV\backslash \WW}\|\leq \sqrt{\uu}-\Omega(\uu^{-1/4})$.
           \item
        The \emph{rough regime} operator $D_{\UU\backslash \VV}$ is diagonal and of dimension $2|\UU\backslash \VV|$ and satisfies $\|D_{\UU\backslash \VV}\|\leq \sqrt{\uu}-\uu^{1/6-o_N(1)}$.
        \item 
The \emph{bulk} operator $\XX$ satisfies $\|\XX\|=(\frac{1}{\sqrt{2}}+o_N(1))\sqrt{\uu}$.
        \item
        The error terms satisfy
        $\|E_\WW\|,\|E_{\VV\backslash \WW}\|=O( (d^{r}+1)\uu^{-r/2+1})$, $\|\EE_{\UU\backslash \VV}\|+\|E_{\UU\backslash \VV}\|=O((\log\log N)^2)$.

        \end{enumerate}
    \end{thm}
These results, along with results concerning the structure of the eigenvectors associated with the operator surrounding vertices of $\WW$, imply that the edge eigenvectors and eigenvalues come from the operator associated with $D_\WW$. This in turn will give the theorems from Section \ref{sec:intro}. Much of the remainder of the paper is dedicated to showing this decomposition.

%% file: 4.Fine.tex
\section{Fine and Intermediate Regime}\label{sec:fine}

In this section we start by proving structural results about the balls around vertices in $\cV$, which we then use to derive a recursion for the largest eigenvalue and eigenvector of the balls around the vertices in that regime. We also derive a first approximation of the largest eigenvalue for all vertices in $\cV$. In part \ref{sec:fine_eigenvector} we then use these ingredients to prove the exponential decay of this eigenvector for all vertices in $\cV$. In the last part \ref{sec:fine_eigenvalue} we then derive a more precise expression for the eigenvalue of vertices in $\cW$, based on concentration results from part \ref{sec:fine_general}.

\subsection{General Structure} \label{sec:fine_general}
 In this section, we create a test vector and test eigenvalue for the neighborhoods of vertices in $\VV$. In order to do this, we first define an event under which the balls around vertices in $\VV$ have a nice structure. While the eigenvalue and eigenvector will be defined using $B_r(x)$, the following structural results are for slightly larger balls so that we can also bound the error coming from truncating the balls and guarantee that there is no intersection with balls of radius 3 around vertices from $\UU.$

 \begin{dfn}\label{dfn:omegadef}
    Define $\omegam$ to be the event that the following are true.
    \begin{enumerate}
        \item 
        For all $x \neq y\in \VV$, $B_{r+3}(x)\cap B_{r+3}(y)=\emptyset$. 
        \item 
         
        For all $x\in \VV$, $B_{r+3}(x)$ is a tree.
        \item 
       For $1\leq i\leq r$ and every vertex $x\in \VV$, 
\[\left||S_{i}(x)|-d^{i-1}\alpha_x\right|= O (d^{i - 3/2} + 1) \uu^{\frac{7}{8}}\]

        Moreover, for every vertex $x\in \WW$, 
\[\left||S_{i}(x)|-d^{i-1}\alpha_x\right|\leq O (d^{i - 3/2} + 1) \uu^{\frac23}.\]
        
        \item 
For $x\in \VV$, every $y\in B_{r+3}(x) \setminus \{ x \},$ satisfies $N_y\leq \uu^{3/4}$.  

Moreover, for $x\in \WW$, every $y\in B_{r+3}(x)\setminus \{ x \}$ satisfies $N_y\leq \uu^{1/3}$.
\item
For every $x\in \VV$,
\[
\left|\sum_{y\in S_1(x)} N_y^2 -(d^2+d)\alpha_x\right|\leq O \left ( \uu^{3/2} \right ).
\]
Moreover, for every $x\in \WW$,
\[
\left|\sum_{y\in S_1(x)} N_y^2-(d^2+d)\alpha_x\right|\leq O \left ( \uu^{2/3} \right ).
\]

    \end{enumerate}
\end{dfn}

Note that statement 3 in \ref{dfn:omegadef} implies that $|S_2(x) | \leq 2 d \alpha_x$ for our regime of $d$. If $d$ was smaller than $\log^{-c} N,$ for some $c > \frac{1}{4},$ this would no longer be true.

Because of the sparsity of the graph and concentration of independent binomials, we can show that the event $\omegam$ almost always occurs. Considering statements similar to most of these bounds have appeared previously, in, e.g. \cite{alt2023poisson}, we defer the proof of the following lemma to the appendix in Section \ref{sec:structure}.
\begin{lemma} \label{lem:fine_balls_disjoint}
The event $\omegam$ occurs with high probability.
\end{lemma}

We also know from Lemma \ref{lem:max_degree}, that with high probability the maximum degree is bounded by $\uu$. Therefore, we will consistently use these structural properties.

\begin{assumption}
For the rest of Section \ref{sec:fine}, we assume that $\omegam$ occurs. Moreover we assume the high probability event that the maximum degree in $G$ is in $\{ \uu-1, \uu \}$.
\end{assumption}

Under this event we know enough about the structure of the neighborhoods around vertices in the intermediate regime, to analyze its top eigenvalue and eigenvector, which we use as the test eigenvector and eigenvalue.

\begin{dfn}\label{dfn:wdef}
For $x\in \VV$, we define $\lambda_x$ to be the top eigenvalue of $A_{B_r(x)}$. Moreover, define $\ww_{+}(x)$ to be the eigenvector corresponding to $\lambda_x$ and $\ww_{-}(x)$ to be the eigenvector of the most negative eigenvalue of $A_{B_r(x)}$. We use the notation $\ww_{\pm}(x)$, when a statement is true for both $\ww_{+}(x)$ and $\ww_{-}(x)$. Depending on the context, we will also use $\ww_{\pm}(x)$ to denote the above eigenvector padded with $0$'s to make it a vector in $\R^{N}$.
\end{dfn}

Under $\omegam$, $B_r(x)$ is a tree, therefore it is bipartite. Therefore the most negative eigenvalue is $-\lambda_x$, and if $y\in B_r(x)$, then
\[
\ww_{-}(x)|_y=(-1)^{d(x,y)}\ww_{+}(x)|_y.
\]
Moreover, as $B_r(x)$ is a tree, its eigenvectors and eigenvalues satisfy a nice recursion. Consider an eigenvector $\ww$ of $A_{B_r(x)}$ with eigenvalue $\lambda$. We consider the eigenvector equation at $x$:

\[
\lambda \ww|_x=\sum_{y\sim x} \ww|_{y}
\]
which, if $\ww|_x\neq 0$, can be rewritten as 
\begin{equation}\label{eq:rootvertex}
\lambda=\sum_{y\sim x} \frac{\ww|_{y}}{ \ww|_x}.
\end{equation}
More generally, we have
for $v\sim u,  v\geq u$, i.e. for $v$ a child of $u$ in the tree rooted at $x$, if $\ww|_v\neq 0$, 

\begin{equation}\label{eq:eigenvector}
\lambda \ww|_v = \ww|_u + \sum_{y_1\sim v, y_1\geq v} \ww|_{y_1}
\Rightarrow
\ww|_v=\frac1{\lambda-\sum_{y_1\sim v, y_1\geq v}\frac{\ww|_{y_1}}{\ww|_v}}\ww|_u.
\end{equation}
Plugging in \eqref{eq:eigenvector} into \eqref{eq:rootvertex} for every $y_1\sim x$ gives
\[
\lambda=\sum_{y_1\sim x} \frac{\frac1{\lambda-\sum_{y_2\sim y_1, y_2\geq y}\frac{\ww|_{y_2}}{\ww|_{y_1}}}\ww|_x}{ \ww|_x}=\sum_{y_1\sim x} \frac1{\lambda-\sum_{y_2\sim y_1, y_2\geq y_1}\frac{\ww|_{y_2}}{\ww|_{y_1}}}.
\]
Repeating this process for all vertices gives that
\begin{equation}\label{eq:eigdef}
\lambda=\sum_{y_1\sim x}\frac{1}{\lambda-\sum_{y_2\sim y_1, y_2\geq y_1} \frac{1}{\lambda-\sum_{y_3\sim y_2, y_3\geq y_2} \frac{1}{\lambda-\sum_{y_4\sim y_3, y_4\geq y_3} \cdots }}},
\end{equation}
where the right hand side is a continued fraction of at most $r$ levels.

Since $A_{B_r(x)}$ is a connected graph, its adjacency matrix is irreducible and this implies by the Perron-Frobenius theorem, that the top eigenvector $\ww|_{+}(x)$ of $B_{r}(x)$ is the only positive eigenvector, implying in particular that \eqref{eq:eigdef} does not contain any 0 denominators. This means that we can use \eqref{eq:eigdef} for our definition of $\lambda_x$. To further examine this, we require an initial two sided bound, based on $B_r(x)$ being close to a star graph. This will be enough to bound the contribution of balls around vertices in $\cV \setminus \cW$, and for vertices in $\cW$, we will eventually bootstrap it into a tighter bound in Lemma \ref{lem:lambdaexpression}.

\begin{lemma} \label{lem:eigenvalueapproximation} 
For any vertex $x\in \VV$,  

\begin{equation*}
\alpha_x\leq \lambda^2_x \leq \alpha_x +O(d).
\end{equation*}
\end{lemma}

\begin{proof}
As the spectral radius of a star is the square root of the degree of the central vertex, $\lambda^2_x\geq \alpha_x$. For the upper bound, we apply Lemma \ref{lem:basic_eigenvalue_approximation}. By the definition of $\omegam$, (3) we can take $s(n) =  2 d $, and by the definition of $\omegam$, (4) we can take $t(n) = \uu^{3/4}$. Lemma \ref{lem:basic_eigenvalue_approximation} thus implies that $\lambda_x = \sqrt{\alpha_x} + O \left (\frac{d}{\sqrt{\alpha_x}} \right )$ implying that $\lambda_x^2 = \alpha_x + O(d).$

\end{proof}

\subsection{Eigenvector structure} \label{sec:fine_eigenvector}
We now prove that the entries of the eigenvector decay exponentially with the distance from the central vertex.
For easier readability, we will suppress $x$ in the notation for the rest of the section, so we write $\alpha:=\alpha_x$, $\lambda:=\lambda_x$, $\ww_{\pm}=\ww_{\pm}(x),$ and $S_1 = S_i(x).$
Moreover we define $N_{y^*}=\max_{y\in B_{r}(x)}N_y$.

\begin{lemma}\label{lem:vertchange}
If $x\in \VV$, then for $u,v\in B_r(x)$ such that $ u \sim v$ and $u \leq v$, 
\[
\ww_+|_u= \left (1+O\left (\uu^{-1/4}\right ) \right )\lambda\ww_+|_v
\]
and if $x\in \WW$, then for $u\in B_r(x)$, 
\[
\ww_+|_u=\left (1+O \left (\uu^{-2/3} \right ) \right )\lambda\ww_+|_v.
\]
\end{lemma}

\begin{proof}
We will show that for a vertex $u$ such that $u\sim v$, $u\leq v$, it holds that 
\begin{equation}\label{eq:rougheigequation}
\ww_+|_u= \left (1+O \left ( \frac{N_{y^*}}{\lambda^2} \right ) \right ) \lambda\ww_+|_v.
\end{equation}
Together with Lemma \ref{lem:eigenvalueapproximation} and Lemma \ref{dfn:omegadef}, this implies the result.

By the eigenvector equation \eqref{eq:eigenvector}, we must have $\ww_+|_u\leq \lambda\ww_+|_v$. For a lower bound on $\ww_+|_u$, we proceed by induction on the distance from $x$, starting from the leaves in $B_r(x)$ (note that these are not necessarily leaves in $G$). Any leaf $v$ only has one neighbor $u$, making this base case trivial, as there is only one neighbor and
$ 
\ww_+|_u=\lambda \ww_+|_v.
$

Now, assume \eqref{eq:rougheigequation} is true for all $z\geq u$.
Then, applying \eqref{eq:eigenvector} to $v$, we get
\begin{eqnarray}
\nonumber\ww_+|_u&=&\left({\lambda-\sum_{y\sim v, y\geq v}\frac{\ww_+|_y}{\ww_+|_v}}\right) \ww_+|_v \\
\nonumber &\geq& \left(\lambda-\frac{N_{y^*}}{\lambda-O \left (\frac{N_{y^*}}{\lambda} \right )}\right)\ww_+|_v\\
&\geq&\left(1-O \left (\frac{N_{y^*}}{\lambda^2} \right ) \right)\lambda\ww_+|_v,\label{eq:eigenlocal}
\end{eqnarray}
where we used the inductive hypothesis to get the first inequality and then used that by Lemma \ref{lem:fine_balls_disjoint} the rooted trees around $\lambda$ satisfy $N_{y*}\ll \lambda^2$.
\end{proof}
Such a tight bound implies exponential decay on various levels. We can now bound the error from approximating these eigenvectors using the truncation. 

\begin{prop}\label{lem:fineerror}
We define 
\[\Lambda:=\sum_{x\in \VV} \bigg(\lambda_x\ww_{+}(x)\ww_{+}(x)^*-\lambda_x\ww_{-}(x)\ww_{-}(x)^*\bigg).
\]
Then
\[
\max_{x\in \VV, \sigma\in\{\pm1\}}\|(A-\Lambda)\ww_{\sigma}(x)\|=O( (d^{r/2}+1)\uu^{-(r-1)/2}).
\]
\end{prop}
\begin{proof}
For any $x\in \VV, \sigma\in \{\pm1\}$, $\ww_\sigma(x)$  satisfies $(A-\Lambda)\ww_{\sigma}(x)=(A-A_{B_r(x)})\ww_{\sigma}(x)$. The only nonzero entries of this vector are supported on $S_{r+1}(x)$. This holds because the only rows of $(A-A_{B_r(x)})$ that have non-zero entries corresponding to $B_r(x)$ are vertices in $S_{r+1}(x)$, since $B_{r+1}(x)$ are disjoint trees by $\omegam$. By Lemma \ref{lem:eigenvalueapproximation} and Lemma \ref{lem:vertchange}, each entry in $\ww_\sigma(x)$ corresponding to vertices in $S_{r}(x)$ has value at most $(1+o_N(1))\uu^{-r/2}$. Moreover, under $\omegam$, the number of vertices in the $r+1$ level is $d^{r}\alpha_x+O(d^{r - 1/2} + 1) \uu^{\frac{7}{8}} \leq O( (d^r + 1) \u)$. 
Therefore, $\|(A-\Lambda)\ww_{\sigma(x)}\| = O( (d^{r/2}+1) \uu^{-r/2+1/2})$.
\end{proof}

This error term gives us sufficient information about $\VV \backslash \WW$, and we can now focus only on the fine regime $\WW$. The following proposition gives bounds on the total mass the eigenvector assigns to each sphere.

\begin{prop}\label{prop:expdecay}
 For all $x\in \WW$, the eigenvector $\ww_+$ satisfies
\begin{equation}\label{eq:centerbound}
\ww_+|_x=\frac1{\sqrt 2}+O \left ((1+d^{-1/2})\uu^{-1/3} \right ).
\end{equation}
    and for $1\leq i \leq r $,
    \begin{equation}\label{eq:annulusbound}
\|\ww_+|_{S_i}\| =\left(\frac d\alpha\right)^{(i-1)/2}\frac1{\sqrt 2} \left (1+O \left ((1+d^{-1/2}+d^{-i+1})\uu^{-1/3} \right ) \right )
\end{equation}
and 
\begin{equation}\label{eq:complementnorm}
\|\ww_+|_{[N]\backslash B_{i}}\|= \sqrt{\frac{1}{1-\frac d\alpha}}\left(\frac{d}{\alpha}\right)^{i/2}\frac{1}{\sqrt 2}(1+O((1+d^{-1/2}+d^{-r+1})\uu^{-1/3})).
\end{equation}
\end{prop}
\begin{proof}
In this proof we repeatedly use the approximation of $\lambda$ from Lemma \ref{lem:eigenvalueapproximation} in order to replace $\lambda$ by $\alpha^\frac{1}{2}$ or vice-versa up to some small multiplicative error.

First note that Lemma \ref{lem:vertchange} implies that for each $v\in S_{i}$,
\begin{equation*}
\ww_+|_{v}=(1+O(\uu^{-2/3}))\lambda^{-i}(\ww_+|_{x}).
\end{equation*}
Therefore, as we know by Lemma \ref{lem:fine_balls_disjoint} that $|S_i|=d^{i-1}\alpha+O((d^{i-3/2}+1)\uu^{2/3})$,
\begin{eqnarray}
\|\ww_+|_{S_i}\|&=& d^{(i-1)/2}\alpha^{1/2}(1+O((d^{-1/2}+d^{-i+1})\uu^{-1/3})\lambda^{-i}\ww_+|_{x}
\label{eq:firstexpo}.
\end{eqnarray}

Note that for this approximation to hold and for the error term to go to 0, we require $d\geq \uu^{-\frac{1}{3r}}$
The complicated error term is necessary as different terms could maximize for different regimes of $d$. 

Specifically, this means that for the normalized vector $\ww_+$,
\begin{eqnarray*}
(\ww_+|_x)^2+\|\ww_+|_{S_1}\|^2&\geq& 1-O\left(\frac d{\uu}(1+(d^{-1/2}+d^{-1})\uu^{-1/3})\right)\\
&\geq& 1-O\left(\frac d\uu\right).
\end{eqnarray*}

Using this together with \eqref{eq:firstexpo} for $\ww_+|_{S_1}$ and solving for $\ww_+|_x$ gives
\begin{eqnarray*}
(\ww_+|_x)^2=\frac12+O((1+d^{-1/2})\uu^{-1/3}+d\uu^{-1})\\
\ww_+|_x=\frac1{\sqrt 2}+O((1+d^{-1/2})\uu^{-1/3}).
\end{eqnarray*}
Combining \eqref{eq:firstexpo} and \eqref{eq:centerbound},
\begin{eqnarray*}
\|\ww_+|_{S_i}\|&=& d^{(i-1)/2}\alpha^{1/2}(1+O((d^{-1/2}+d^{-i+1})\uu^{-1/3}))\lambda^{-i}(\frac1{\sqrt 2}+O((1+d^{-1/2})\uu^{-1/3}))\\
 &=&d^{(i-1)/2}\alpha^{-(i-1)/2}\frac1{\sqrt 2}(1+O((1+d^{-1/2}+d^{-i+1})\uu^{-1/3})).
\end{eqnarray*}

Similarly, for \eqref{eq:complementnorm}, we have by \eqref{eq:firstexpo},
\begin{eqnarray*}
\|\ww_+|_{[N]\backslash B_{i}}\|^2&=&\sum_{j = i+1}^r \|\ww_+|_{S_j}\|^2\\
&=&\frac{1}{1-\frac{d}{\alpha}}d^{i}\alpha^{-i}\frac1{2}(1+O((1+d^{-1/2}+d^{-r+1})\uu^{-1/3})).
\end{eqnarray*}
\end{proof}

Note that Proposition \ref{prop:expdecay} implies that almost all mass of the vector $\ww_+$ is on $x$ and $S_1.$

\subsection{Eigenvalue structure} \label{sec:fine_eigenvalue}
Along with the eigenvector, we further analyze the eigenvalue.  To do this, we expand \eqref{eq:eigdef} as an infinite sum which we get by repeatedly using the expansion
$\frac{1}{\lambda-q}=\sum_{k=0}^\infty \frac{q^k}{\lambda^{k+1}}$.

 We will bound the eigenvalue $\lambda$ through the moments of the degree sequence surrounding $x\in \cW$. We recall the definitions of $\beta^{(2)}$ and $\beta^{(1,1)}$ from Definition \ref{dfn:betadef}. Note that $\beta^{(1,1)} = |S_3|$ and $\beta^{(2)} = \sum_{y \sim x} N_y^2.$

\begin{lemma}\label{lem:eigdecomp}
With high probability,
for any $x\in \WW$,
    \begin{equation}\label{eq:lambdasquaredfinal}
\lambda^2=\alpha+\lambda^{-2}\beta+\lambda^{-4}(\beta^{(2)}+\beta^{(1,1)})+O((d^2+d)\uu^{-5/3}).
\end{equation}
\end{lemma}

\begin{proof}
    We can rewrite \eqref{eq:eigdef} as 
    \[
    \lambda^2=\sum_{y_1\sim x}\frac{1}{1-\frac{1}{\lambda^2}\sum_{y_2\sim y_1, y_2\geq y_1} \frac{1}{1-\frac{1}{\lambda^2}\sum_{y_3\sim y_2, y_3\geq y_2} \frac{1}{1-\frac{1}{\lambda^2}\sum_{y_4\sim y_3, y_4\geq y_3} \cdots }}}.
    \]
and expand this as 
 \[
    \lambda^2=\sum_{y_1\sim x}\sum_{k_1=0}^{\infty} \left(\frac{1}{\lambda^2}\sum_{y_2\sim y_1, y_2\geq y_1} \frac{1}{1-\frac{1}{\lambda^2}\sum_{y_3\sim y_2, y_3\geq y_2} \frac{1}{1-\frac{1}{\lambda^2}\sum_{y_4\sim y_3, y_4\geq y_3} \cdots }}\right)^{k_1}.
    \]
and once again as
\begin{equation}\label{eq:lambdaexpanded}
     \lambda^2=\sum_{y_1\sim x}\sum_{k_1=0}^{\infty} \left(\frac{1}{\lambda^2}\sum_{y_2\sim y_1, y_2\geq y_1} \sum_{k_2=0}^\infty \left(\frac{1}{\lambda^2}\sum_{y_3\sim y_2, y_3\geq y_2} \frac{1}{1-\frac{1}{\lambda^2}\sum_{y_4\sim y_3, y_4\geq y_3} \cdots }\right)^{k_2}\right)^{k_1}.
\end{equation}
    Of course we could repeat this process, but this is sufficient accuracy for our purposes.

We do the same analysis as before, starting at the innermost level, corresponding to the leaves, and then inducting our way back up the tree. For any vertex $v$, we can write a recursive equation by defining
\[
F(v):=\frac{1}{1-\frac1{\lambda^2}\sum_{y\sim v, y\geq v} F(y)}
\]
which gives that
\[
\lambda^2=\sum_{y_1\sim x}\sum_{k_1=0}^{\infty} \left(\frac{1}{\lambda^2}\sum_{y_2\sim y_1, y_2\geq y_1} \sum_{k_2=0}^\infty \left(\frac{1}{\lambda^2}\sum_{y_3\sim y_2, y_3\geq y_2} F(y_3)\right)^{k_2}\right)^{k_1}.
\]
We now estimate $F_{v}$ for $v\in B_{r}(x)$. For any leaf $v$, as there are no $y\geq v$,  $F(v)=1$. For the rest, we use induction. Recall that $N_y^*$ is the maximum degree in $B_r(x) \setminus \{ x \},$ and that $N_{y^*} \ll \lambda^2$ under the event $\omegam$. Assume that for all $y\sim v, y \geq v$, $F(y)=1+O \left (\frac{N_{y^*}}{\lambda^2} \right ) $. Then
\[
F(v)=\frac{1}{1-\frac1{\lambda^2}\sum_{y\sim v, y\geq v}F(y)}
\leq \frac{1}{1-\frac{N_{y^*}}{\lambda^2} \left (1+O \left (\frac{N_{y^*}}{\lambda^2} \right ) \right )}
=1+O \left (\frac{N_{y^*}}{\lambda^2} \right ) .
\]

Therefore \eqref{eq:lambdaexpanded} becomes
\[
     \lambda^2=\sum_{y_1\sim x}\sum_{k_1=0}^{\infty} \left(\frac{1}{\lambda^2}\sum_{y_2\sim y_1, y_2\geq y_1} \sum_{k_2=0}^\infty \left(\frac{1}{\lambda^2}\sum_{y_3\sim y_2, y_3\geq y_2} 1+O \left (\frac{N_{y^*}}{\lambda^2} \right ) \right)^{k_2}\right)^{k_1}.
\]

We want to show this expansion relies only on the terms in Definition \ref{dfn:betadef}. Therefore we rewrite the first few terms as
\begin{eqnarray*}
\alpha&=&\sum_{y_1\sim x}1\\
\lambda^{-2}\beta&=&\sum_{y_1\sim x}\left(\frac{1}{\lambda^2}\sum_{y_2\sim y_1, y_2\geq y_1} \left(\frac{1}{\lambda^2}\sum_{y_3\sim y_2, y_3\geq y_2}  1+O \left (\frac{N_{y^*}}{\lambda^2} \right )\right)^{0}\right)^1\\
\lambda^{-4}\beta^{(1,1)}&=&\sum_{y_1\sim x}\left(\frac{1}{\lambda^2}\sum_{y_2\sim y_1, y_2\geq y_1}^\infty \left(\frac{1}{\lambda^2}\sum_{y_3\sim y_2, y_3\geq y_2} 1 \right)^{1}\right)^1.
\end{eqnarray*}
The contribution of $\beta^{(2)}$ is more complicated, but as $\beta^{(2)}=\sum_{y_1\sim x} N_y^2$, we can write
\[
\lambda^{-4}\beta^{(2)}=\sum_{y_1\sim x}\left(\frac{1}{\lambda^2}\sum_{y_2\sim y_1, y_2\geq y_1}  \left(\frac{1}{\lambda^2}\sum_{y_3\sim y_2, y_3\geq y_2}  1+O \left (\frac{N_{y^*}}{\lambda^2} \right )\right)^{0}\right)^2
\]
meaning we have
\begin{multline*}
\sum_{y_1\sim x}\left(\frac{1}{\lambda^2}\sum_{y_2\sim y_1, y_2\geq y_1} \sum_{k_2=0}^\infty \left(\frac{1}{\lambda^2}\sum_{y_3\sim y_2, y_3\geq y_2}  1+O \left (\frac{N_{y^*}}{\lambda^2} \right )\right)^{k_2}\right)^{2}-\lambda^{-4}\beta^{(2)}
\\
=\sum_{y_1\sim x}\left( \frac{1}{\lambda^2} \sum_{y_2\sim y_1, y_2\geq y_1} \sum_{k_2=1}^\infty \left(\frac{1}{\lambda^2}\sum_{y_3\sim y_2, y_3\geq y_2}  1+O \left (\frac{N_{y^*}}{\lambda^2} \right ) \right)^{k_2}\right)^{2} \\
+2\lambda^{-4}\sum_{y_1\sim x}N_y\sum_{y_2\sim y_1, y_2\geq y_1} \sum_{k_2=1}^\infty \left(\frac{1}{\lambda^2}\sum_{y_3\sim y_2, y_3\geq y_2}  1+O \left (\frac{N_{y^*}}{\lambda^2} \right ) \right)^{k_2}.
\end{multline*}

Therefore subtracting these terms from \eqref{eq:lambdaexpanded} gives
\begin{eqnarray}
\nonumber\lambda^2-\alpha-\lambda^{-2}\beta-\lambda^{-4}\beta^{(2)}-\lambda^{-4}\beta^{(1,1)}=
& & \sum_{y_1\sim x}\sum_{k_1=3}^{\infty} \left(\frac{1}{\lambda^2}\sum_{y_2\sim y_1, y_2\geq y_1} \sum_{k_2=0}^\infty \left(\frac{1}{\lambda^2}\sum_{y_3\sim y_2, y_3\geq y_2}  1+O \left (\frac{N_{y^*}}{\lambda^2} \right ) \right)^{k_2}\right)^{k_1}\\
\nonumber&+&\sum_{y_1\sim x}\frac{1}{\lambda^2}\sum_{y_2\sim y_1, y_2\geq y_1} \sum_{k_2=2}^\infty \left(\frac{1}{\lambda^2}\sum_{y_3\sim y_2, y_3\geq y_2} 1+O \left (\frac{N_{y^*}}{\lambda^2} \right ) \right)^{k_2}\\
\nonumber&+&\sum_{y_1\sim x}\frac{1}{\lambda^2}\sum_{y_2\sim y_1, y_2\geq y_1} \frac{1}{\lambda^2} \sum_{y_3\sim y_2, y_3\geq y_2} O \left (\frac{N_{y^*}}{\lambda^2} \right ) \\
\nonumber&+&\sum_{y_1\sim x}\left(\frac{1}{\lambda^2}\sum_{y_2\sim y_1, y_2\geq y_1} \sum_{k_2=1}^\infty \left(\frac{1}{\lambda^2}\sum_{y_3\sim y_2, y_3\geq y_2} 1+O \left (\frac{N_{y^*}}{\lambda^2} \right ) \right)^{k_2}\right)^{2}\\
\nonumber &+& 2\lambda^{-4}\sum_{y_1\sim x}N_y\sum_{y_2\sim y_1, y_2\geq y_1} \sum_{k_2=1}^\infty \left(\frac{1}{\lambda^2}\sum_{y_3\sim y_2, y_3\geq y_2} 1+O \left (\frac{N_{y^*}}{\lambda^2} \right ) \right)^{k_2}.
\end{eqnarray}

Each of our error terms now has coefficient $\lambda^{-6}$ or smaller. Therefore, by the same exponential decay and the fact that $\lambda^{-6} \leq \alpha^{-3}$ and $N_{y^*} \ll \lambda^2,$  each of the five sums can be written as 
\begin{equation}\label{eq:lambdaerror}
O(\alpha^{-3} N_{y^*}(\beta^{(1,1)}+\beta^{(2)})).
\end{equation}

By the assumptions of $\omegam$, namely the bound on the maximal degree and the approximations of $\beta^{(1,1)}$ and $\beta^{(2)},$ as well as our bounds for $d$,
for $x\in \WW$
\begin{eqnarray*}
O(\alpha^{-3} N_{y^*}(\beta^{(1,1)}+\beta^{(2)}))&=&O(\uu^{-3}\uu^{1/3}(\beta^{( 2)}+\beta^{(1,1)}))\\
&=&O(\uu^{-8/3}((d^2+d)\uu+\uu^{2/3}+d^2\uu+(d^{3/2}+1)\uu^{2/3}))\\
&=&O((d^2+d)\uu^{-5/3}).
\end{eqnarray*} 
 \end{proof}

We now solve for $\lambda$ in this equation. Note that while $\lambda^2 = \alpha (1+o(1))$ by Lemma \ref{lem:eigenvalueapproximation}, this by itself does not give a precise enough approximation. Instead we bootstrap this initial approximation using the structural properties from $\omegam.$

\begin{lemma}For $x\in \WW$, \label{lem:lambdaexpression}
    \begin{equation}\label{eq:lambdasquaredfinal2}
\lambda^2=\alpha+\frac{\beta}{\alpha}+\frac{\beta^{(1,1)}+\beta^{(2)}}{\alpha^2}-\frac{(\beta)^2}{\alpha^3}+O((d^2+d)\uu^{-5/3}).
\end{equation}
\end{lemma}
\begin{proof}
We rewrite \eqref{eq:lambdasquaredfinal} as
\[
\lambda^2=\alpha+\frac{1}{\alpha+(\lambda^2-\alpha)}\beta + \frac{1}{(\alpha+(\lambda^2-\alpha))^2}\beta^{(1,1)}+\frac{1}{(\alpha+(\lambda^2-\alpha))^2}\beta^{(2)}+O \left ( \left (d^2+d \right )\uu^{-5/3} \right ).
\]
Moreover, by Lemma \ref{lem:eigenvalueapproximation}, $\lambda^2-\alpha=O(d)$, so we can approximate
\[
\frac{1}{\alpha+(\lambda^2-\alpha)}=\frac1{\alpha}\frac{1}{1+\frac{(\lambda^2-\alpha)}{\alpha}}=\frac1\alpha \left (1 + O \left ( \frac{d}{\uu} \right ) \right ).\]
Therefore
\[
\lambda^2 
= \alpha+\frac{1}{\alpha}\beta \left (1+O \left (\frac d{\uu} \right ) \right )
+\frac{1}{\alpha^2} \left (\beta^{(1,1)}+\beta^{(2)} \right )
\left (1+O \left (\frac d{\uu} \right ) \right)
+O \left ( \left ( d^2+d \right )\uu^{-5/3} \right )
= \alpha+\frac{\beta}\alpha
+O \left ( \frac{ d^2 + d }{ \uu } \right ).
\]
This implies that
$ \frac{1}{\lambda^2} = \frac{1}{\alpha + \frac{\beta}{\alpha} + \left ( \lambda^2 - \alpha + \frac{\beta}{\alpha} \right ) } = \frac{1}{\alpha - \frac{\beta}{\alpha} } \left ( 1 + O \left ( \frac{d^2 + d}{\uu^2} \right ) \right ). $
Plugging this more precise approximation for $\lambda^2$ once again into \eqref{eq:lambdasquaredfinal} we get
\begin{eqnarray*}
\lambda^2
&=&\alpha 
+ \frac{1}{\alpha+\frac{\beta}{\alpha}}\beta \left ( 1 + O \left ( \frac{d^2 + d}{\uu^2} \right ) \right )
+ \frac{1}{(\alpha+\frac{\beta}{\alpha})^2}(\beta^{(1,1)}+\beta^{(2)}) \left ( 1 + O \left ( \frac{d^2 + d}{\uu^2} \right ) \right )
+ O((d^2+d)\uu^{-5/3})\\
&=& \alpha+\frac{1}{\alpha+\frac{\beta}{\alpha}}\beta+\frac{1}{\left ( \alpha+\frac{\beta}{\alpha} \right )^2}\left ( \beta^{(1,1)}+\beta^{(2)} \right ) + O\left ( \left (d^2+d \right )\uu^{-5/3} \right ).
\end{eqnarray*}
Due to our bounds on $d$ from \ref{dfn:rdfn}, we can expand
$\frac1{\alpha+\frac{\beta}{\alpha}}=\frac1\alpha-\frac{\beta}{\alpha^3}+O \left ( \frac{d^2}{\uu^3} \right )$, which gives us
\begin{eqnarray*}
\lambda^2&=&\alpha+\frac{\beta}\alpha+\frac{\beta^{(1,1)}+\beta^{(2)}}{\alpha^2}-\frac{(\beta)^2}{\alpha^3}
+ 
O \left ( \left ( d^2+d \right )\uu^{-5/3} \right )
\end{eqnarray*}
as desired. 
\end{proof}

Given this much tighter approximation of $\lambda_x$ for $x \in \cW$, we can now show that the order of the eigenvalues corresponding to $B_r(x)$ for $x \in \WW$ is the same as the lexicographic order of $\alpha_x$ and $\beta_x$. 
\begin{lemma}\label{lem:difference}
For two vertices $u,v\in \WW$, if $\alpha_u \geq \alpha_v,$ $\lambda^2_u-\lambda^2_v \geq \alpha_u-\alpha_v + O((d^{1/2}+1)\uu^{-1/3})$. Moreover, if $\alpha_u = \alpha_v$, and $\beta_u \geq \beta_v,$ then $\lambda_u^2-\lambda_v^2 \geq \frac{\beta_u-\beta_v}\uu+O((1+d^{3/2})\uu^{-4/3})$. Therefore, for $x\in \WW$, $\lambda_x^2$ are ordered according to the lexicographic ordering of $(\alpha_x,\beta_x)$. 
\end{lemma}
\begin{proof}
By \eqref{eq:lambdasquaredfinal2},
\[ \lambda^2_u-\lambda^2_v 
\geq  
\alpha_u-\alpha_v 
- \left| \frac{\beta_u}{\alpha_u}-\frac{\beta_v}{\alpha_v} \right |
+ \left| \frac{\beta^{(1,1)}_u}{\alpha^2_u}-\frac{\beta^{(1,1)}_v}{\alpha^2_v} \right | 
+ \left | \frac{\beta^{(2)}_u}{\alpha^2_u}-\frac{\beta^{(2)}_v}{\alpha^2_v} \right | 
+ \left | \frac{(\beta_u)^2}{\alpha_u^3}-\frac{(\beta_v)^2}{\alpha^3_v} \right|
+O \left ( \left ( d^2+d \right ) \uu^{-5/3} \right ).\]
Using the definition of $\omegam$ (3-5), namely the concentration of $\beta, \beta^{(1,1)}$ and $\beta^{(2)}$, this implies 
\[
\lambda^2_u-\lambda^2_v \geq \alpha_u-\alpha_v + O \left ( \left ( d^{1/2} + 1 \right ) \uu^{-1/3} \right ).
\]

Now, assume that $\alpha_u = \alpha_v$. Then by \eqref{eq:lambdasquaredfinal2}, 
\begin{eqnarray*}
\lambda^2_u-\lambda^2_v &\geq& 
\frac{\beta_u-\beta_v} {\alpha_u} 
- \left|\frac{\beta^{(1,1)}_u}{\alpha^2_u}-\frac{\beta^{(1,1)}_v}{\alpha^2_v} \right | 
+ \left| \frac{\beta^{(2)}_u}{\alpha^2_u}-\frac{\beta^{(2)}_v}{\alpha^2_v} 
\right | 
+ \left| \frac{(\beta_u)^2}{\alpha_u^3}-\frac{(\beta_v)^2}{\alpha^3_v} \right|
+O \left ( \left ( d^2+d \right ) \uu^{-5/3} \right ).
\end{eqnarray*} Once again, by the definition of $\omegam$ (3-5), this implies 
\begin{eqnarray*}
\lambda_u^2 - \lambda_v^2 
& \geq &  \frac{\beta_u-\beta_v}\uu +O \left ( \left (d^{3/2} + 1 \right ) \uu^{-4/3} \right ).
\end{eqnarray*}
To see that this induces a lexicographic ordering, if $\alpha_u\neq \alpha_v$, then $\alpha_u - \alpha_v \geq 1 \gg (d^{1/2}+1)\uu^{-1/3}$. Similarly, if $\alpha_u = \alpha_v$, but $\beta_u \neq \beta_v$, then $\frac{\beta_u-\beta_v}{\uu} \geq \frac1\uu\gg(1+d^{3/2})\uu^{-4/3}$, by our assumptions on $d$ from Definition \ref{dfn:rdfn}.  
\end{proof}

%% file: 5.Rough.tex
\section{Rough Regime}\label{sec:rough}
In this section we construct approximate eigenvectors corresponding to small balls around vertices in $\cU \setminus \cV$, and we derive a less precise approximation for the eigenvalues corresponding to those balls. The approach used in this section is very similar to the approach in Section 6.4 of \cite{alt2023poisson}.
Note a result like Lemma \ref{lem:eigenvalueapproximation}, which we will eventually use to show that eigenvalues from vertices in $\cV \setminus \cW$ cannot compete with the largest ones from vertices in $\cW$, cannot directly be derived for all vertices in $\cU.$ The main obstructions are that the growth of the spheres and the maximum degree in the balls cannot be bounded as tightly as for vertices in $\cV.$

More precisely our goal for this section is to show the following. 
\begin{lemma}\label{lem:roughconstruction}
For $x\in \UU$, we can create a set of unit vectors $\ww_\sigma(x)$ with $\sigma\in \{\pm1\}$ such that
\begin{enumerate}
\item
For $u,v \in \UU, u \neq v$, and $\sigma_1, \sigma_2 \in \{\pm1\}$ we have $supp(\ww_{\sigma_1}(u))\cap supp(\ww_{\sigma_2}(v))=\emptyset$.
\item
We have $\| A \ww_\sigma(x) - \sigma \lambda_x \ww_\sigma(x) \| = O ( \log\log N )$, where $\lambda_x = \sqrt{\alpha_x+\frac{\beta_x}{\alpha_x}}$.
\end{enumerate}
\end{lemma}

Towards this end, we will first prove two weaker structural lemmas around vertices in $\UU$. Firstly we have weaker bounds on the neighborhood growth and the fluctuations of the degrees of the neighbors.
The following lemma has a similar proof as Lemma \ref{lem:fine_balls_disjoint}, and we defer the proof to Section \ref{sec:structure} of the appendix.
\begin{lemma} \label{lem:size_balls}
We have that with high probability for any vertex $x \in \cU$ simultaneously and any $i \leq 3$, the sphere $S_i(x)$ at distance $i$ from $x$ satisfies

$$ \left |S_i(x) \right | = O \left ((d + \log \log N)^{i-1}\uu \right ) .$$
Moreover, 
\[
\sum_{ y \in S_1(x) } \left ( N_y - d \right )^2\leq O \left (  \left ( \log N \right )^2 \right ).
\]

\end{lemma}

Next, we show the balls around vertices in $\cU$ are close to disjoint trees: with high probability the neighborhoods around vertices in the rough regime are almost trees and contain few disjoint paths that contain other vertices from the rough regime. This result basically corresponds to Lemma 5.5 and Lemma 7.3 in \cite{alt2021extremal}, albeit for a different regime of $d$. The proof is also very similar and is deferred to Section \ref{sec:structure} of the appendix.

\begin{lemma} \label{lem:rough_tree_disjoint}
    Let $\cU_\eta = \left \{ x \in [N]: \alpha_x \geq \eta \uu \right \},$ and $s$ be some positive integer, then with high probability for some constants $C_1$ and $C_2$ that only depend on $\eta$, simultaneously for all $x \in \cU_\eta$,
    \begin{enumerate}
        \item $|E \left ( B_s(x) \right )| < \left | V \left ( B_s(x) \right ) \right | - 1 + C_1$
        and
        \item $B_s(x)$ contains less than $C_2$ edge disjoint paths in $B_s(x)$ containing other vertices from $\cU_\eta$.
    \end{enumerate}
\end{lemma}
Note that $\cU = \cU_\frac{1}{2}$ and that it is enough to take constants $C_1 = 2$ and $C_2 > \frac{2}{\eta}$.

We now construct a ``pruned'' graph in which the neighborhoods of vertices in $\cU$ are disjoint trees. The construction works in the same manner as in Lemma 7.2 of \cite{alt2021extremal} and we use it to prove a statement similarly to Proposition 6.19 in \cite{alt2023poisson}. Once more the proof can be found in the appendix.
\begin{lemma}\label{lem:prunedgraph}
Recall that we denote by $G$ the random graph sampled from $\cG(N,\frac dN)$. With high probability, there is a subgraph $\hat G\subset G$ such that for all vertices $x \in \cU,$
\begin{enumerate}
\item
Balls of radius 3 around $x$ in $\hat{G}$, which we denote by $\hat{B}_3(x)$, are disjoint;
\item
The subgraphs induced by $\hat{B}_3(x)$ are trees;
\item
The maximum degree of $G - \hat G$ is bounded; 
\item
For $i\leq 3$, the spheres $\hat{S}_i(x)$ in the pruned graph $\hat{G}$ satisfy
\[
\left |\hat{S}_{i}(x) \right | = O \left ( (d + \log \log N)^{i-1}\uu\right )  ;
\]
\item
We have that
\[
\sum_{y \in \hat{S}_1(x)} \left ( \hat N_y(x)- \frac{\hat{\beta}}{\hat \alpha} \right )^2\leq  O \left ( (\log N)^2 \right ).
\]
\end{enumerate}
\end{lemma}

\begin{proof}[Proof of Lemma \ref{lem:roughconstruction}]
Define $\hat\alpha_x,\hat\beta_x$ to be the parameters of the pruned graph $\hat G$. We use the same test vector as \cite{alt2023poisson}, restated with our parameters this is the unit vector 
\begin{equation} \label{eq:rough_eigenvector}
    \ww_{\sigma}(x):=
    \frac{1}{\sqrt{2}}\left(\frac{\sqrt{\hat \alpha_x}}{\sqrt{\hat\alpha_x+\frac{\hat\beta_x}{\hat\alpha_x}}} \vone_{x}
    +\sigma\frac{1}{\sqrt{\hat\alpha_x}}\vone_{\hat S_1(x)}
    +\frac{1}{\sqrt{\hat\alpha_x(\hat\alpha_x +\frac{\hat\beta_x}{\hat\alpha_x})}}\vone_{\hat S_2(x)}\right). 
\end{equation}

The first statement of the Lemma now follows by Lemma \ref{lem:prunedgraph}. 

We now fix $x$ and $\sigma$ and drop them from our notation for better readability. To prove the second statement we define $\hat\lambda = \sqrt{\hat\alpha+\frac{\hat\beta}{\hat\alpha}}$ and $\hat{A} = A_{\hat{G}}$
and use a triangle inequality to bound $$\| A \ww - \sigma \lambda \ww \| \leq \left \|A \ww - \hat{A} \ww \right \| 
+ \left \|\hat{A} \ww - \sigma \hat{\lambda} \ww  \right \| 
+ \left \|\hat{\lambda} \ww  - \lambda \ww  \right \|.$$

The first term on the right hand side is at most constant, as by Lemma \ref{lem:prunedgraph} the maximum degree of $G - \hat G$ is bounded by a constant and since the maximum row sum is an upper bound for the maximum eigenvalue of a positive symmetric matrix.

Similarly the last term is bounded since by Lemma \ref{lem:prunedgraph}, $\hat{\alpha}$ differs from $\alpha$ by at most a constant, and $\beta$ and $\hat{beta}$ can both be bounded by $(d + \log \log N) \u.$ This implies that
\[
|\lambda - \hat{\lambda}| = \sqrt{\alpha +\frac {\beta }{\alpha }}-\sqrt{\hat \alpha +\frac {\hat\beta }{\hat\alpha }}=\sqrt{\alpha }\sqrt{1+\frac {\beta }{\alpha ^2}}-\sqrt{\alpha }\sqrt{1+\frac{\hat \alpha -\alpha }{\alpha }+\frac {\hat\beta }{\alpha \hat\alpha }}\ll 1.
\]

The second term can be computed as
\begin{align*}
& \sqrt{2} \left ( \hat A  \ww -\sigma \hat \lambda \ww\right )   \\
& =
\left( \sigma \sqrt{\hat{\alpha}}- \sigma \sqrt{\hat{\alpha}}\right)\vone_{x}
+
\sum_{y\in \hat{S}_1} \left(\frac{\sqrt{\hat{\alpha}}}{\hat\lambda}+\frac{\hat{N}_y}{\sqrt{\hat{\alpha}}\hat{\lambda}}-\frac{\hat{\lambda}}{\sqrt{\hat{\alpha}}}\right)\vone_y
+\sum_{y\in \hat S_2}
\left(\frac{\sigma}{\sqrt{\hat{\alpha}}}
-
\frac{\sigma}{\sqrt{\hat\alpha}}\right)
\vone_y
+ \sum_{y\in \hat S_3} \frac{1}{\sqrt{\hat{\alpha}} \hat \lambda}\vone_y.
\end{align*}

Therefore, using Lemma \ref{lem:prunedgraph} and the lower bound on $\alpha_x$ for $x \in \cU$, we get
\begin{eqnarray*}
2 \left \|\hat A \ww -\sigma \hat \lambda \ww \right \|^2 
&\leq&
\frac{1}{\hat{\alpha} \left (\hat\alpha+\frac{\hat\beta}{\hat\alpha} \right )}\sum_{y\in \hat S_1}\left(\hat{N}_y-\frac{\hat{\beta}}{\hat\alpha}\right)^2+|\hat{S}_3|\frac{1}{\hat{\alpha}\left ( \hat{\alpha}+\frac{\hat\beta}{\hat\alpha}\right )}\\
&\leq&
O \left ( (\log\log N)^2 +\frac{(d+\log\log N)^2}{\uu} \right ).
\end{eqnarray*}

Putting these three bounds together and using our expression \eqref{eq:uapprox} for $\u$ and the bounds in \ref{dfn:rdfn} for $d$, we get that, 
\[
\|A \ww -\sigma \lambda \ww\| = O(\log\log N).
\]
\end{proof}

%% file: 6.Structure.tex
\section{Block Decomposition}\label{sec:finaldecomp}

This section is devoted to proving Theorem \ref{thm:structure}, for which we use results from Sections \ref{sec:fine} and \ref{sec:rough}. For this we first bound the contribution of the remainder of the matrix, i.e. from vectors that are orthogonal to the largest eigenvectors of small balls around the high-degree vertices.
We end this section by using the approximate diagonalizetion to prove Theorem \ref{thm:maineigenvalue}.

\subsection{Bulk vectors}
We now prove that there is no contribution from any other vector. To do this, we use the decomposition of Krivelevich and Sudakov \cite{krivelevich2003largest}. This lets us reduce to only considering stars around high degree vertices. Here, we state a structure theorem that combines elements of the proof of Theorem 1.1 and Lemma 2.2 in \cite{krivelevich2003largest}. To do this, recall that $\Gamma_x$ are all vertices adjacent to $x$ and consider the sets of vertices 
    \begin{eqnarray*}
\YY_1&:=&\left\{x\in[N]:\alpha_x\geq \uu^{3/4}\right\}\\
\YY_2&:=&\left\{x\in[N]:\Gamma_x\cap \YY_1\neq \emptyset\right\}
\end{eqnarray*}

\begin{prop}[\cite{krivelevich2003largest}]\label{prop:ks}
    For $G\sim \GG(N,\frac{d}{N})$ graph, if $d=o(\log^{1/2} N)$, then with high probability, there is a subgraph $\mathcal{H}\subset G$ such that 
    \begin{enumerate}
        \item
        $\mathcal{H}$ is contained in the bipartite subgraph induced by $\YY_1$ and $\YY_2$,
    
        \item
        $\mathcal{H}$ is a union of stars on disjoint vertices,
         \item 
        $\|A_{G\backslash \mathcal{H}}\|=O(d+\uu^{7/16})$.
    \end{enumerate}
\end{prop}
Note that $\mathcal{H}$ is the graph $G_6-H$ in \cite{krivelevich2003largest}. This is strong enough to show that no other vector interferes in the largest eigenvalues.

Define $U_{\UU}$ to be the space spanned by $\ww_{\pm}(x)$ as defined in Definition \ref{dfn:wdef} for $x\in \VV$, and $\ww_{\pm}(x)$ as defined in Lemma \ref{lem:roughconstruction} for $x\in \UU\backslash \VV$. 
\begin{lemma}\label{lem:bulk} For any vector $\vv \in \mathbb{R}^N$ that satisfies $\|v |\ = 1$ and $\vv\perp U_\UU$,
    \begin{equation}
        \langle \vv, A \vv \rangle \leq (1+o_N(1))\frac{1}{\sqrt{2}}\sqrt{\uu}.
    \end{equation}
\end{lemma}
\begin{proof}
By Proposition \ref{prop:ks}, we have that
\[
\langle \vv,A\vv\rangle \leq \max_{x\in \YY_1} \langle \vv,A_{\tilde B_1(x)}\vv\rangle+O(d+\uu^{7/16})
\]
where $\tilde B_{1}(x)$ is the ball of radius $1$ around $x$ in $\mathcal H$ from Proposition \ref{prop:ks} and $A_{\tilde B_{1}(x)}$ is the adjacency matrix of the graph on the vertices $[N]$ induced by the ball $\tilde B_{1}(x).$

Therefore we split into cases based on the degree of $x$. For $x\in \VV$, we know that $\vv\perp \frac1{\sqrt{2}}(\ww_+(x)+\ww_-(x))$. Therefore
\[
\langle \vv, \vone_x\rangle =\langle \vv ,\frac1{\sqrt{2}}(\ww_+(x)+\ww_-(x))\rangle +\langle \vv ,\vone_x-\frac1{\sqrt{2}}(\ww_+(x)+\ww_-(x))\rangle\leq \|\vone_x-\frac1{\sqrt{2}}(\ww_+(x)+\ww_-(x))\|.
\]
By Proposition \ref{prop:expdecay}, 
\[
\|\vone_x-\frac1{\sqrt{2}}(\ww_+(x)+\ww_-(x))\|=O(\uu^{-1/3}).
\]
 We then have
\[
\langle \vv,A_{\tilde B_{1}(x)} \vv\rangle\leq 2|\langle \vv, \vone_x\rangle|\sum_{y\sim x}|\langle \vv, \vone_y\rangle|=O(\uu^{-1/3}\cdot \sqrt{\uu}).
\]
 
Similarly, for $x\in \UU\backslash \VV$, we have $\langle \vv, \vone_x\rangle\leq \|\vone_x-\frac1{\sqrt{2}}(\ww_+(x)+\ww_-(x))\|$. By the definition of the eigenvector in \eqref{eq:rough_eigenvector}, and using properties from Lemma \ref{lem:prunedgraph},
we have that 
\[
\left \|\vone_x-\frac1{\sqrt{2}}(\ww_+(x)+\ww_-(x)) \right \|
=
O \left ( 
\sqrt{
\left [ 
\frac{\sqrt{\hat{\alpha}+\frac{\hat{\beta}}{\hat{\alpha}}}-\sqrt{\hat{\alpha}} }{\sqrt{\hat{\alpha}+\frac{\hat{\beta}}{\hat{\alpha}}}} 
\right ]^2
+
\hat{\beta} \frac{1}{\hat{\alpha} \left ( \hat{\alpha} + \frac{ \hat{\beta}}{\hat{\alpha}} \right )} 
}
\right )
=
O \left ( \sqrt{ \frac{ \beta^2}{\alpha^3} + \frac{\beta}{\alpha^{2}} } \right) 
= 
O \left ( \frac{ d+ \log \log N}{\uu} \right ).
\]

Therefore, by the same argument as before
\[
\langle \vv,A_{\tilde B_1(x)}\vv\rangle=
O \left ( \frac{ d + \log \log N}{\sqrt{\uu}} \right )
=o_N \left ( 1 \right ).
\]

For any vertex $x\in \YY_1\backslash \UU$, the maximum degree is $\uu/2$ and the spectral norm is given by the spectral radius of a star graph, namely for any vector $\vv$ such that $\| \vv \| = 1,$
\[\langle \vv,A_{\tilde B_1(x)}\vv\rangle\leq \sqrt{\frac{\uu}{2}}.\] Combining these cases gives the result.
\end{proof}

\subsection{Structure Theorem}
We now have all the ingredients to prove the structure theorem.
\begin{proof}[Proof of Theorem \ref{thm:structure}]
We can now fully define the block decomposition from \eqref{eq:blockdecomp},
   \begin{equation*}
       A=U\left(\begin{array}{cccc}
    D_{\WW}&0&0&E_{\WW}^*\\
    0& D_{\VV\backslash \WW}&0&E_{\VV\backslash \WW}^*\\
    0&0& D_{\UU\backslash \VV}+\EE_{\UU\backslash \VV}&E_{\UU\backslash \VV}^*\\
    E_{\WW}& E_{\VV\backslash \WW }&E_{\UU\backslash \VV }&\XX
\end{array}\right)U^{*}
\end{equation*}

We first define the unitary matrix $U$. We set the first $2|\WW|$ columns of $U$ to vectors $\ww_{\pm}(x)$  for $x\in \WW$, and denote this part of the matrix by $U_\cW$, then we set the next $2|\VV\backslash \WW|$ columns to $\ww_{\pm}(x)$ for $x \in \VV\backslash \WW$ and denote this part of the matrix by $U_{\cV \setminus \cW}$, for $\ww_{\pm}(x)$ as defined in Definition \ref{dfn:wdef}. The next $2 | \cU \setminus \cV|$ columns are the vectors $\ww_\pm(x)$ for $x \in \cU \setminus \cV$ as defined in Lemma \ref{lem:roughconstruction}, and we denote this part of the matrix by $U_{\cU \setminus \cV}$. We denote these three parts of the matrix together by $U_{\UU}$. We then complete $U$ arbitrarily with a basis of the rest of $\R^N$, namely $U_{\UU^\perp}\subset \R^N$. 

It is implied by the definition of $U$ that the diagonal matrices $D_\cW$ and $D_{\cV \setminus \cW}$ have entries $\sigma \lambda_x$ on the diagonal, i.e. the eigenvalue of the truncated balls corresponding to each $w_\sigma(x)$. The diagonal operator $D_{\UU\backslash \VV}$, is defined to have entries $\sigma \lambda_x$, from Lemma \ref{lem:roughconstruction}.

$0$'s exist in the requisite places as we can assume by $\omegam$ that balls of vertices in $\VV$ are disjoint, and for each $x\in \VV$, the maximum degree of a vertex in $B_{r+3}(x)\backslash x$ is $\uu^{3/4}$, implying that there are no intersections with balls of radius 3 around vertices in $\cU \setminus \cV$. By Lemma \ref{lem:sizes}, with high probability, the $e^{\log^{1/8}N}$ vertices of largest degree have degree at least $\uu-2\log^{1/8}N$. Therefore the eigenvalues corresponding to these vertices have value at least $\sqrt{\uu-2\log^{1/8}N}=\sqrt{\uu}-O(\log^{-3/8}N)$ by Lemma \ref{lem:eigenvalueapproximation}.

To get a bound on $D_{\cV \setminus \cW},$ we use the upper bound from Lemma \ref{lem:eigenvalueapproximation}. This gives that for any vertex $x \in \cV,$ and for the range of $d$ defined in \ref{dfn:rdfn},
$$ 
\lambda_x 
\leq 
\sqrt{\uu - \uu^\frac{1}{4} + O(d)} 
= 
\sqrt{\u} 
- \frac{ \uu^{-\frac{1}{2}}}{2} 
+ O \left ( \frac{d}{\sqrt{\uu}} \right )
\leq 
\sqrt{\u} 
- \Theta \left ( \uu^{-\frac{1}{2}} \right ).
$$

By Lemma \ref{lem:fineerror}, for any $x\in \VV$, $\sigma\in \{\pm 1\}$, $\|(A-\Lambda)\ww_{\sigma}(x)\|=O((d^{r/2}+1)\uu^{-(r-1)/2})$. We will now show that this implies a bound on $\| E_W \|:$
Using that $U^*_{\cW}$ is a surjective projection of $\mathbb{R}^N$ onto $\mathbb{R}^{2 | \cW|}$ and $U_{\UU^\perp}$ is an injective embedding of $\mathbb{R}^{N - 2|\cU|}$ onto $\mathbb{R}^N,$ we can transform $E_W$, which maps $\mathbb{R}^{2|\cW|}$ to $\mathbb{R}^{N - 2 |\cU|},$ into an operator from $\mathbb{R}^N$ to $\mathbb{R}^N$, with the same spectral properties. Therefore, using additionally that outside of the choice of $\sigma\in \{\pm1\}$, the supports of $\ww_\sigma(x)$ are independent,
\begin{multline*}
\|E_{\WW}\|=
\|U_{\UU^\perp}E_{\WW}U_\WW^*\|
=\max_{\vv\in span(U_\WW), \| \vv \| = 1}\|U_{\UU^\perp}E_\WW U_\WW^*\vv\|
\leq 
\max_{x\in \WW,\sigma\in \{\pm 1\}} 2 \|U_{\UU^\perp}E_\WW U_\WW^*\ww_\sigma (x)\|\\
= \max_{x\in \WW,\sigma\in \{\pm 1\}} 2 \|(A-U_{\WW}D_\WW U_\WW^*)\ww_\sigma (x)\|
= \max_{x\in \WW,\sigma\in \{\pm 1\}}2\|(A-\Lambda)\ww_{\sigma}(x)\|
= O( (d^{r/2}+1)\uu^{-(r-1)/2}).
\end{multline*}
Here we use the definition of $\Lambda$ from Lemma \ref{lem:fineerror}. The same is true for $\|E_{\VV\backslash \WW}\|$.

Instead of bounding the operator norms of $E_{\cU \setminus \cV}$ and $\cE_{\cU \setminus \cV}$ individually, we bound the operator norm of their concatenation, which will be an upper bound for both.
Similarly to before we can write
$$
\left \| \begin{bmatrix}
   \cE_{\cU \setminus \cV} \\
   E_{\cU \setminus \cV}
\end{bmatrix} \right \| 
= \left \| \begin{bmatrix}
    U_{\cU \setminus \cV} U_{\cU^\perp}
\end{bmatrix} \begin{bmatrix}
   \cE_{\cU \setminus \cV} \\
   E_{\cU \setminus \cV}
\end{bmatrix} 
U_{\cU\setminus \cV}^*
\right\|. $$
By subsequently proceeding in the same way as for the error coming from the fine regime, Lemma \ref{lem:roughconstruction}, 2., gives the desired bound.

Finally for the bulk, $\|\XX\|\leq (\frac1{\sqrt2}+o_N(1))\sqrt\uu$ by Lemma \ref{lem:bulk}.
\end{proof}

This immediately gives the following. 
\begin{cor}\label{claim:errorbound} 
\[
\left\|\left(\begin{array}{ccc}
     D_{\VV\backslash \WW}&0&E_{\VV\backslash \WW}^*\\
    0& D_{\UU\backslash \VV}+\EE_{\UU\backslash \VV}&E_{\UU\backslash \VV}^*\\
     E_{\VV\backslash \WW }&E_{\UU\backslash \VV }&\XX
\end{array}\right)\right\|\leq \sqrt{\uu}-\Theta(\uu^{-1/4}).
\]
\end{cor}
\begin{proof}
This norm is at most
\begin{multline}
\max \left \{\|D_{\VV\backslash \WW}\|,\left\|\left(\begin{array}{cc}
   D_{\UU\backslash \VV}+\EE_{\UU\backslash \VV}&E_{\UU\backslash \VV}^*\\
     E_{\UU\backslash \VV }&\XX
\end{array}\right)\right\| \right \}+\|E_{\VV\backslash \WW}\|\\
\leq \max \Big \{\|D_{\VV\backslash \WW}\|,\max \big \{\|D_{\UU\backslash \VV} + \cE_{\cU \setminus \cV} \|,\|\XX\| \big \}+\|E_{\UU\backslash \VV} \Big \}+\|E_{\VV\backslash \WW}\|.
\end{multline}
The bound then follows from Theorem \ref{thm:structure} and our bounds on $d$ from \ref{dfn:rdfn}. 
\end{proof}

With this, we can show that the top eigenvalues correspond to $\WW$. 

\begin{prop}\label{prop:eigstructure}
    For every $k\leq e^{\log^{1/8}N}$, the $k$th largest eigenvalue $\lambda$ of $A$ corresponds to the $k$th largest lexicographic maximizer $x\in \WW$ of $(\alpha_x,\beta_x)$ in that $\lambda=\lambda_x+O( (d^r+1)\uu^{-r+1})$.

\end{prop}

\begin{proof}
We start with the matrix 
\[
U\left(\begin{array}{cccc}
    D_{\WW}&0&0&0\\
    0& D_{\VV\backslash \WW}&0&E_{\VV\backslash \WW}^*\\
    0&0& D_{\UU\backslash \VV}+\EE_{\UU\backslash \VV}&E_{\UU\backslash \VV}^*\\
    0& E_{\VV\backslash \WW }&E_{\UU\backslash \VV }&\XX
\end{array}\right)U^{*}
\]

 We then make the transformation by performing the summation
\[
U\left(\begin{array}{cccc}
    D_{\WW}&0&0&0\\
    0& D_{\VV\backslash \WW}&0&E_{\VV\backslash \WW}^*\\
    0&0& D_{\UU\backslash \VV}+\EE_{\UU\backslash \VV}&E_{\UU\backslash \VV}^*\\
    0& E_{\VV\backslash \WW }&E_{\UU\backslash \VV }&\XX
\end{array}\right)U^{*}
+U \left(\begin{array}{cccc}
    0&0&0&E_{\WW}^*\\
     0&0&0&0\\
      0&0&0&0\\
     E_{\WW}&0&0&0.
\end{array}\right)U^{*}
\]
By perturbation theory, e.g. \cite[7.1.1]{baumgartel1985analytic}, \cite[Equation 23]{bamieh2020tutorial}, each eigenvalue has changed by at most $O(\|E_\WW\|^2)=O( (d^r+1)\uu^{-r+1})$.
Moreover, as $r\geq 5$, by Theorem \ref{thm:structure} and Corollary \ref{claim:errorbound}, after the perturbation, nothing outside of $D_{\WW}$ can correspond to one of the $e^{\log^{1/8}N}$ largest eigenvalues. Moreover, by Lemma \ref{lem:difference}, the ordering of eigenvalues must match the ordering in $D_W$, inducing the lexicographic ordering.
\end{proof}

\begin{proof}[Proof of Theorem \ref{thm:maineigenvalue}]
Consider the vertex $x$ corresponding to one of the $e^{\log^{1/8}N}$ largest eigenvalues. By Lemma \ref{lem:lambdaexpression} and the concentration results from the definition of $\omegam$, (3-5), we have that 
\[\lambda_x^2=\alpha_x+\frac{\beta_x}{\alpha_x}+\frac{d^2}{\alpha_x}+\frac{d^2+d}{\alpha_x}-\frac{d^2}{\alpha_x}+O(\frac{d^{3/2}+1}{\uu^{4/3}}).\]
 Therefore by Proposition \ref{prop:eigstructure}, the true eigenvalue $\lambda $ satisfies
\begin{eqnarray*}
\lambda&=&\sqrt{\alpha_x+\frac{\beta_x}{\alpha_x}+\frac{d^2+d}{\alpha_x}+O(\frac{d^{3/2}+1}{\uu^{4/3}})}+O( (d^{r}+1)\uu^{-{(r-1)}}).\\
&=&\sqrt{\alpha_x+\frac{\beta_x}{\alpha_x}+\frac{d^2+d}{\alpha_x}}+O((d^{3/2}+1)\uu^{-11/6}+ (d^{r}+1)\uu^{-{(r-1)}})\\
&=&\sqrt{\alpha_x+\frac{\beta_x}{\alpha_x}+\frac{d^2+d}{\alpha_x}}+O((d^{3/2}+1)\uu^{-11/6}).
\end{eqnarray*}
as we have assumed $r\geq5$. The lexicographic ordering follows immediately from Proposition \ref{prop:eigstructure}.

\end{proof}

%% file: 7.Anticoncentration.tex
\section{Anticoncentration}\label{sec:anticoncentration}
In this section we prove that the distribution for $\lambda^2$ is anticoncentrated at the edge of the spectrum. To start, we show that the joint distribution of $\{\alpha_x\}_{x\in \WW}$ and $\{\beta_x\}_{x\in \WW}$ is approximately that of independent Poissons.  This result is similar to \cite{alt2023poisson} Lemma 7.1, but proven in a somewhat different way. We then proceed similarly to \cite{alt2023poisson} to derive that this implies Theorem \ref{thm:process}, which says that that the maximal pairs $(\alpha_x, \beta_x)$ are close to the maximal values of a Poisson process. This result is then used to show Lemma \ref{lem:eigenvaluespacing}, which gives a lower bound on the distance between the size of the 2-spheres around vertices with maximal or almost maximal degree. In the final lemma of this section we show that this implies spacing of the largest eigenvalues.

\begin{lemma}\label{lem:totalvariation}
    For $d$ according to \ref{dfn:rdfn}, let $G$ be a graph generated from the Erd\H{o}s-R\'{e}nyi graph distribution $\cG \left ( N, \frac{d}{N} \right )$. Moreover, for $k \leq e^{log^{2/3} N}$, consider vertices $z_1,\ldots z_k\in[N]$, along with $\frac12\uu\leq v_1,\ldots, v_k\leq 2\uu$ and $w_1,\ldots w_k$ such that  $1\leq w_i\leq dv_i+\uu^{7/8}$ for $1\leq i\leq k$. Then define i.i.d. $X_1,X_2,\ldots X_k\sim Pois(d)$ and independent $Y_{v_1}\sim Pois(dv_1),Y_{v_2}\sim Pois(dv_2),\ldots, Y_{v_k} \sim Pois(dv_k)$. If $A$ is the event that there are no intersections between the balls of radius 1 around the vertices $z_1, \dots, z_k$, and no edges from $S_1(z_i)$ to $S_1(z_j)$ for any $i,j$ (including i=j), then
\begin{equation}\label{eq:decorr}
\pr\left(\bigcap_{i=1}^k
\left \{ \alpha_{z_i} = v_i, \beta_{z_i} =w_i\right \}\cap A\right)
= \left(1+N^{-1+o_N(1)} \right)\prod_{i=1}^k\pr(X_i=v_i) \pr(Y_{v_i}=w_i)
\end{equation}
\end{lemma}

\begin{proof}
    Let $Z = \{ z_1, \dots, z_k\}$, with $A_Z$ being the adjacency matrix of $Z$. Recall that $\Gamma_{z_i}$ denotes the neighbors of a vertex $z_i$.
    
    We analyse the event that $\cap_i \{ \alpha_{z_i} = v_i \}$ and that there are no edges between any $z_i$, as well as no intersection between the neighborhoods of the $z_i$, sequentially. 
    
    That $A_Z = 0$ happens with probability $(1-d/N)^{\binom{k}{2}}.$ Then we first need to choose exactly $v_1$ vertices among $[N] \setminus Z$ connected to $z_1$. 
    Subsequently we need to choose exactly $v_2$ vertices among $[N] \setminus ( Z \cup \Gamma_{z_1} )$ and make sure that there are no edges between $z_2$ and $\Gamma_{z_1}$, and so on. 
    Note that this way the edges we consider at each step are independent of the previously considered events and moreover the number of edges between $z_i$ and $[N] \setminus ( Z \cup \cup_{j=1}^{i-1} \Gamma_j)$ is binomially distributed with parameters $N - k - \sum_{j=1}^{i-1} v_i$ and $d/N.$ This gives 
    
    \begin{align*}
        & \P \left ( \cap_{i=1}^k \{ \alpha_{z_i} = v_i \} \cap A_Z = 0 \cap \{ \cap_{i=1}^k \Gamma_{z_i} = \emptyset \} \right ) \\
        & = \left [ 
            \prod_{i = 1}^k 
            \P \left ( \Binom \left ( N - k - \sum_{j=1}^{i-1} v_j, \frac{d}{N} \right ) = v_i \right ) 
            \left ( 1 - \frac{d}{N} \right )^{\sum_{j=1}^{i-1} v_i}
        \right ]
        \left ( 1 - \frac{d}{N} \right )^{\binom{k}{2}}.
    \end{align*}
    We now use Lemma \ref{lem:binomtopois} to approximate the binomial probabilities, the bound on the error terms follows from our assumptions on $v_i$, $k$ and the bounds on $d$ from \ref{dfn:rdfn}.
    
    \begin{align*}
         \P \left ( \Binom \left ( N - k - \sum_{j=1}^{i-1} v_j, \frac{d}{N} \right ) = v_i  \right )
        & = \left ( 1 + \tilde{O} \left ( N^{-1} \right ) \right )
        \P \left ( \Pois \left ( d - \frac{ d \left (k + \sum_{j=1}^{i-1} v_j \right )}{N} \right ) = v_i \right ) \\
        & = \left( 1 + N^{-1+o_N(1)} \right ) \P \left (\Pois(d) = v_i  \right ).
    \end{align*}
    
    Moreover $\left ( 1 - \frac{d}{N} \right )^{\sum_{j=1}^{i-1} v_i}$ as well as $\left ( 1 - \frac{d}{N} \right )^{\binom{k}{2}}$ can also be written as $1 + N^{-1+o_N(1)}$. There are $N^{o_N(1)}$ such error terms, which all together implies that
    \begin{align} \label{eq:alpha_approx}
        \P \left ( \cap_{i=1}^k \{ \alpha_{z_i} = v_i \} \cap \{ A_Z = 0 \} \ \cap \{ \cap_{i=1}^k \Gamma_{z_i} = \emptyset \} \right )
        & = 
        \left( 1 + N^{-1+o_N(1)} \right ) \P \left (\Pois(d) = v_i  \right ).
    \end{align}
    We now condition on the event $\cap_{i=1}^k \{ \alpha_{z_i} = v_i \} \cap \{ A_Z = 0 \} \cap \{ \cap_{i=1}^k \Gamma_{z_i} = \emptyset \}$ and similarly analyse $\{\beta_{z_i} = w_i \}$. Note that the event $A$ does not require that the $S_2(z_i)$ are all disjoint which makes the analysis slightly simpler. The number of edges from $S_1(z_i)$ to $[N] \setminus ( Z \cup \Gamma_Z )$ is Binomial with parameters $v_i ( N - k - \sum_{j=1}^k v_j)$ and $\frac{d}{N}$, and those random variables are independent since we condition on $A_Z = 0 $ and $ \cap_{i=1}^k \Gamma_{z_i} = \emptyset $. Finally, the probability that there are no edges within and across any $\Gamma_{z_i}$ is equal to $(1-d/N)^{\sum_{i \neq j} v_iv_j + \sum_i \binom{v_i}{2}}$.

Thus, defining $A_\Gamma$ to be the adjacency matrix of $ \Gamma_Z = \cup_{i=1}^k \Gamma_{z_i},$ we get
\begin{align*}
    & \P \left ( \cap_{i=1}^k \{ \beta_{z_i} = w_i \} \cap A_\Gamma = 0 \big | \cap_{i=1}^k \{ \alpha_{z_i} = v_i \} \cap \{ A_Z = 0 \} \cap \left \{ \cap_{i=1}^k \Gamma_{z_i} = \emptyset \right \} \right ) \\
    & = 
    \left [
        \prod_{i = 1}^{k} 
        \P \left ( \Binom \left ( v_i \left ( N - k - \sum_{j = 1}^k v_j \right ), \frac{d}{N} \right ) = w_i \right )
        \right ]
    \left ( 1 - \frac{d}{N} \right )^{ \sum_{i,j=1, i \neq j}^k v_iv_j + \sum_{i=1}^k \binom{v_i}{2} }
\end{align*}
The last term can as before be written as $1 + N^{-1+o_N(1)}$. When $v_i=0$, we can immediately replace the Binomial random variables by Poisson random variables with parameter 0, since they are both constant 0. For $v_i>0$, we once more use the Poisson approximation from Lemma \ref{lem:binomtopois}, which together with our bounds on $v_i$, $w_i$, $k$ and $d$, gives
\begin{align*}
    \P \left ( 
    \Binom \left ( v_i \left ( N - k - \sum_{j = 1}^k v_j \right ), \frac{d}{N} 
    \right ) = w_i \right )
    & = 
    \left ( 1 + \tilde{O}( N^{-1} )\right )
    \P \left ( \Pois \left ( v_i d \left ( 1 - \frac{ k + \sum_{j=1}^k v_j }{N} \right ) \right ) = w_i \right ) \\
    & = 
    \left ( 1 + N^{-1+o_N(1)}\right )
    \P \left ( \Pois \left ( v_i d ) = w_i \right ) \right )
\end{align*}
Combining this with \eqref{eq:alpha_approx}, implies the result.
\end{proof}

We need to extend our analysis from a small number of vertices to the entire set. Therefore, we once again emulate the argument of \cite{alt2023poisson}. Define the parameter $(Z_x)_{x\in[N]}$ to be an exchangeable family of random variables in a measurable space $\mathcal{Z}$. We then define for $F\subset \mathcal Z^k$ and the point process $\Phi$, 
\[
q_{\Phi}(F):=N(N-1)\cdots (N-k+1)\pr((Z_1,\ldots Z_k)\in F). 
\]
\begin{lemma}[\cite{alt2023poisson} Lemma 7.8]\label{lem:smalltoall}

For $n,m\in \N$ and disjoint, measurable $I_1,\ldots I_n$
\begin{multline}
\pr \left ( \Phi(I_1)=k_1,\ldots,\Phi(I_n)=k_n \right )
=\frac1{k_1!\cdots k_{n}!}\sum_{\substack{ \ell_1,\ldots,\ell_n \\
\sum \ell_i \leq m}}
\frac{(-1)^{\sum_{i}\ell_i}}{\ell_1!\cdots \ell_n!}q_\Phi \left ( I_1^{k_1+\ell_1}\times \cdots \times I_n^{k_n+\ell_n} \right )\\
+O \left (\frac1{k_1!\cdots k_{n}!}\sum_{\substack{ \ell_1,\ldots,\ell_n \\
\sum \ell_i = m+1}}
\frac{(-1)^{\sum_{i}\ell_i}}{\ell_1!\cdots \ell_n!}q_\Phi \left (I_1^{k_1+\ell_1}\times \cdots \times I_n^{k_n+\ell_n} \right ) \right ).
\end{multline}
\end{lemma}

We can then use this to show that edge eigenvalues form a Poisson process.
\begin{proof}[Proof of Theorem \ref{thm:process}]
By Lemma \ref{lem:sizes}, Theorem \ref{thm:maineigenvalue}, and Lemma \ref{lem:difference}, with high probability, only vertices of degree between $\uu-2\log^{1/8}N$ and $\uu$ contribute to the top $e^{\log^{1/8}N}$ eigenvalues. 
Similarly, for every relevant vertex $v$, $\beta\leq d\alpha_v+\uu^{7/8}$ for every relevant vertex $v$ with high probability by Lemma \ref{lem:fine_balls_disjoint}. In this region, $\lambda^2=\alpha_v+\beta_v/\alpha_v+{(d^2+d)}/{\alpha_v}+O((d^{3/2}+1)\uu^{-4/3})$. Moreover, with high probability, neighborhoods of size $r$ around such vertices $v$ are disjoint and treelike by Lemma \ref{lem:fine_balls_disjoint}. 
Therefore, define $\mathcal{A}(x,y)$ to be the event that for the ordered pair $(x,y)\in \N^2$, $\uu-2\log^{1/8}N\leq  x\leq \uu$ and $0\leq y\leq dx+\uu^{7/8}$, and let $\cB$ be the event that the neighborhoods around all such vertices $v \in [N]$ are disjoint tree-like. If we define $\Lambda(\alpha,\beta)=\alpha+\beta/\alpha$, then with high probability $\Phi$ is contained in the intensity measure of $(\Lambda(\alpha_v,\beta_v)+\epsilon(d,N))\vone_{\mathcal{A}(\alpha_v,\beta_v) \cap \cB}$, where $\epsilon=O(\frac{d^{3/2}+1}{\uu^{4/3}})$. 
Moreover, $\Psi$ is contained in the intensity measure of $\Lambda(X,Y_{X})\vone_{\mathcal{A}(X,Y_X)}$, where $X\sim Pois(d)$ and $Y_X\sim Pois(dX)$.

It is then sufficient to show that for $X\sim Pois(d)$,
\begin{equation}\label{eq:tvar}
d_{TV}\left((\alpha_v,\beta_v)\vone_{\mathcal{A}(\alpha_v,\beta_v) \cap \cB},(X,Y_X)\vone_{\mathcal{A}(X,Y_X)}\right)=o_N(1).
\end{equation}
Define the two dimensional point process ${\Phi'}$ given by $(\alpha_v,\beta_v)\vone_{\mathcal{A}(\alpha_v,\beta_v) \cap \cB}$, and the Poisson process $\Psi'$ as induced by $(X,Y_X)\vone_{\mathcal{A}(X,Y_X)}$. We consider all potential ordered pairs $(x,y)$ such that $\uu-2\log^{1/8}N\leq x\leq \uu$ and $0\leq y\leq dx+\uu^{7/8} $. Therefore, if $n_x:=2\log^{1/8}N+1, n_y=d\uu+\uu^{7/8}$+1, there are at most $n:=n_xn_y$ possibilities.

Define $|\Phi'((x,y))|$ to be the number of points the point process $\Phi'$ has at $(x,y)$. 
We consider an event $E$, which for some set $K_E$ of vectors in $\N^{n}$, is defined as follows. We write $\EE(X)$ here to mean that the event $X$ occurs. 
\[
E:=\bigsqcup_{\mathbf{k}\in K_E} 
\EE \left ( 
\left |\Phi'((x_{1},y_1)) \right |=k_1,\ldots, 
\left |\Phi'((x_{n_x},y_{n_y})) \right | = k_n \right ).
\]
By Lemma \ref{lem:smalltoall},
\begin{multline}
\Pr \left ( |{\Phi'}((x_1,y_1))|=k_1,\ldots, |\Phi'((x_{n_x},y_{n_y}))|=k_n \right )
=\frac1{k_1!\cdots k_{n}!}\sum_{\substack{\ell_1,\ldots,\ell_{n}\\ \sum l_i \leq e^{\log^{2/3}N}/2-1}}
\frac{(-1)^{\sum_{i}\ell_i}}{\ell_1!\cdots \ell_{n}!}
q_{\Phi'} \left ( I_1^{k_1+\ell_1}\times \cdots \times I_{n}^{k_n + \ell_{n}} \right )\\
+O\left(\frac{1}{k_1!\cdots k_{n}!}\sum_{\substack{\ell_1,\ldots,\ell_{n}\\
\sum \ell_i =e^{\log^{2/3}N}/2}}
q_{\Phi'} \left ( I_1^{k_1+\ell_1}\times \cdots \times I_{n}^{k_n+\ell_{n}} \right ) \right )
\end{multline}
where $I_i$ is the lattice point $(x_i,y_i)$.

We use this threshold for $\sum k_i,$ as by Lemma \ref{lem:sizes} and Markov's inequality, with probability $1-e^{-\Omega(\log^{2/3}N)}$, there are at most $e^{\log^{2/3}N}/2-1$ vertices with degree larger than $\uu - 2 \log^{1/8}N$, which implies that we only need to consider vectors such that $k:=\sum_{i=1}^{n} k_i\leq e^{\log^{2/3}N}/2-1$.

By Lemma \ref{lem:totalvariation},
\begin{multline*}
\sum_{\mathbf{k}\in K_{E}}\frac1{k_1!\cdots k_{n}!}\sum_{\substack{\ell_1,\ldots,\ell_{n}\\ \sum l_i \leq e^{\log^{2/3}N}/2-1}}
\frac{(-1)^{\sum_{i}\ell_i}}{\ell_1!\cdots \ell_{n}!}q_{\Phi'}(I_1^{k_1+\ell_1}\times \cdots \times I_{n}^{k_n + \ell_{n}})\\=(1+N^{-1+o_N(1)})
\sum_{\mathbf{k}\in K_{E}}\frac1{k_1!\cdots k_{n}!}\sum_{\substack{\ell_1,\ldots,\ell_{n}\\ \sum l_i \leq e^{\log^{2/3}N}/2-1}}
N^{k+\ell} \frac{(-1)^{\sum_{i}\ell_i}}{\ell_1!\cdots \ell_{n}!} \prod_{i=1}^n\left(\P(Pois(d)=x_i) \P(Pois(dx_i)=y_i)\right)^{k_i + \ell_i}.
\end{multline*}
Also,
\begin{multline*}
\frac{1}{k_1!\cdots k_{n} !}\sum_{\substack{\ell_1,\ldots,\ell_{n}\\
\sum \ell_i =e^{\log^{2/3}N}/2}}
q_{\Phi'}(I_1^{k_1+\ell_1}\times \cdots \times I_{n}^{k_n + \ell_{n}})\\
=O\left(\frac{1}{k_1!\cdots k_{n}!}\sum_{\substack{\ell_1,\ldots,\ell_{n}\\
\sum \ell_i =e^{\log^{2/3}N}/2}}
\frac{1}{(\ell/n)!^n}N^{k+\ell}\pr \left ( \Pois(d)=\uu-2\log^{1/8}N  \right ) ^{k+\ell} \right).
\end{multline*}
where $\ell=\sum_{i} \ell_i=e^{\log^{2/3}N}/2$.
By the definition of the Poisson,
\[
\frac{1}{k_1!\cdots k_{n}!}\sum_{\substack{\ell_1,\ldots,\ell_{n}\\
\sum \ell_i =e^{\log^{2/3}N}/2}}
\frac{1}{((\ell/n)!)^n}N^{k+\ell}\pr \left (Pois(d)=\uu-2\log^{1/8}N \right )^{k+\ell} =e^{-\Omega(e^{\log^{2/3}N})}.
\]

By using Lemma \ref{lem:smalltoall} once again,
\begin{multline*}
\frac1{k_1!\cdots k_{n}!}\sum_{\sum \ell_1\ldots,\ell_{n}\leq e^{\log^{2/3}N}/2-1}N^{k+\ell} \frac{(-1)^{\sum_{i}\ell_i}}{\ell_1!\cdots \ell_{n}!}\prod_{i=1}^n\left(\P(Pois(d)=x_i) \P(Pois(dx_i)=y_i)\right)^{k_i+\ell_i}\\
=(1+N^{-1+o_N(1)})\Pr(\Psi'(I_1)=k_1+\ell_1,\ldots, \Psi'(I_n)=k_n+\ell_n,\Psi'(I_{n})=k_{n})+e^{-\Omega(e^{\log^{2/3}N})}.
\end{multline*}
We now wish to pass from this error to total variation distance. The total number of possibilities of $k$ for $\sum_{i=1}^n k_i\leq e^{\log^{2/3}N}/2$ is given by the balls and bins paradigm as $\sum_{k=0}^{e^{\log^\frac{2}{3} N}/2} \binom{n + k - 1}{k - 1} \leq  e^{\log^{3}N}$. Therefore, the error for any event is at most 
\begin{equation} \label{eq:2dPoissonProcess}
d_{TV}\left((\alpha_v,\beta_v)\vone_{\mathcal{A}(\alpha_v,\beta_v) \cap \cB},(X,Y_X)\vone_{\mathcal{A}(X,Y_X) } \right)=e^{\log^3 N}e^{-\Omega(e^{\log^{2/3}N})}+N^{-1+o_N(1)}=o_N(1).
\end{equation}
\end{proof}

Now we can simply work with independent Poissons, for which the distribution of the maximum is easier to analyze. We start by determining an interval into which the maximizers of the $Y_i$, which approximate the $\beta_x$, fall.

\begin{lemma} \label{lem:largest-k-sphere}
    Consider any function $\zeta(N)=\omega_N(1)$ and $1\leq K=\zeta^{o_N(1)}$.   For fixed $m>0$, $\aa=\Theta(\uu)$, and i.i.d. $Y_1, \dots, Y_\zeta\sim Pois(d\aa)$,  with probability $1-O_N(\frac{1}{\sqrt{\log \zeta}})$, the $K$ largest values $Y_{(1)},\ldots,Y_{(K)}$ are such that for every $1\leq i\leq K$,
    $$ Y_{(i)}\in \left [ 
    d\aa+  \sqrt{ 2 d\aa\log \zeta} - \sqrt{d\aa}\frac{\log K+\frac12\log2-\log \cpois+\frac32\log \log \zeta }{\sqrt{2\log \zeta}}, d\aa+ \sqrt{2 d\aa\log \zeta}
    \right ]
    $$
    where $\cpois$ is the constant from Lemma \ref{lem:sharp_poisson_tail}.
\end{lemma}
\begin{proof}
If $Y_{(K)}$ is less than some value $T$, then there are at least $\zeta-K+1$ $Y_i$'s less than $T$. Therefore,
\[
\P(Y_{(K)} \geq T) \geq 1 - \binom{\zeta}{K-1}( 1 - \pr(Y_1 \geq T))^{\zeta-K+1} \geq 1 - \binom{\zeta}{K-1}e^{-(\zeta-K+1) \P(Y_1 \geq T)}\geq 1-e^{(K-1)\log \zeta-(\zeta-K+1) \P(Y_1 \geq T)}.
\]

    To bound $\P(Y_1 \geq T)$ for $T = d\aa +\sqrt{ 2 d \aa \log \zeta} - \frac{\sqrt{ d \aa}(\log K+\frac12\log2-\log \cpois+\frac32\log \log \zeta )}{\sqrt{2\log \zeta}}$, we use the tail bound from Corollary \ref{eq:poistailex}, with
  \[\delta := \frac{1}{\sqrt{d\aa}} \left ( \sqrt{ 2 \log \zeta} - \frac{\log K+\frac12\log2-\log \cpois+\frac32\log \log \zeta }{\sqrt{2\log \zeta}}\right ).\] As
\[d\aa \delta^2/2 = \log \zeta- \log K+\frac12\log2+\log \cpois-\frac32\log \log \zeta + \frac{(\log K+\frac12\log2-\log \cpois+\frac32\log \log \zeta)^2 }{4\log \zeta},\]
Corollary \ref{eq:poistailex} gives that
    \begin{align*}
        &\P \left ( Y_1 \geq d\aa+ \sqrt{d\aa} \left ( \sqrt{ 2 \log \zeta} - \frac{\log K+\frac12\log2-\log \cpois+\frac32\log \log \zeta }{\sqrt{2\log \zeta}}\right )\right)\\
        &\geq \frac{\cpois}{\zeta} \frac{e^{ \log K+\frac12\log2-\log \cpois+\frac32\log \log \zeta - \frac{(\log K+\frac12\log2-\log \cpois+\frac32\log \log \zeta)^2 }{4\log \zeta} }}{  \sqrt{ 2 \log \zeta} - \frac{\log K+\frac12\log2-\log \cpois+\frac32\log \log \zeta }{\sqrt{2\log \zeta}} }\\
        &=(1+o_N(1))\frac{K\log \zeta}{\zeta}.
    \end{align*}
    Thus
  \begin{equation}
       1- e^{(K-1)\log \zeta - (\zeta-K+1) \P( Y_1 \geq T)} = 1-\zeta^{-1+o_N(1)}.
  \end{equation}

    To prove the upper bound we proceed similarly, using that by a union bound
    \begin{equation}\label{eq:maxupper}
    \P(Y_{(1)} \geq T ) \leq  \zeta\P( Y_1 \geq T).
    \end{equation}
    Using once more the tail bound from Corollary \ref{eq:poistailex} with $\delta = \sqrt{\frac{2 \log \zeta}{d\aa}}$, we obtain 
    \begin{align*}
        \P(Y_1 \geq T) \leq \frac{e^{-\log \zeta}}{\sqrt{2 \log \zeta}},
    \end{align*}
    which means that \eqref{eq:maxupper} can be upper bounded by
    $\frac{1}{\sqrt{2\log \zeta}}.$
    
\end{proof}

The following lemma will imply spacing between eigenvalues.

\begin{lemma}\label{lem:eigenvaluespacing}
Fix $\aa\in \{\uu-1,\uu\}$. For any $K=\log^{o(1)}N$, with high probability the maximum $K+1$ values of $\beta^{(1)}_x$ of vertices with degree $\aa$ are separated by at least $\frac{(d\uu)^{1/2}}{\log(\frac{\uu}{d})^3(K+1)^{3\log\log \log N}}$. 
\end{lemma}
\begin{proof}
Denote by $\zeta$ the number of vertices of degree $\aa$. We will split into two cases, based on whether $\zeta$ is small relative to $K$. As we will see, if $\zeta$ is small, then we can bound the probability using the anticoncentration of the Poisson. If $\zeta$ is larger, we can shift our focus to the regime of Lemma \ref{lem:largest-k-sphere}. It is sufficient to split our cases according to $(K+1)^{\log\log\log(N)}$.
Consider two vertices $u,v$ such that $\alpha_u=\alpha_v=\aa$. 
If $\zeta \leq (K+1)^{\log\log\log(N)}$, then \eqref{eq:2dPoissonProcess} in the proof of Theorem \ref{thm:process} implies that the distribution of the $\beta_x$'s approaches the distribution of Poissons, so the probability that $|\beta^{(1)}_u-\beta^{(1)}_v|\leq \eta$ is at most $\frac{2\eta}{\sqrt{d\aa}}+\tilde O(N^{-1/2})$, considering the mode of a Poisson is at its mean, with probability at most $\frac{1}{\sqrt{d\aa}}$. Therefore the probability that any pair is within distance $\eta$ is at most
\[
\binom{\zeta}2 \frac{2\eta}{\sqrt{d\aa}}\leq \frac{\eta(K+1)^{2\log\log \log N}}{\sqrt{d\aa}}
\]
This converges to 0 for $\eta=\frac{(d\uu)^{1/2}}{(K+1)^{3\log\log \log N}}$. 

Otherwise, if $\zeta \geq (K+1)^{\log\log\log(N)}$, referring once more to \eqref{eq:2dPoissonProcess} in the proof of Theorem \ref{thm:process} and Lemma \ref{lem:largest-k-sphere}, with high probability the $K$ maximizers $x$ of $\beta^{(1)}_x$ satisfy 
\begin{equation}\label{eq:window}
\beta^{(1)}_x
\in 
\left [
d\aa+\sqrt{ 2 d\aa\log \zeta} - \sqrt{d\aa}\frac{\log K+\frac12\log2-\log \cpois+\frac32\log \log \zeta }{\sqrt{2\log \zeta}},d\aa+\sqrt{2(\log\zeta) d\aa}
\right ].
\end{equation}
To show the improvement in density, we consider the probability that $\beta^{(1)}_x=d\aa+t$, for $|t| = (1+o(_N1)) \sqrt{2(\log \zeta)d\aa}.$ We then have by the Stirling approximation,
\[
\frac{e^{-d\aa}(d\aa)^{d\aa+t}}{(d\aa+t)!}= (1+o_N(1))\frac{e^{t}}{(1+\frac{t}{d\aa})^{d\aa+t}\sqrt{2\pi( d\aa+t)}}.
\]
To approximate this, we have
\begin{eqnarray*}
\left ( 1+\frac{t}{d\aa} \right )^{d\aa+t}
&=& e^{\log \left ( 1 + \frac{t}{d \aa} \right ) (d \aa + t )} 
= e^{\left ( \frac{t}{d \aa} - \frac{t^2}{2(d\aa)^2}  + O \left ( \frac{t^3}{(d\aa)^3} \right ) \right ) (d \aa + t)} 
= e^{t + \frac{t^2}{2d \aa} + O \left ( \frac{ t^3}{(da)^2} \right )}
\end{eqnarray*}
In our window, $t\geq \sqrt{2(\log\zeta)d\aa}-\sqrt{d\aa}\frac{\log K+\frac12\log2-\log \cpois+\frac32\log \log \zeta }{\sqrt{2\log \zeta}}$, and we have
\[
\frac{e^{-d\aa}(d\aa)^{d\aa+t}}{(d\aa+t)!}= \frac{1}{e^{(1+o_N(1))t^2/(2d\aa))}\sqrt{2\pi( d\aa+t)}}\leq \frac{e^{\log K+\frac12\log2-\log \cpois+\frac32\log \log \zeta}}{\zeta^{1-O(\sqrt{\frac{\log \zeta}{d\aa}})} \sqrt{2\pi( d\aa + t)}}.
\]

Here we must have $\sqrt{\frac{\log\zeta}{d\uu}}\rightarrow 0$, therefore, we use the assumption that $d\gg \frac{(\log\log N)^2}{\log N}$. The probability that there are at least two vertices in a window of length $2\eta$ around some $d \aa + t$ with $t = (1+o_N(1)) \sqrt{2 da \log \zeta}$ is therefore \begin{equation}\label{eq:windowprob}
\binom{\zeta}{2}\left (\frac{2\eta e^{\log K+\frac12\log2-\log \cpois+\frac32\log \log \zeta}}{ \zeta^{1-O(\sqrt{\frac{\log \zeta}{d\aa}})}\sqrt{2\pi(d\aa+t)}}\right )^2\leq4\cpois^{-1} K^2 \eta^2 \log(\frac{\uu}{d})^3(d\uu)^{-1}
\end{equation}
for sufficiently large $N$, 
considering that with high probability $\zeta\leq (\frac{\uu}{d})^{3/2}$ by Lemma \ref{lem:sizes}.

To translate this into distance between $\beta^{(1)}_u$ and $\beta^{(1)}_v$, we cover the large interval corresponding to \eqref{eq:window} with small intervals of length $2\eta$, and centers spaced at distance $\eta$. To cover this large interval, we need at most $\eta^{-1}\sqrt{d\uu}$ small intervals. Therefore, union bounding the probability \eqref{eq:windowprob} gives that
\begin{align*}
&\pr\bigg(\exists u,v\in [N]:\alpha_u=\alpha_v=\aa,|\beta^{(1)}_u-\beta^{(1)}_v|\leq \eta\bigg)\vone\left(\zeta\geq (K+1)^{\log\log\log N}\right)\\
&\leq4\cpois^{-1} K^2 \eta^2 \log(\frac{\uu}{d})^3(d\uu)^{-1} \eta^{-1}\sqrt{d\uu}\\
&\leq5\cpois^{-1} K^2 \eta \log(\frac{\uu}{d})^3(d\uu)^{-1/2}
\end{align*}
This probability converges to 0 for $ \eta=\frac{(d\uu)^{1/2}}{\log(\frac{\uu}{d})^3(K+1)^{3\log\log \log N}}$.
\end{proof}

\begin{lemma}\label{lem:lambdaubound}
With high probability, for $u\neq v\in \WW$ corresponding to the largest $K+1$ $\lambda$'s, we have $|\lambda_u-\lambda_v|\geq \frac{d^{1/2}}{3\uu\log(\frac{\uu}{d})^3(K+1)^{3\log\log \log N}}$.
\end{lemma}
\begin{proof}
   By Lemma \ref{lem:difference}, this is immediately true if $\alpha_u \neq \alpha_v$. Therefore assume $\alpha_u=\alpha_v$. By Lemma \ref{lem:eigenvaluespacing}, with high probability the $K+1$ maximizers of $\beta^{(1)}_u$ are spaced at distance $\frac{(d\uu)^{1/2}}{\log(\frac{\uu}{d})^3(K+1)^{3\log\log \log N}}$. Therefore, by Lemma \ref{lem:difference}, for $u\neq v$ and sufficiently large $N$, and as $d\gg \log^{-5/3} N$,
   \begin{eqnarray}
   \nonumber|\lambda_u-\lambda_v|&=&\frac{|\lambda^2_u-\lambda^2_v|}{|\lambda_u+\lambda_v|}\\
   \nonumber&\geq&\left(\frac{\frac{(d\uu)^{1/2}}{\log(\frac{\uu}{d})^3(K+1)^{3\log\log \log N}}}{\uu} +O((1+d^{3/2})\uu^{-4/3})\right)\frac{1}{2\sqrt{\uu}+O(\frac d{\sqrt{\uu}})}\\
   \label{eq:eigenvaluespacing}
   &\geq&\frac{d^{1/2}}{3\uu\log(\frac{\uu}{d})^3(K+1)^{3\log\log \log N}}
   \end{eqnarray}
   by the lower bound on $d$.
\end{proof}

%% file: 8.Eigenvector.tex
\section{Eigenvector Structure}\label{sec:eigenvector}
\begin{prop}\label{prop:vecstructure}

For $k\leq K=\log^{o_N(1)}N$, define $x$ to be the vertex corresponding to the $k$th largest eigenvalue of $A$, as per Theorem \ref{thm:maineigenvalue}. The eigenvector $\vv$ of $\lambda$ satisfies 

    \[
    \|\vv-\ww_+(x)\|=O \left ( \uu^{-r/2+2} \right ).
    \]
\end{prop}
\begin{proof}
By Theorem \ref{thm:maineigenvalue}, there is a correspondence between the top $K+1$ eigenvalues and eigenvectors of the matrix $A$ and the top $K+1$ eigenvalues $\lambda_x$ of the truncated balls around vertices $x \in \cW$ together with their eigenvectors. Moreover, by Lemma \ref{lem:lambdaubound}, the difference between each pair of these $K+1$ eigenvalues is at least $\frac{d^{1/2}}{4\uu\log(\frac\uu d)^3(K+1)^{3\log\log\log N}}$. Standard perturbation theory (see \cite{greenbaum2020first} Theorem 2 and the remarks following it) gives that, if we fix the index $1\leq i\leq K$,

\begin{eqnarray*}
\|\vv-\ww_{+}(x)\|&\leq& \|E_W\|\cdot \Big (\min_{j\neq i}|\lambda-\lambda_j| \Big )^{-1}\\
&\leq& O \left ( \frac{\uu\log(\frac{\uu}{d})^3(K+1)^{3\log\log \log N}}{d^{1/2}}(d^{r}+1)\uu^{-r/2+1/2} \right )\\
&=&O \left (\uu^{-r/2+2} \right ).
\end{eqnarray*}
by our assumptions on $d$ from Definition \ref{dfn:rdfn} and $K$, and the bound on $\| E_W \| $ from Theorem \ref{thm:structure}. 
\end{proof}

\begin{proof}[Proof of Theorem \ref{thm:maineigenvector}]
By Proposition \ref{prop:expdecay}, Proposition \ref{prop:vecstructure}, and the triangle inequality, \begin{eqnarray*}
\vv|_{x}&=&\frac1{\sqrt 2}+O \left ( \left (1+d^{-1/2} \right )\uu^{-1/3}+\uu^{-r/2+2} \right )\\
&=&\frac1{\sqrt 2}+O \left ( \left (1+d^{-1/2} \right )\uu^{-1/3} \right )
.\end{eqnarray*}
as $r\geq 5$.
Moreover, as $r\geq 2r'$, for $1 \leq i \leq r',$
\begin{eqnarray*}
\left \|\vv|_{S_i(x)}\right \| &= &\left(\frac{d}{\alpha}\right)^{(i-1)/2}\frac1{\sqrt 2}
\left (1+O \left ( \left (1+d^{-1/2}+d^{-(i-1)} \right )\uu^{-1/3} \right) \right )+O \left (\uu^{-r/2+2} \right )\\
&=&\left(\frac{d}{\alpha}\right)^{(i-1)/2}\frac1{\sqrt 2}\left (1+O \left ( \left (1+d^{-1/2}+d^{-i+1} \right )\uu^{-1/3} \right) \right)
\end{eqnarray*}
as desired.

Similarly, 
\begin{eqnarray*}
  \left \|\ww_+|_{[N]\backslash B_{i}(x)} \right \|  = 
\left(\frac{d}{\alpha}\right)^{i/2} 
\left (1+O \left ( \left ( 1+d^{-1/2}+d^{-r+1} \right )\uu^{-1/3} \right) \right ).  
\end{eqnarray*}

\end{proof}

%% file: 9.appendix.tex
\section{Estimates}\label{sec:estimates}

\subsection{Binomial estimates}
Although we will mostly approximate the degrees by Poisson random variables it is sometimes more convenient to work with the precise distribution. To this end we state here a classical tail bound we will use.

We start by reproducing a classic tail bound for Binomial random variables.

\begin{lemma}[Lemma 4.7.2 in \cite{ash65}] \label{lem:ash-binomial}
    Let $X \sim \Binom( n, p )$, and define $I_p(q) := q \log \frac{q}{p} + (1-q) \log \frac{1-q}{1-p}$. 
    Then for $k > np$
    $$\P( X \geq k ) \leq e^{-n I_p \left ( \frac{k}{n} \right )},$$
    where $I_p(q) = q \log \frac{q}{p} + (1-q) \log \frac{1-q}{1-p}$.
\end{lemma}
By considering $n-X$, the bound above implies that similarly for $ k < np $, $\P( X \leq k ) \leq e^{-n I_p \left ( \frac{k}{n} \right )}.$

As we will be interested in vertices of large degree in our graph, and the degrees follow a binomial distribution, we will repeatedly use tail bounds such as the following.
\begin{lemma} \label{lem:binom_rel_ent}
Let $ m = n + o(n) $ and $p = \frac{d}{N}$, and define $X \sim \Bin(m,p)$ then, for $\tau = o(n),$ it holds for some constant $c$, that for $n$ large enough,
\begin{equation}
    \P \left ( X \geq \tau \right ) \leq e^{-\tau \log (\tau) + c \tau }.
\end{equation}
\end{lemma}
\begin{proof}
By Lemma \ref{lem:ash-binomial},
\begin{equation}
    \P \left ( X \geq \tau \right ) \leq e^{- m I_p(\frac{\tau}{m})}.
\end{equation}
Using that $\log(x) \geq \frac{x}{1+x}$ for $x > -1$, and that $\frac{d-\tau}{n-\tau} > - \frac{1}{2}$ for $n$ large enough, we get that 
\begin{align*}
    I_p \left (\frac{\tau}{m} \right ) 
    & = \frac{\tau}{m} \log \left ( \frac{\tau}{d } \frac{n}{m} \right ) 
    + \left ( 1 - \frac{\tau}{m} \right ) \log \left ( \frac{ n - \tau }{ n - d} \frac{m-\tau}{n-\tau} \frac{n}{m} \right ) \\
    & \geq \frac{\tau}{m} \log (\tau) - \frac{\tau}{m} \log (d) - 2 \left ( 1 - \frac{\tau}{m} \right ) \frac{d-\tau}{n-\tau} + \left ( 1 - \frac{\tau}{m} \right ) \log \left ( \frac{m-\tau}{n-\tau} \right ) + \log \left ( \frac{n}{m} \right )
\end{align*}
Now note that, after multiplication by $m$, all but the first term are $O\left ( \tau \right )$, which implies the bound.
\end{proof}

\subsection{Poisson Approximation}

\begin{proof}[Proof of Lemma \ref{lem:binomtopois}] We simplify using Sterling's approximation and the fact that $\frac{e^c}{(1+\frac cn)^n}=1+O(\frac{c^2}{n})$ for $|c| < n$:
\begin{eqnarray*}
\P(X=k)&=&\binom{n}{k} p^k (1-p)^{n-k}\\
&=& \left (1+O \left (\frac1n \right ) \right )\frac{1}{k!}\frac{n^n}{(n-k)^{n-k}e^{k}}\sqrt{\frac{1}{1-\frac{k}{n}}}p^k(1-p)^{n-k}\\
&=& \left (1+O \left (\frac{k^2+(np)^2+1}n \right ) \right )\frac{e^{-np}(np)^k}{k!}\\
&=&\left (1+O \left (\frac{k^2+(np)^2+1}n \right ) \right )\P(Y=k)
\end{eqnarray*}

\end{proof}
\begin{proof}[Proof of Lemma \ref{cor:bintopoistail}]
We have
\begin{eqnarray*}
\pr(X\geq k)&=& \left (1+O \left (\frac{k^2+(np)^2+1}{n} \right ) \right )\pr(Y\geq k)+O\left(\P \left (X\geq \sqrt n \right )+\P \left (Y\geq \sqrt n \right )\right)
\end{eqnarray*}
and the latter term satisfies
\begin{eqnarray*}
O\left(\P \left (X\geq \sqrt n \right )+\P \left (Y\geq \sqrt n \right )\right) 
= O\left(\binom{n}{\sqrt{n}}p^{\sqrt n}+\frac{(np)^{\sqrt n}}{\sqrt n!}\right) 
= O\left(\left( ep \sqrt{n} \right)^{\sqrt n}\right).
\end{eqnarray*}
\end{proof}

\begin{proof}[Proof of Lemma \ref{lem:sharp_poisson_tail}]

 \begin{align*}
        \P  \big ( X \geq \lambda ( 1 + \delta ) \big ) & = \sum_{ k=\lambda ( 1 + \delta )   }^{\infty}
            \P ( X = k ) \\
        &=e^{ -\lambda }
            \sum_{ k=\lambda ( 1 + \delta )   }^{\infty}
            \frac{ \lambda^k }{ k! }.
            \end{align*}
Note that for $k$ in our range, $\lambda/k\leq 1/(1+\delta)$. Therefore, we can upper bound this with a geometric series.
\begin{eqnarray*}
e^{ -\lambda }
            \sum_{ k=0  }^{\infty}
            \frac{ \lambda^k }{ k! }&\leq &e^{ -\lambda }\frac{\lambda^{\lambda ( 1 + \delta ) }}{(\lambda ( 1 + \delta ) )!}
            \sum_{ k=0  }^{\infty}
            \frac{1 }{ (1+\delta)^k }\\
            &\leq&e^{ -\lambda }\frac{\lambda^{\lambda ( 1 + \delta ) }}{(\lambda ( 1 + \delta ) )!}
            (\frac{1+\delta}{\delta}).
\end{eqnarray*}

Using a Stirling approximation, we obtain
\begin{eqnarray*}
    e^{ -\lambda }\frac{\lambda^{\lambda ( 1 + \delta ) }}{(\lambda ( 1 + \delta ) )!}(\frac{1+\delta}{\delta})
    &=&
    (1+o_\lambda(1))e^{-\lambda}\frac{e^{\lambda(1+\delta)}}{\sqrt{2\pi \lambda(1+\delta)}}\frac{\lambda^{\lambda ( 1 + \delta )}}{(\lambda ( 1 + \delta ) )^{\lambda ( 1 + \delta )}}(\frac{1+\delta}{\delta})\\
    &=&
   (1+o_\lambda(1)) \frac{\exp(-\lambda (1+\delta)\log(1+\delta)+\lambda(1+\delta)-\lambda)}{\sqrt{2\pi \lambda(1+\delta)}\frac{\delta}{1+\delta}}\\
    &\leq&
    \frac{\exp(-\lambda (1+\delta)\log(1+\delta)+\lambda \delta )}{\sqrt{\lambda}\min\{\sqrt{\delta},\delta\}}
\end{eqnarray*}
for sufficiently large $\lambda$. 

\textbf{Lower bound:}
    We once again have that
    \begin{align*}
        \P  \big ( X \geq \lambda ( 1 + \delta ) \big ) & = \sum_{ k=\lambda ( 1 + \delta )   }^{\infty}
            \P ( X = k ) \\
        &=e^{ -\lambda }
            \sum_{ k=\lambda ( 1 + \delta )   }^{\infty}
            \frac{ \lambda^k }{ k! }\\
            &=\frac{(1+o_\lambda(1))e^{ -\lambda }}{\sqrt{2\pi}}\sum_{k\geq \lambda(1+\delta)}\frac{\lambda^ke^k}{k^{k+1/2}}
            \end{align*}
by a Stirling approximation. We write
\[
c:=\lambda(1+\delta),~~~~~~f(x):=-(x+1/2)\log x+x\log \lambda+x.
\]
Therefore, we have
\[
f'(x)=-\log x+\log \lambda +\frac{1}{2x},~~~~~f''(x)=-\frac{1}{x}+\frac{1}{2x^2}
\]
We then approximate our probability by an integral, which serves as a lower bound as our function is decreasing. We then perform a Laplace method type bound. 
\begin{eqnarray*}
\sum_{k= \lambda(1+\delta)}^{{c+c^{1/3}}}\frac{\lambda^ke^k}{k^{k+1/2}}&\geq&\int_{c}^{c+c^{1/3}} \frac{\lambda^{x}e^{x}}{(x)^{{x}+1/2}}dx\\
&=&\int_c^{c+c^{1/3}} \exp\left(f(x)\right)dx\\
&=&e^{f(c)}\int_c^{c+c^{1/3}} \exp\left(f(x)-f(c) \right)dx\\
&=&e^{f(c)}\int_c^{c+c^{1/3}} \exp\left(f'(c)(x-c)   +O_\lambda(f''(c)(x-c)^2 )\right)dx
\end{eqnarray*}
where the last statement follows from Taylor expanding and the formula of $f''(x)$. By the choice of our window, $f''(c)(x-c)^2=O(\lambda^{-1/3})$. Therefore, this is 
\begin{eqnarray*}
e^{f(c)}\int_c^{c+c^{1/3}} \exp\left(f'(c)(x-c)   +O_\lambda(f''(c)(x-c)^2 )\right)dx&=&(1+o_\lambda(1))e^{f(c)-cf'(c)}\int_{c}^{c+c^{1/3}} \exp(f'(c)x) dx\\
&=&-(1+o_\lambda(1))\frac{e^{f(c)-cf'(c)+cf'(c)}}{f'(c)}(1-e^{c^{1/3}f'(c)})\\
&=&-(1+o_\lambda(1))\frac{e^{f(c)}}{f'(c)}.
\end{eqnarray*}
By our choice of $c$, and under the assumption that 
$\delta>\frac{1}{\sqrt{\lambda}}$, we have
\[
-\frac{1}{f'(c)}=(1+o_\lambda(1))\frac{1}{\log (1+\delta)}.
\]
Putting this together with the fact that $\P  \big ( X \geq \lambda ( 1 + \delta ) \big )\geq  \P  \big ( X = \lambda ( 1 + \delta ) \big )$, we have that

\begin{eqnarray*}
\P  \big ( X \geq \lambda ( 1 + \delta ) \big )&\geq&(1+o_\lambda(1))e^{-\lambda}\frac{\lambda^{\lambda(1+\delta)}e^{\lambda(1+\delta)}}{(\lambda(1+\delta))^{\lambda(1+\delta)+1/2}}\cdot \max\left\{\frac{1}{\log (1+\delta)},1\right\}\\
&=&(1+o_\lambda(1))\cpois\frac{\exp(-\lambda h(\delta))}{\sqrt{\lambda}\min\{\sqrt{\delta},\delta\}}
\end{eqnarray*}
for some constant $\cpois$. 
\end{proof}

\begin{proof}[Proof of \ref{lem:weibullbound}]
We will only work with the upper tail as the lower tail is similar. The bound we get from Lemma \ref{lem:weakertail} is 
\[
\Pr \left (X>t+\E(X) \right ) \leq \exp\left(-\frac{t^2}{2d+\frac23t}\right).
\]
Therefore, 
\[
\Pr \left ( X^2\geq t^2+2t\E[X]+\E[X]^2 \right )\leq \exp\left(-\frac{t^2}{2d+\frac23t}\right).
\]
which we rewrite as
\[
\Pr \left ( X^2-(d^2+d)\geq t^2+(2t-1)\E[X] \right )\leq \exp\left(-\frac{t^2}{2d+\frac23t}\right).
\]
By using the substitution $t=\sqrt{y+d^2+d}-d$, we can rewrite this as
\[
\Pr \left ( X^2-(d^2+d)\geq y \right )\leq \exp \left (-\frac{y+2d^2+d-2d\sqrt{y+d^2+d}}{\frac23\sqrt{y+d^2+d}+\frac43d} \right ).
\]

As mentioned, this will follow from \cite{bakhshizadeh2023sharp}, Theorem 1. First, we wish to show that, using the notation in the paper, $v(L,\beta)\rightarrow \Var(X^2)$ for large $L$.  We first show that the strategy of proof of Lemma 4 extends to our scenario. The probability that $X^2-(d^2+d)\geq y$ is at most $e^{-\frac{\sqrt{y}}{12}}$ for $y\geq 4(d^2+d)$. Therefore, we choose our $Y$ to be 
\[
Y=(X^2-\E(X^2))^2\exp( X/12)\vone_{X^2>E(X^2)}.
\]

$Y$ is integrable as the moment generating function $\E(e^{cX})$ is finite for all $c$. Therefore, for sufficiently large $mt$, $v(L,\beta)\approx \Var(X^2)$.

This makes the formulation of Theorem 1 in \cite{bakhshizadeh2023sharp} much simpler, and $t_{\max}$ is such that \[t_{\max}=\Var(X^2)\frac{mt_{\max}+2d^2+d-2d\sqrt{mt_{\max}+d^2+d}}{\frac23 mt_{\max}\sqrt{mt_{\max}+d^2+d}+\frac43d}.\]

Therefore $t_{\max}=(1+o_m(1))(\frac32\Var(X^2))^{2/3}m^{-1/3}$. Using this on the four terms in Theorem 1 from \cite{bakhshizadeh2023sharp}, as well as the bound that $I(t)\geq \sqrt{t}/12$ if $t\geq 4(d^2+d)$, gives the result. 

\end{proof}

\section{Structure surrounding vertices} \label{sec:structure}

\begin{proof}[Proof of Lemma \ref{lem:fine_balls_disjoint}, 1]
We essentially follow the proof of Lemma 7.3 in \cite{alt2021extremal}.
If the balls of radius $r$ around vertices of $\cV$ are not disjoint there must be two vertices $x$ and $y$ in $\mathcal{V}$, connected by a simple path of length at least 1 (if the two vertices are connected by an edge) and at most $2r$. Let us denote the vertices on this path by $z_1, \dots, z_s$, where $s$ can range from 0 to $2r-1$. If we define $z_0 = x$ and $z_{s+1} = y$, then we must have an edge from $z_i$ to $z_{i+1}$ for $i = 0, \dots, s$. Eventually we will take a union bound over all $x, y$ and paths between them, but let us first compute the probability of a single such event.

For convenience of notation let $\tau =\uu - \uu^{2/3}$, the degree lower bound for the intermediate regime. For a given $s$ as well as distinct vertices $x,y, z_1, \dots, z_s$ in $[N]$,
\begin{align*}
    & \P \left (x, y \in \mathcal{V}, (z_i, z_{i+1}) \in E(G) \text{ for } i \in \{0,\dots, s\} \right ) \\
    & \leq \P \left ( \left |\Gamma_x\backslash \{y,z_1\} \right |\geq  \tau - 2, \left |\Gamma_y\backslash \{x,z_s\} \right | \geq \tau - 2, (z_i, z_{i+1}) \in E(G) \text{ for } i \in \{0, \dots, s\} \right ) \\
    & = \P \left ( \left  |\Gamma_x\backslash \{y,z_1\} \right |\geq  \tau - 2 \right ) \P \left ( \left |\Gamma_y\backslash \{x,z_s\} \right |\geq \tau-2 \right ) \left (\frac dN \right )^{s+1} 
\end{align*}
as now all events are independent and the path contains $s+1$ edges.

As $ |\Gamma_x\backslash \{y,z_1\}|$ and $ |\Gamma_y\backslash \{x,z_s\}|$ are distributed as $\Bin(N-2, d/n)$, we can use Lemma \ref{lem:binom_heavy_tail} to get
\begin{align*}
    \P \left ( |\Gamma_x\backslash \{y,z_1\}| \geq \tau-2 \right ) \leq (1+o_N(1)) e^{-\tau\log\tau+\tau\log d + \tau - d }\leq e^{O(\uu^{2/3})-\log N}
\end{align*}
by Lemma \ref{lem:sizes}.

Thus by using a union bound, we can bound the probability that two balls of radius $r$ around two vertices in $\mathcal{V}$ intersect, by considering that we can choose $x,y$ in $\binom{N}{2}$ ways and then we need to choose the path between $x$ and $y$, i.e. for some $s$ between $0$ and $2r$, we need $s$ ordered vertices, for which there are $(N-2)_s$ ways. Combining this with the above bounds gives that the probability of two vertices in $\VV$ being connected by such a path is bounded by
\begin{align*}
    \binom{N}{2} \sum_{s=0}^{2r} (N-2)_s \left(\frac dN\right)^{s+1} e^{-(2+o_N(1))\log N} 
    & \leq 
    N^2 \sum_{s=0}^{2r} N^s \left ( \frac{d}{N} \right )^{s+1}e^{-(1+o_N(1))2\log N} \\
    & \leq 
    f(d,r) e^{-(1+o_N(1))\log N}
\end{align*}
where we bounded $\sum_{s=0}^{2r} d^{s+1} $ by $f(d,r) := d\frac{d^{2r+1}-1}{d-1}$ when $d \neq 1$ and $f(d,r) = 2r+1$ when $d = 1$.
Thus for any constant $r$ the event holds with high probability.
\end{proof}

In order to prove that the balls around vertices in the fine regime are with high probability trees, we start by bounding the probability that the ball around a fixed vertex contains $m$ excess edges, this result and its proof are almost identical to Lemma 5.5 in \cite{alt2021extremal} but we chose to include them for sake of completeness.
\begin{lemma} \label{lem:excess_edges}
    For a vertex $x \in [N]$, any integer $C_1 \geq 1$ and any constant $s$ it holds that 
    \begin{equation*}
        \P \left ( E(B_s(x)) \geq V( B_s(x) -1 + C_1 | S_1(x) \right )
        \leq 
        C d^{3 C_1} \left ( \frac{ |S_1|}{N} \right )^{C_1}, 
    \end{equation*}
    where $C$ is a constant that depends on the constants $s$ and $C_1$.
\end{lemma}

\begin{proof}
    We use the argument from Lemma 5.5 in \cite{alt2021extremal}, but obtain a slightly different bound that is better suited for our regime of $d$.

    Let $T$ be a spanning tree of $B_s(x)$. If $B_s(x)$ contains at least $C_1$ excess edges, there are $C_1$ edges in $B_s(x)$ not contained in $T$, denote those by $E_E$. Let $V_E$ denote the vertices incident to those edges and $E_P$ the edges on the unique paths in $T$ from $x$ to the vertices of $V_E$. Finally let $V_P$ denote the vertices incident to edges in $E_P$. (See Figure 4 in \cite{alt2021extremal} for an illustration.) We define $H$ to be the graph with vertices $V_E \cup V_P$ and edges $E_E \cup E_P$. 

    Let $S^F_r(x)$ denote that sphere of radius $r$ around $x$ in the graph $F$. Then the graph $H$ is a graph on the vertices $[N]$ that satisfies the following properties:
    \begin{enumerate}
        \item $x \in F$,
        \item $S_1^F(x) \subseteq S_1^G(x)$,
        \item $|S_1^F(x)| \geq 1$
        \item $E(F) = V(F) - 1 + C_1$ and
        \item $V(F) \leq 2 C_1 r + 1$.
    \end{enumerate}
    The last property holds since the edges $E_E$ incident to at most $2C_1$ distinct vertices and the paths in $T$ from $x$ to those vertices are of length at most $r$, which implies that $V_P$ contains at most $2 C_1 r$ additional distinct vertices besides $x$.

    Thus we can bound the probability that $E( B_s(x) ) \geq V( B_s(x) ) - 1 + C_1$ by the probability that $B_s(x)$ contains a subgraph $F$ satisfying properties 1.-5. above. 
    For a given $x \in \VV$, we start by conditioning on $S_1(x)$. Then let $\mathbb{F}(x)$ denote the subgraphs satisfying properties 1.-5. Recalling that $G$ denotes our Erd\H{o}s-R\'enyi graph $\cG \left ( N, \frac{d}{N} \right )$, we can bound
    \begin{align*}
        \P \left ( E \left ( B_s(x) \right ) \geq V \left ( B_s(x) \right ) - 1 + C_1 | S_1(x) \right )
        & \leq 
        \P \left ( \cup_{F \in \mathbb{F}(x)} F \subseteq G \big  | S_1(x) \right ) \\
        & \leq
        \sum_{ F \in \mathbb{F}(x) } \P \left ( F \subseteq G \big  | S_1(x) \right ) .
    \end{align*}
    Now note that conditioned on $S_1$ we can construct any graph $F \in \mathbb{F}(x)$ by first choosing $1 \leq s \leq C_1$ vertices from $S_1(x)$, then choosing $0 \leq t \leq 2 C_1 r - s$ (we lose the $+1$ since $x$ is always part of $F$) vertices from $[N] \setminus \cB_1(x)$, and then building a tree with these $s + t + 1$ vertices such that the first $s$ are neighbors of $x$ and the remaining $t$ vertices connect to that graph (but not to $x$), and then adding $C_1$ additional edges. We can bound the number of such graphs by the number of labeled trees on $s+t+1$ vertices (for which Cayley's formula gives that there are $(s+t+1)^{s+t-1}$) times the number of ways of choosing $C_1$ edges, which can be bounded by $(s+t+1)^{2C_1}.$
    The probability that such a graph is contained in $G$ is then equal to $\left ( \frac{d}{N} \right )^{ t + C_1}$ since the number of edges in $F$ without those between $x$ and vertices in $S_1(x)$ is equal to $t + C_1.$
    So continuing from above, we get
    \begin{align*}
        & \leq \sum_{s = 1}^{C_1} \sum_{t = 0}^{ 2 C_1 r - s} \binom{|S_1|}{s} \binom{ N - |S_1(x)| - 1}{t} (s+t+1)^{s+t-1 + 2C_1} \left ( \frac{d}{N} \right )^{ t + C_1} \\
        & \leq \sum_{s = 1}^{C_1} \sum_{t = 0}^{ 2 C_1 r - s} \frac{ |S_1|^s}{s!} \frac{N^t}{t!} (s+t+1)^{s+t + 2C_1 -1} \left ( \frac{d}{N} \right )^{ t + C_1} \\
        & \leq  \frac{1}{N^{C_1}} \left ( d( 2C_1 r + 1 )^2  \right )^{C_1} \sum_{s = 1}^{C_1}  \frac{ |S_1|^s}{s!} (2 C_1 r + 1 )^{s}  \sum_{t = 0}^{ 2 C_1 r - s} \frac{1}{t!} (d(2 C_1 r + 1 ))^{t} \\
        & \leq  
        \frac{1}{N^{C_1}} \left ( d(2 C_1 r + 1 )^2 \right )^{C_1} C_1 \left ( |S_1| (2C_1 r + 1) \right )^{C_1} 2 C_1 r \left ((d(2 C_1 r + 1 ))^2 \right )^{2C_1r} \\
        & \leq 
        \frac{1}{N^{C_1}} 2 r C_1^{ 2 } d^{3 C_1} (2C_1 r + 1)^{5 C_1 }   |S_1| ^{C_1} \\
        & \leq 
        C d^{3 C_1} \left ( \frac{ |S_1|}{N} \right )^{C_1}, \end{align*}
        where $C$ is a constant that depends on the constants $r$ and $C_1$.
\end{proof}

\begin{proof}[Proof of Lemma \ref{lem:fine_balls_disjoint}, 2]
Note that by a union bound over all vertices, we get
\begin{align*}
    & \P \left ( \exists x \in \mathcal{V}: \mathcal{B}_r(x) \text{ is not a tree} \right ) \\
    \leq 
    & \sum_{x \in [N]} \P \left ( \mathcal{B}_r(x) \text{ is not a tree}, x \in \mathcal{V}, \uu - \uu^{\frac{2}{3}} \leq \alpha_x < 2 \frac{\log N}{\log \log N}  \right ) 
    + \P \left ( \alpha_x \geq 2 \frac{\log N}{\log \log N}  \right ) \\
    \leq 
    & \sum_{x \in [N]} \E \left [ \P \left ( \mathcal{B}_r(x) \text{ is not a tree} | S_1(x) \right ) \mathbf{1}\left (\uu - \uu^{\frac{2}{3}} \leq \alpha_x < 2 \frac{\log N}{\log \log N} \right ) \right ] 
    + \P \left ( d_x \geq 2 \frac{\log N}{\log \log N} \right ) \\
    \leq 
    & \sum_{x \in [N]} \E \left [ C d^{3 C_1} \frac{ |S_1| }{ N } \mathbf{1}\left (\uu - \uu^{\frac{2}{3}} \leq \alpha_x < 2 \frac{\log N}{\log \log N} \right ) \right ] 
    + \P \left ( d_x \geq 2 \frac{\log N}{\log \log N} \right ) \\
    \leq 
    & N C d^{3 C_1} \frac{2 \log N}{N} \frac{e^{d+ \u^{\frac{2}{3} \log \uu }}}{N} + N^{-\frac{3}{2}},
\end{align*}
where we apply the bound from Lemma \ref{lem:excess_edges} with $C_1 = 1$, as well as the bound on $|\cV|$ from \ref{lem:sizes} and then used Lemma \ref{lem:binom_heavy_tail} for the second term.
\end{proof}

\begin{proof}[Proof of Lemma \ref{lem:fine_balls_disjoint}, 3]
We show this for $\VV$. The proof for $\WW$ is idendical. For a vertex $x$, and constants $C_i$ that will be set later, let us define the events 
$$\cG_i(x) : = \left \{ \left||S_{i}(x)|-d^{i-1} \alpha_x \right|\leq C_i \left ( d^{i - \frac{3}{2}} + 1 \right ) \uu^{\frac{7}{8}} \right \}$$ and $\cF_i(x) := \bigcap_{j = 1}^i \cG_i(x).$ We will write $\cG_i$ and $\cF_i$ whenever it is clear from context which vertex they refer to. First note that under $\cF_i(x)$, $|B_i(x)| \leq \sqrt{N}.$

Now fix a vertex $x$. 
$\cG_1$ holds trivially by the definition of $\alpha_x$. 

For $i \geq 2$ we now first show that conditional on $S_1$ the probability that $S_i$ is large given that $S_{i-1}$ is small is small. More precisely we show that
\begin{equation}
\P \left ( \cG_i^c \cap \cF_{i-1} | S_1 \right ) 
\leq 
2 \exp \left \{ -  \uu^\frac{3}{4} \right \} 
\end{equation}

To show the above equation first observe that conditioned on $B_{i-1}$, $S_i$ consists of all the neighbors of vertices in $S_{i-1}$ that are not in $B_{i-1}$. Thus, conditionally on $B_{i-1}$,  $|S_i|$ is distributed as $\Binom(|S_{i-1}|(N-|B_{i-1}|),d/N)$. 

Note that this implies that
\begin{equation}
\E[ |S_{i}| | B_{i-1} ] = d |S_{i-1}| - d \frac{|S_{i-1}||B_{i-1}|}{N}.
\end{equation}
Thus under the event $\cF_i$, by Lemma \ref{lem:weakertail}, because $|B_i| \leq \sqrt{N},$
\begin{align*}
    & \P \left ( \left | |S_i| - \E \left [ |S_i| \big | B_{i-1} \right ] \right | \geq \sqrt{d |S_{i-1}|} \uu^\frac{3}{8} + \uu^\frac{7}{8} \Big | B_{i-1} \right ) \\
    & \leq 
    2 \exp 
    \left \{ -
    \frac
    { \left ( \sqrt{d |S_{i-1}|} \uu^\frac{3}{8} + \uu^\frac{7}{8} \right )^2}
    {2 \left (d |S_{i-1}| - d \frac{|S_{i-1}||B_{i-1}|}{N} \right ) + \frac{2}{3} \left ( \sqrt{d |S_{i-1}|} \uu^\frac{3}{8} + \uu^\frac{7}{8} \right ) }
    \right \} \\
    & \leq 
    2 \exp 
    \left \{ -
    \frac
    { \left ( \sqrt{d |S_{i-1}|} \uu^\frac{3}{8} + \uu^\frac{7}{8} \right )^2}
    {2 \left (d |S_{i-1}| - d \right ) + \frac{2}{3} \left ( \sqrt{d |S_{i-1}|} \uu^\frac{3}{8} + \uu^\frac{7}{8} \right ) }
    \right \} \\
    & \leq 2 \exp \left \{ - C \uu^\frac{3}{4} \right \}
\end{align*}
for some constant $C$ that does not depend on $i$.

Now we need to transform the above inequality in the one that we are actually trying to prove. 
For this we need to estimate some quantities:
Let us define $\delta_{i-1} = |S_{i-1} | - d^{i-2} \alpha_x$, note that under $\cF_{i-1}$
\begin{equation*}
    d \delta_{i-1} \leq 2 C_{i-1} (d^{ i - \frac{3}{2} } + 1) u^\frac{7}{8}
\end{equation*}
and generally $ d \leq \left ( d^{ i - \frac{3}{2}} + 1 \right )  \uu^\frac{7}{8}$.

For easier readability set $ \e_i = C_i \left ( d^{ i - \frac{3}{2}} + 1 \right )  \uu^\frac{7}{8} $.
Then 
\begin{align*}
    \left | | S_i | - d^{i-1} \alpha_x \right | \geq \e_i
    & \Rightarrow 
    \big| |S_i| - d|S_{i-1}| \big | \geq \e_i - d \delta_{i-1} \\
    & \Rightarrow 
    \big| |S_i| - \E[ |S_i| | B_{i-1} ] \big | \geq \e_i - d \delta_{i-1} - d \\
    & \Rightarrow \big| |S_i| - \E[ |S_i| | B_{i-1} ] \big | \geq ( C_i - 2 C_{i-1} - 1) \left ( d^{ i - \frac{3}{2}} + 1 \right )  \uu^\frac{7}{8}
\end{align*}

When $\cF_{i-1}$ holds and $ \alpha_x \leq 2 \uu$,
\begin{align*} \label{eq:S_ierrorbound}
    \sqrt{d |S_{i-1}|}  u ^\frac{3}{8} + \uu^\frac{7}{8}
    & \leq \sqrt{d \left ( d^{i-2} \alpha_x + C_{i-1} \left ( d^{i- \frac{5}{2}} + 1 \right ) \uu^\frac{7}{8}   \right ) }u^\frac{3}{8} + \uu^\frac{7}{8} \\
    & \leq \left (\sqrt{ 2 (C_{i-1} + 1)  } + 1 \right ) ( d^{i-\frac{3}{2}} + 1) u^\frac{7}{8}.
\end{align*}

If we set $C_i$ such that $C_i - 2 C_{i-1} -1 \geq \left (\sqrt{ 2 (C_{i-1} + 1)  } + 1 \right )$, then whenever $cF_{i-1}$ holds and $\alpha_x \leq 2 \uu$,
\begin{equation*}
    \P \left ( 
    \left | | S_i | - d^{i-1} \alpha_x \right | 
    \geq C_i ( d^{i - \frac{3}{2}} + 1 ) u^\frac{7}{8} 
    \big |B_{i-1} 
    \right ) 
    \leq 
    \P \left ( 
    \big| |S_i| - \E[ |S_i| | B_{i-1} ] \big |
    \geq \sqrt{d |S_{i-1}|}  u ^\frac{3}{8} + \uu^\frac{7}{8}
    \big | B_{i-1} 
    \right ).
\end{equation*}

Finally we put all of this together in a union bound
\begin{align*}
	& \P \Big ( \exists x \in \cV: \cup_{i = 1}^{r+3} \cG_i^c(x) \Big ) \\
	& \leq 
	N \E \left [ \P \left (  \cup_{i = 1}^{r+3} \cG_i^c(x) \big | S_1(x) \right ) \mathbf{1}(x \in \cV) \right ] \\
	& \leq N \E \left [ \sum_{i = 1}^{r+3} \P \left ( \cG^c_i(x) \cap \cF_{i-1}(x) \big | S_1(x) \right ) \mathbf{1} \left (\uu - \uu^\frac{2}{3} \leq \alpha_x \leq 2 \uu \right ) \right ] + N\P \left ( \alpha_x > 2 \uu \right ) \\
        & \leq N \E \left [ \sum_{i = 1}^{r+3} 
        \E \left [  \P \left ( \big | |S_i| - d^{i-1} \alpha_x \big |  \geq \e_i  \bigg | B_{i-1} \right ) \mathbf{1} \left ( \cF_{i-1} \right ) \Big | S_1 \right ]
        \mathbf{1} \left (\uu - \uu^\frac{2}{3} \leq \alpha_x \leq 2 \uu \right ) \right ] + N\P \left ( \alpha_x > 2 \uu \right ) \\
	& \leq N \E \left [ 
	\sum_{i = 2}^{r+3} 
	2 \exp \left \{ 
	-  C \uu^\frac{3}{4} 
	\right \}  
	\mathbf{1} \left ( \uu - \uu^\frac{2}{3} \leq \alpha_x \leq 2 \uu \right ) 
	\right ] 
	+ N\P \left ( \alpha_x > 2 \uu \right ) \\
	& \leq 
	N (r+3) e^{-C \uu^\frac{3}{4}}\P \left ( \uu - \uu^\frac{2}{3} \leq \alpha_x \right ) + N \P \left ( \alpha_x > 2 \uu \right ) \\
	&\leq 
	N (r+3) e^{-C \uu^\frac{3}{4}} \frac{1}{N} \frac32e^{d+ \uu^{\frac{2}{3}} \log \uu}\sqrt{\frac{\uu}{d}} + N \frac{1}{N^\frac{3}{2}}
\end{align*}
where we used that $ \uu = \Theta \left ( \frac{\log N}{\log \log N} \right )$ by \eqref{eq:uapprox} and then applied Lemma \ref{lem:binom_heavy_tail} for the second term and the bound on $|\cV|$ from Lemma \ref{lem:sizes}for the first term.

\end{proof}

\begin{proof}[Proof of Lemma \ref{lem:fine_balls_disjoint}, 4]
We will prove this for $\VV$, the proof with $\WW$ is identical. First note that by Lemma \ref{lem:weakertail}, for $X \sim \Binom(N, d/N)$,
$\P \left ( X \geq \uu^{\frac{3}{4}} \right ) 
\leq 
e^{-\Omega(\uu^{\frac{3}{4}})},
$
since $d \leq (\log N)^\frac{1}{5}$ by \ref{dfn:rdfn}.
The basic idea now is that by Lemma \ref{lem:fine_balls_disjoint} 3, there are $O((r+3)d^{r+2}\uu)$ vertices in $B_{r+3}(x)$ and by Lemma \ref{lem:sizes} there are $e^{O(\uu^\frac{2}{3}\log \uu)}$ vertices in $\cV$, so union bounding over all those vertices implies the result. Let us now make this precise.

We show this level by level. 
For all $y \in S_i(x)$, conditioned on $B_i$, the $N_y$ are independent and distributed as $\Binom(N - |B_{i}|, d/N),$ which is stochastically dominated by $\Binom(N, d/N)$. Thus the probability that any $N_y \geq u^\frac{3}{4}$ is bounded by $e^{- \Omega(u^\frac{3}{4})}$.

Putting everything together and using the notation $\cF_i$ as defined in the previous proof, we first get
\begin{align*}
    & \P \left ( \exists x \in \cV: \exists y \in B_{r+3}(x): N_y > \uu^\frac{3}{4} \right )
    & \leq \P \left ( \exists x \in \cV: \exists y \in B_{r+3}(x): N_y > \uu^\frac{3}{4},\alpha_x \leq 2 \uu  \right ) 
    + \P \left ( \exists x: \alpha_x > 2 \uu \right )
\end{align*}
and we know by \ref{lem:fine_balls_disjoint}, 3. that the latter event happens with low probability.
The first term on the other hand we can bound by
\begin{align*}
    & \sum_{x \in [N]} 
    \P \left (x \in \cV, \cup_{i = 1}^{r+3} \left \{ \exists y \in S_i(x): N_y > \uu^\frac{3}{4} \right \}, \alpha_x \leq 2 \uu \right ) \\
    & \leq 
    \sum_{x \in [N]}
    \sum_{i = 1}^{r+3} 
    \P \left (  \exists y \in S_i(x): N_y > \uu^\frac{3}{4} , \cF_i(x) , x \in \cV , \alpha_x \leq 2 \uu\right ) 
    + 
    \sum_{x \in [N]}
    \sum_{i = 1}^{r+3} 
    \P \left ( \cF_i^c(x), x \in \cV, \alpha_x \leq 2 \uu \right ) \\
\end{align*}
Now the latter term is small by the previous proof (note that we used a union bound there as well.)
For the former term we proceed as follows:
\begin{align*}
    & \leq 
    \sum_{x \in [N]}
    \sum_{i = 1}^{r+3} 
    \E \left [ \E  \left [ 
    \mathbf{1} \left (  \exists y \in S_i(x): N_y > \uu^\frac{3}{4} \right ) \big | B_{i} \right ] 
    \mathbf{1}_{ \cF_i } 
    \mathbf{1}_{x \in \cV}
    \mathbf{1}_{\alpha_x \leq 2 \uu}
    \right ] \\
    & \leq 
    \sum_{x \in [N]}
    \sum_{i = 1}^{r+3} 
    e^{-\Omega(u^\frac{3}{4})}
    \E \left [ |S_i| 
    \mathbf{1}_{\cF_i(x)} 
    \mathbf{1}_{x \in \cV} 
    \mathbf{1}_{\alpha_x \leq 2 \uu}
    \right ] \\
    & \leq 
    \sum_{x \in [N]}
    e^{-\Omega(u^\frac{3}{4})} 
    \sum_{i=1}^{r+3} 
    \E \left [ 
    \left ( d^{i-1}  \alpha + O( d^{i-\frac{3}{2}} \uu^\frac{7}{8} + \uu^\frac{7}{8} ) 
    \right )
    \mathbf{1}_{\alpha_x \leq 2 \uu}
    \mathbf{1}_{x \in \cV}
    \right ]\\
    & \leq 
    e^{-\Omega(u^\frac{3}{4})} 
    O \left ( (r+3) (1+d^{r+2}) \uu \right ) 
    \frac32e^{d+ \uu^{\frac{2}{3}} \log \uu}\sqrt{\frac{\uu}{d}},
\end{align*}
which is small as $N \to \infty.$

\end{proof}

\begin{proof}[Proof of Lemma \ref{lem:fine_balls_disjoint}, 5]
We prove this for $\WW$, the proof for $\VV$ is equivalent. Here we use Lemma \ref{lem:weibullbound} with $t=2\uu^{2/3}$. The probability bound we obtain is

\[
\exp(-\Omega(\uu^{-1/3}/(d^3+1))).
\]
 Therefore, for our range of $d$, it is possible to union bound over all vertices in $|\cW|$, as this gives us
\[
\exp(\uu^{1/4})\exp(-\Omega(\uu^{-1/3}/(d^3+1)))=\exp(-\Omega(\uu^{-1/3}/(d^3+1)))
\]
by the bound in Lemma \ref{lem:sizes}.

\end{proof}

\begin{proof}[Proof of Lemma \ref{lem:size_balls}]

We show that with high probability, for all vertices in $\cU$,
\begin{equation}
     |S_i|  \leq 4^{i-1} (d + \log \log N - \log d )^{i-1} \uu
\end{equation}
which implies the statement by our bounds on $d$ from Definition \ref{dfn:rdfn}.

The strategy is similar as in the proof of Lemma \ref{lem:fine_balls_disjoint}, 3:
First note that by Lemma \ref{lem:weakertail}, under the event $\cF_{i-1}$,
\begin{align*}
    \E \left [|S_{i-1}| \big | B_{i-1} \right ] = d |S_{i-1} | - | B_{i-1} | |S_{i-1} | \frac{ d }{ N } \leq d |S_{i-1} | - d
\end{align*}
using this bound and the fact that $d \leq d + \log \log N - \log d$, we get that
\begin{align*}
    & \P \left ( \big | |S_i| - \E \left [ |S_i | B_{i-1} \right ] \big | \geq \left ( (d + \log \log N - \log d) \sqrt{2 |S_{i-1}| u} + (d + \log \log N - \log d) \uu \right ) \big | B_{i-1} \right ) \\
    & \leq 
    2 \exp \left \{ - 
    \frac
    { \left ( (d + \log \log N - \log d) \sqrt{2 |S_{i-1}| u} + (d + \log \log N - \log d) \uu \right )^2}
    {2 \E \left [ |S_i| | B_{i-1} \right ] 
    +
    \frac{2}{3} \left ( (d + \log \log N - \log d) \sqrt{|S_{i-1}| u} + (d + \log \log N - \log d) \uu \right ) }
    \right \} \\
    & \leq 2 e^{ - (d+ \log \log N - \log d) \uu } \\
    & \leq 2 N^{-1+o(1)}
\end{align*}
by considering which term in the denominator is smaller and then using the approximation from \eqref{eq:uapprox} for $\uu$.

Now note that under the event $\cF_i$, 
\begin{align*}
    \E \left [|S_{i-1}| \big | B_{i-1} \right ] \leq d 4^{i-2}( d + \log \log N - \log d )^{i-2} \uu \leq 
    4^{i-2} ( d + \log \log N - \log d )^{i-1} \uu
\end{align*}
such that
\begin{align*}
    & |S_i| \geq 4^{i-1} ( d + \log \log N - \log d)^{i-1} \uu \\ 
    \Rightarrow & 
    |S_i| - \E \left [ |S_i| \big | B_{i-1} \right ] 
    \geq 
    3 \cdot 4^{i-2} ( d + \log \log N - \log d)^{i-1} \uu \\
    \Rightarrow & 
    |S_i| - \E \left [ |S_i| \big | B_{i-1} \right ] 
    \geq 
    ( d + \log \log N - d ) \sqrt{ 2 |S_{i-1}| \uu} + (d + \log \log N - \log d ) \uu.
\end{align*}
This implies that 
\begin{align*}
    \P \left ( |S_i| \geq 4^{i-2} ( d + \log \log N - \log d )^{i-1} \uu | B_i \right ) \leq 2 N^{-1 + o(1)}.
\end{align*}
We now proceed as in the end of the proof of Lemma \ref{lem:fine_balls_disjoint}, 3, by using the bound on $\cU$ from Lemma \ref{lem:sizes}.

For the second statement of the Lemma, note that it is sufficient to bound $\sum_{y\sim x}N_y^2$. We once more use Lemma \ref{lem:weibullbound} as in the proof of Lemma \ref{lem:fine_balls_disjoint}, 5, setting $t=\cweib^{-2}\log^2N$.
\end{proof}

\begin{proof}[Proof of Lemma \ref{lem:rough_tree_disjoint}]
    We use the arguments from the proofs of Lemma 5.5 and Lemma 7.3 in \cite{alt2021extremal}, but choose to include them here for sake of completeness.

    For any $x \in [N]$ let $\mathcal{E}_x$ denote the event that there are at least $C_1$ excess edges in $B_s(x)$.
    Then by a union bound
    \begin{align*}
        \P \left ( \exists x \in \cU: \cE_x  \right ) 
        & \leq 
        \sum_{ x \in [N]} \P \left ( \cE_x,  x \in \cU , \alpha_x < 2 \uu \right ) + \P ( \alpha_x \geq 2 \uu ),
    \end{align*}
    where the second summand can be bounded by $N^{-\frac{3}{2}}$ according to Lemma \ref{lem:binom_heavy_tail}.
    
    In order to bound the first term, we want to condition on $S_1(x)$, and then apply Lemma \ref{lem:excess_edges} for this we write
    \begin{align*}
        \P \left ( \cE_x,  x \in \cU , \alpha_x < 2 \uu \right ) 
        & =
        \E \left [ 
        \P ( \cE_x |S_1(x) )
        \mathbf{1}(\{ x \in \cU, \alpha_x < 2 \uu \})
        \right ] \\
        & \leq 
        \E \left [ C d^{3 C_1} \left ( \frac{ |S_1|}{N} \right )^{C_1} \mathbf{1} \left ( \eta \uu \leq \alpha_x \leq 2 \uu \right ) \right ] \\
        & \leq 
        C \left ( 2d^3 \right )^{C_1} \left ( \frac{ \log N}{N} \right )^{C_1} \frac{1}{N^\frac{\eta}{2}}
    \end{align*}

    Thus by taking $C_1 \geq 2$ and then doing a union bound over all $x$ we get the desired result.

    For the second statement of the Lemma, we proceed as in the proof of Lemma 7.3 in \cite{alt2021extremal} and write $\cI_x$, the event that there are at least $C_2$ disjoint paths in $B_s(x)$ ending at vertices in $\cU_\eta$, as a union over the specific paths:
    \begin{align*}
        \cI_x 
        =
        \bigcup_{\mathbf{y}, \mathbf{z}} \Gamma^{(C_2)}_{ \mathbf{y}, \mathbf{z} },
    \end{align*}
    where the union is taken over all vectors $\mathbf{y} = (y_1, \dots, y_{C_2} )$ with distinct entries in $[N] \setminus \{ x \}$ and the $C_2$-tuples $\mathbf{z}$ of disjoint vectors $(z^{(1)}, \dots, z^{(C_2)})$ of length $r_j \in \{ 0, \dots, s \}$ for $j \in [C_2]$, and $$\Gamma_{ \mathbf{y}, \mathbf{z} } = \left \{ y_j \in \cU_\eta, \{x,z_1^{(j)} \}, \{ z_i^{(j)}, z_{i+1}^{(j)} \}, \{ z_{r_j}^{(j)}, y^{j} \} \in E(G) \forall i \in [r_j - 1], j \in [k] \right \}.$$

    For some fixed $\mathbf{y}$ and $\mathbf{z}$, and thus fixed set of $(r_1, \dots, r_{C_2})$, since all paths are disjoint, when we denote by $\cN_x$ the neighborhood of a vertex $x$, and use the independence of the edges, we get that 
    \begin{align*}
        \P \left ( \Gamma_{\mathbf{y}, \mathbf{z}} \right ) 
        \leq &
        \P \left (  \big |\cN_x \cap ([N] \setminus \mathbf{y}) \big | \geq \eta \uu - C_2 \right ) \\
        & \prod_{j=1}^{C_2} \P \left ( \big | \cN_y \cap ([N] \setminus \{ x \} \cup \mathbf{y}) \} \big | \geq \eta \uu - C_2 - 1 \right )  \\
        & \left ( \frac{d}{N} \right )^{\sum_{ j  = 1}^{C_2} r_j + 1}.
    \end{align*}

    We now apply Lemma \ref{lem:sharp_poisson_tail} and Corollary \ref{cor:bintopoistail} to bound the remaining probabilities: since $C_2$ is constant all these probabilities will be bounded by 
    \begin{equation*}
        e^{-\eta u \log u + c \eta u} = e^{- (1+o(1)) \eta \log N}.
    \end{equation*}
    This implies that the above probability is bounded by 
    \begin{equation*}
        \frac{d^{\sum_{ j  = 1}^{C_2} r_j + 1}}{ N^{\left (\sum_{j=1}^{C_2} r_j + 1 \right ) +  \eta C_2 (1+o(1))} }.
    \end{equation*}
    To complete the union bound we need to count the number of terms, i.e. possible paths, for each sequence of $r_j$s. To do this we note that given that $x$ is fixed, there are $\binom{N - 1}{C_2}$ ways of picking $\mathbf{y}$ and for the $z^{(j)}_i$ on each path there are $\binom{N - k - \sum_{i = 1}^{j-1} r_i}{r_j}$ ways of picking them. Thus 
    \begin{align*}
        & \P \left ( \cI_x \right ) \\
        & \leq 
        \binom{N - 1}{C_2} \sum_{r_1 = 0}^s \cdots \sum_{r_{C_2} = 0}^s \binom{N-C_2 - 1 }{r_1} \cdots \binom{N - C_2 - \sum_{i = 1}^{C_2-1 r_i}}{r_{C_2}} \frac{d^{\sum_{ j  = 1}^{C_2} r_j + 1}}{ N^{\left (\sum_{j=1}^{C_2} r_j + 1 \right ) +  \eta C_2 (1+o(1))} } \\
        & \leq C  \frac{ d^{C_2 (s + 1)}}{N^{\eta C_2 (1 + o(1) )} }
    \end{align*}
    where $C$ is a constant that depends on the constants $s$ and $C_2$.

    By taking $C_2 > \frac{2}{\eta},$ and then taking a union bound over all $x$, this implies that all $B_s(x)$ only contain a constant number of disjoint paths ending at other vertices from $\cU$ with high probability.
        
\end{proof}

\begin{proof}[Proof of Lemma \ref{lem:prunedgraph}]
To construct the pruned graph $\hat{G}$ we delete edges in the same manner as in Lemma 7.2 in \cite{alt2021extremal}:
For every vertex $x \in \cU$, and its neighbor $y$, consider the set of vertices $T_y$ that are connected to $y$ by a path of length at most $3$, without traversing the edge $(x,y)$. If $x$ is in this set, or the graph induced by $T_y$ on $G$ is not a tree, then we prune the edge $(x,y)$. Denote the set of edges that are pruned in this way $P_x$. According to Lemma \ref{lem:rough_tree_disjoint}, with high probability, each vertex $x \in \cU$ has less than $C_1$ ``excess'' edges that create cycles in $B_3(x)$. Thus by the above procedure we prune at most $C_1 - 1$ edges that are adjacent to $x$.

In the second step, we work with the graph on $[N]$ with edges $E(G) \setminus P_x$, in which $B_3(x)$ is a tree. In that graph we consider for each neighbor $y$ of $x$, the vertices $V_y$ in $B_3(x)$ that are connected to $y$ by a path that does not use the edge $(x,y)$. If any of the vertices in $V_y$ is in $\cU$, we prune the edge $(x,y)$ and add it to $P_x$. By Lemma \ref{lem:rough_tree_disjoint} we prune at most $C_2 - 1$ edges adjacent to $x$ by doing this procedure.

We then apply these steps, by choosing an arbitrary order of vertices in $\UU$, then pruning edges surrounding these vertices sequentially. Let $H$ be the graph on $[N]$ that only consists of the edges $\cup_{ x \in cU} P_x$ that we pruned. We then define our pruned graph $\hat{G}$ to be the graph $G$ with edges $E(G) \setminus \cup_{ x \in cU} P_x$. By construction $\hat{G}$ satisfies 1. and 2.

 Note that only vertices $x \in \cU$ and vertices $y \in \cup_{x \in \cU} S_1(x)$ are not isolated in $H$. It is clear that at each step of this procedure we prune at most $C_2 + C_1 - 2$ edges adjacent to some $x$. Moreover, note that that any subsequent step cannot affect the degree of $x$ in $H$: otherwise, if we have already pruned for $x \in \cU$, if in a subsequent pruning for $x' \in \cU$ we were to delete an edge adjacent to $x$, this would mean that $(x',x)$ is an edge in $G$, in which case we would already have pruned it when doing the pruning for $x$.
 
Now let $y \in \cup_{x \in \cU} S_1(x) \setminus \cU$, i.e. let $y$ be a vertex that is not in $\cU$ and is a neighbor of some vertex $x \in \cU$. By Lemma \ref{lem:rough_tree_disjoint}, $y$ can be adjacent to at most $C_2 - 1$ additional vertices from $\cU$, since otherwise $B_2(x)$ contains more than $C_2 - 1$ vertices  from $\cU$.  Thus we prune at most $C_2 - 1$ edges adjacent to $y$. Hence the maximal degree of the graph $H$ is $C_1 + C_2 - 2,$ implying 3.

Recall the assumption that the maximum degree is at most $\uu$. Thus for each edge $(x,y)$ that we prune, $\beta_x$ is reduced by at most $\uu$. Additionally for each vertex $ y \in S_1(x)$, we delete at most $C_2 - 1$ edges by doing the pruning procedure for other $x' \in \cU$. This implies that $0 \leq \beta_x - \hat{\beta}_x \leq \alpha_x (C_2 - 1) + (C_1 + C_2 - 2) \uu  = O( \uu ),$ which implies 4.

Lemma \ref{lem:size_balls} gives a bound on the growth of the spheres in the original graph and since $\alpha_x$ and $\hat{\alpha}_x$ are of the same order, 5. follows immediately. 

For the last statement we rewrite 
\begin{equation*}
    \sum_{y \in \hat{S}_1(x)} \left ( \hat { N}_y - \frac{ \hat{\beta}}{\hat{\alpha}} \right )^2
    \leq 3 \left [ 
    \sum_{y \in \hat{S}_1(x)} \left ( \hat { N}_y -  N_y \right )^2 
    + 
    \sum_{y \in \hat{S}_1(x)} \left (  N_y -  d \right )^2 
    +
    \sum_{y \in \hat{S}_1(x)} \left ( d -  \frac{ \hat{\beta}}{\hat{\alpha}} \right )^2
    \right ].
\end{equation*}
The first term can be bounded by $O(\hat{\alpha})$ since $G - \hat G$ has a bounded degree by Lemma \ref{lem:prunedgraph}, 3, for the second term we use Lemma \ref{lem:size_balls} and the last term can be bounded by $\hat{\alpha} \left ( d - \frac{\hat{\beta}}{\hat{\alpha}} \right )^2$ and then be bounded using Lemma \ref{lem:prunedgraph}, 3.
    
\end{proof}

%% file: main.bbl
\begin{thebibliography}{BGGM14}

\bibitem[ABP09]{auffinger2009poisson}
Antonio Auffinger, G{\'e}rard {Ben Arous}, and Sandrine P{\'e}ch{\'e}.
\newblock Poisson convergence for the largest eigenvalues of heavy tailed
  random matrices.
\newblock {\em Ann. Inst. Henri Poincar\'{e} Probab. Stat.}, 45(3):589--610,
  2009.

\bibitem[ADK21a]{alt2021delocalization}
Johannes Alt, Raphael Ducatez, and Antti Knowles.
\newblock Delocalization transition for critical {E}rd{\H{o}}s--{R}{\'e}nyi
  graphs.
\newblock {\em Communications in Mathematical Physics}, 388(1):507--579, 2021.

\bibitem[ADK21b]{alt2021extremal}
Johannes Alt, Raphael Ducatez, and Antti Knowles.
\newblock Extremal eigenvalues of critical {E}rd{\H{o}}s--{R}{\'e}nyi graphs.
\newblock {\em The Annals of Probability}, 49(3):1347--1401, 2021.

\bibitem[ADK22]{alt2022completely}
Johannes Alt, Raphael Ducatez, and Antti Knowles.
\newblock The completely delocalized region of the {E}rd{\H{o}}s--{R}{\'e}nyi
  graph.
\newblock {\em Electronic Communications in Probability}, 27, 2022.

\bibitem[ADK23a]{alt2023localized}
Johannes Alt, Raphael Ducatez, and Antti Knowles.
\newblock Localized phase for the {E}rd{\H{o}}s--{R}\'enyi graph.
\newblock {\em arXiv preprint arXiv:2305.16294}, 2023.

\bibitem[ADK23b]{alt2023poisson}
Johannes Alt, Raphael Ducatez, and Antti Knowles.
\newblock Poisson statistics and localization at the spectral edge of sparse
  {E}rd{\H{o}}s--{R}{\'e}nyi graphs.
\newblock {\em The Annals of Probability}, 51(1):277--358, 2023.

\bibitem[Alo98]{alon1998spectral}
Noga Alon.
\newblock Spectral techniques in graph algorithms.
\newblock In {\em LATIN'98: Theoretical Informatics: Third Latin American
  Symposium Campinas, Brazil, April 20--24, 1998 Proceedings 3}, pages
  206--215. Springer, 1998.

\bibitem[AM93]{aizenman1993localization}
Michael Aizenman and Stanislav Molchanov.
\newblock Localization at large disorder and at extreme energies: An elementary
  derivations.
\newblock {\em Communications in Mathematical Physics}, 157:245--278, 1993.

\bibitem[And58]{anderson1958absence}
Philip~W Anderson.
\newblock Absence of diffusion in certain random lattices.
\newblock {\em Physical review}, 109(5):1492, 1958.

\bibitem[Ash65]{ash65}
Robert~B. Ash.
\newblock {\em Information {Theory}}.
\newblock Number~19 in Interscience tracts in pure and applied mathematics.
  Interscience, New York, 1965.

\bibitem[Bam20]{bamieh2020tutorial}
Bassam Bamieh.
\newblock A tutorial on matrix perturbation theory (using compact matrix
  notation).
\newblock {\em arXiv preprint arXiv:2002.05001}, 2020.

\bibitem[Bau85]{baumgartel1985analytic}
Hellmut Baumg{\"a}rtel.
\newblock Analytic perturbation theory for matrices and operators.
\newblock {\em Operator theory}, 15, 1985.

\bibitem[BBG21]{bhattacharya2021spectral}
Bhaswar~B Bhattacharya, Sohom Bhattacharya, and Shirshendu Ganguly.
\newblock Spectral edge in sparse random graphs: Upper and lower tail large
  deviations.
\newblock {\em The Annals of Probability}, 49(4):1847--1885, 2021.

\bibitem[BG13]{bordenave2013localization}
Charles Bordenave and Alice Guionnet.
\newblock Localization and delocalization of eigenvectors for heavy-tailed
  random matrices.
\newblock {\em Probability Theory and Related Fields}, 157(3-4):885--953, 2013.

\bibitem[BG17]{bordenave2017delocalization}
Charles Bordenave and Alice Guionnet.
\newblock Delocalization at small energy for heavy-tailed random matrices.
\newblock {\em Communications in Mathematical Physics}, 354:115--159, 2017.

\bibitem[BG20]{bhattacharya2020upper}
Bhaswar~B Bhattacharya and Shirshendu Ganguly.
\newblock Upper tails for edge eigenvalues of random graphs.
\newblock {\em SIAM Journal on Discrete Mathematics}, 34(2):1069--1083, 2020.

\bibitem[BGBK19]{benaych2019largest}
Florent Benaych-Georges, Charles Bordenave, and Antti Knowles.
\newblock Largest eigenvalues of sparse inhomogeneous
  {E}rd{\H{o}}s--{R}{\'e}nyi graphs graphs.
\newblock {\em Ann. Probab.}, 47(3):1653--1676, 2019.

\bibitem[BGBK20]{benaych2020spectral}
Florent Benaych-Georges, Charles Bordenave, and Antti Knowles.
\newblock Spectral radii of sparse random matrices.
\newblock {\em Ann. Inst. Henri Poincar\'{e} Probab. Stat.}, 56(3):2141--2161,
  2020.

\bibitem[BGGM14]{benaych2014central}
Florent Benaych-Georges, Alice Guionnet, and Camille Male.
\newblock Central limit theorems for linear statistics of heavy tailed random
  matrices.
\newblock {\em Communications in Mathematical Physics}, 329(2):641--686, 2014.

\bibitem[BMdlP23]{bakhshizadeh2023sharp}
Milad Bakhshizadeh, Arian Maleki, and Victor~H de~la Pena.
\newblock Sharp concentration results for heavy-tailed distributions.
\newblock {\em Information and Inference: A Journal of the IMA}, 12(3):iaad011,
  2023.

\bibitem[Bol01]{bollobas01randomgraphs}
Béla Bollobás.
\newblock {\em Random Graphs}.
\newblock Cambridge Studies in Advanced Mathematics. Cambridge University
  Press, 2 edition, 2001.

\bibitem[Chu97]{chung1997spectral}
Fan~RK Chung.
\newblock {\em Spectral graph theory}, volume~92.
\newblock American Mathematical Soc., 1997.

\bibitem[DS20]{ding2020localization}
Jian Ding and Charles~K Smart.
\newblock Localization near the edge for the anderson bernoulli model on the
  two dimensional lattice.
\newblock {\em Inventiones mathematicae}, 219:467--506, 2020.

\bibitem[EKYY12]{erdHos2012spectral}
L{\'a}szl{\'o} Erd{\H{o}}s, Antti Knowles, Horng-Tzer Yau, and Jun Yin.
\newblock Spectral statistics of {E}rd{\H{o}}s-{R}{\'e}nyi graphs {II}:
  Eigenvalue spacing and the extreme eigenvalues.
\newblock {\em Communications in Mathematical Physics}, 314(3):587--640, 2012.

\bibitem[EKYY13]{erdHos2013spectral}
L{\'a}szl{\'o} Erd{\H{o}}s, Antti Knowles, Horng-Tzer Yau, and Jun Yin.
\newblock Spectral statistics of {E}rd{\H{o}}s--{R}{\'e}nyi graphs {I}: Local
  semicircle law.
\newblock {\em Annals of Probability}, 41:2279–--2375, 2013.

\bibitem[ER60]{erdHos1960evolution}
Paul Erd{\H{o}}s and Alfr{\'e}d R{\'e}nyi.
\newblock On the evolution of random graphs.
\newblock {\em Publ. Math. Inst. Hung. Acad. Sci}, 5(1):17--60, 1960.

\bibitem[FS83]{frohlich1983absence}
J{\"u}rg Fr{\"o}hlich and Thomas Spencer.
\newblock Absence of diffusion in the anderson tight binding model for large
  disorder or low energy.
\newblock {\em Communications in Mathematical Physics}, 88(2):151--184, 1983.

\bibitem[GLO20]{greenbaum2020first}
Anne Greenbaum, Ren-cang Li, and Michael~L Overton.
\newblock First-order perturbation theory for eigenvalues and eigenvectors.
\newblock {\em SIAM review}, 62(2):463--482, 2020.

\bibitem[Gly87]{glynn1987upper}
Peter~W Glynn.
\newblock Upper bounds on poisson tail probabilities.
\newblock {\em Operations research letters}, 6(1):9--14, 1987.

\bibitem[GMP77]{goldsheid1977random}
Ilya Goldsheid, Stanislav Molchanov, and Leonid Pastur.
\newblock A random homogeneous schr{\"o}dinger operator has a pure point
  spectrum.
\newblock {\em Funct. Anal. Appl}, 11(1):1--10, 1977.

\bibitem[Gui21]{guionnet2021bernoulli}
Alice Guionnet.
\newblock Bernoulli random matrices.
\newblock {\em arXiv preprint arXiv:2112.05506}, 2021.

\bibitem[HK21]{he2021fluctuations}
Yukun He and Antti Knowles.
\newblock Fluctuations of extreme eigenvalues of sparse erd{\H{o}}s--r{\'e}nyi
  graphs.
\newblock {\em Probability Theory and Related Fields}, 180(3-4):985--1056,
  2021.

\bibitem[HLW06]{hoory2006expander}
Shlomo Hoory, Nathan Linial, and Avi Wigderson.
\newblock Expander graphs and their applications.
\newblock {\em Bulletin of the American Mathematical Society}, 43(4):439--561,
  2006.

\bibitem[HLY20]{huang2020transition}
Jiaoyang Huang, Benjamin Landon, and Horng-Tzer Yau.
\newblock Transition from {T}racy--{W}idom to {G}aussian fluctuations of
  extremal eigenvalues of sparse {E}rd{\H{o}}s--{R}{\'e}nyi graphs.
\newblock {\em Annals of Probability}, 2020.

\bibitem[JRL11]{janson2011random}
Svante Janson, Andrzej Rucinski, and Tomasz Luczak.
\newblock {\em Random graphs}.
\newblock John Wiley \& Sons, 2011.

\bibitem[Kes59]{kesten1959symmetric}
Harry Kesten.
\newblock Symmetric random walks on groups.
\newblock {\em Transactions of the American Mathematical Society},
  92(2):336--354, 1959.

\bibitem[KS97]{kottos1997quantum}
Tsampikos Kottos and Uzy Smilansky.
\newblock Quantum chaos on graphs.
\newblock {\em Physical review letters}, 79(24):4794, 1997.

\bibitem[KS03]{krivelevich2003largest}
Michael Krivelevich and Benny Sudakov.
\newblock The {Largest} {Eigenvalue} of {Sparse} {Random} {Graphs}.
\newblock {\em Combinatorics, Probability and Computing}, 12(1):61--72, January
  2003.
\newblock Publisher: Cambridge University Press.

\bibitem[KSV04]{khorunzhy2004eigenvalue}
Oleksiy Khorunzhy, Mariya Shcherbina, and Valentin Vengerovsky.
\newblock Eigenvalue distribution of large weighted random graphs.
\newblock {\em Journal of Mathematical Physics}, 45(4):1648--1672, 2004.

\bibitem[LS18]{lee2018local}
Ji~Oon Lee and Kevin Schnelli.
\newblock Local law and tracy-widom limit for sparse random matrices.
\newblock {\em Probability Theory and Related Fields}, 171(1):543--616, 2018.

\bibitem[Sos04]{soshnikov2004poisson}
Alexander Soshnikov.
\newblock Poisson statistics for the largest eigenvalues of {W}igner random
  matrices with heavy tails.
\newblock {\em Electron. Comm. Probab.}, 9:82--91, 2004.

\bibitem[Tal95]{talagrand1995concentration}
Michel Talagrand.
\newblock Concentration of measure and isoperimetric inequalities in product
  spaces.
\newblock {\em Publications Math{\'e}matiques de l'Institut des Hautes Etudes
  Scientifiques}, 81:73--205, 1995.

\bibitem[Tao22]{tao2022improved}
Terence Tao.
\newblock An {I}mprovement to {B}ennett's {I}nequality for the {P}oisson
  {D}istribution, 2022.
\newblock
  {\href{https://terrytao.wordpress.com/2022/12/13/an-improvement-to-bennetts-inequality-for-the-poisson-distribution/}{https://terrytao.wordpress.com/2022/12/13/an-improvement-to-bennetts-inequality-for-the-poisson-distribution/}
  Accessed 2023-06-01}.

\bibitem[Zak06]{zakharevich2006generalization}
Inna Zakharevich.
\newblock A generalization of {W}igner’s law.
\newblock {\em Communications in mathematical physics}, 268:403--414, 2006.

\end{thebibliography}
